\documentclass[11pt,a4paper]{article}
\newcommand{\todo}[1]{\textcolor{red}{TODO: #1}}
\newcommand{\xl}[1]{\textcolor{magenta}{XL: #1}}
\newcommand{\JH}[1]{\textcolor{blue}{JH: #1}}
\usepackage{epsf,epsfig,amsfonts,amsgen,amsmath,amstext,amsbsy,amsopn,amsthm,cases,listings,color
}
\usepackage{wasysym}
\usepackage{ebezier,eepic}
\usepackage{color}
\usepackage{multirow}
\usepackage{epstopdf}
\usepackage{graphicx}
\usepackage{pgf,tikz}
\usetikzlibrary{patterns}
\usepackage{mathrsfs}
\usepackage[marginal]{footmisc}
\usepackage{enumitem}
\usepackage[titletoc]{appendix}
\usepackage{booktabs}
\usepackage{url}
\usepackage{mathtools}
\usepackage{xcolor}
\usepackage{pgfplots}
\usepackage{authblk}
\usepackage{amssymb}
\usepackage{wasysym}
\usepackage{empheq}
\usepackage{dsfont}
\usepackage{float}
\usepackage{subfig}

\usepackage{caption}
\usepackage{subcaption}
\pgfplotsset{compat=1.18}
\usepackage{mathrsfs}
\usepackage{wasysym} 
\usetikzlibrary{arrows}
\usepackage{aligned-overset}
\usepackage{bm}
\usepackage[backref=page]{hyperref}
\usepackage{makecell}
\usepackage{multicol}

\newcommand{\Mod}[1]{\ (\mathrm{mod}\ #1)}

\allowdisplaybreaks[1]

\definecolor{uuuuuu}{rgb}{0.27,0.27,0.27}
\definecolor{sqsqsq}{rgb}{0.1255,0.1255,0.1255}

\newtheorem{definition}{Definition} [section]
\newtheorem{theorem}[definition]{Theorem}
\newtheorem{lemma}[definition]{Lemma}

\newtheorem{conjecture}[definition]{Conjecture}
\newtheorem{claim}[definition]{Claim}
\newtheorem{problem}[definition]{Problem}

\newtheorem{fact}[definition]{Fact}
\newenvironment{pf}{\noindent{\bf Proof}}

\newcommand{\uproduct}{\mathbin{\;{\rotatebox{90}{\textnormal{$\small\Bowtie$}}}}}

\newcommand{\RNum}[1]{\uppercase\expandafter{\romannumeral #1\relax}}

\newcommand{\rep}[2]{{#1}^{(#2)}}
\newcommand{\blow}[2]{#1(\!(#2)\!)}
\newcommand{\multiset}[1]{\{\hspace{-0.25em}\{\hspace{0.1em}#1\hspace{0.1em}\}\hspace{-0.25em}\}}

\def\qed{\hfill \rule{4pt}{7pt}}

\def\pf{\noindent {\it Proof }}

\begin{document}
\title{\bf\Large Toward a rainbow Corr\'{a}di--Hajnal Theorem \RNum{1}}
\author{
Jinghua Deng, 
Jianfeng Hou, 
Caiyun Hu,
Xizhi Liu, 
\vspace{-1.5em} 
}
\date{\today}
\maketitle
\begin{abstract}
    We study an anti-Ramsey extension of the classical Corr\'{a}di--Hajnal Theorem: how many colors are needed to color the complete graph on $n$ vertices in order to guarantee a rainbow copy of $t K_{3}$, that is, $t$ vertex-disjoint triangles. 
    We provide a conjecture for large $n$, consisting of five classes of different extremal constructions, corresponding to five subintervals of $\left[1,\, \tfrac{n}{3}\right]$ for the parameter $t$. 
    In this work, we establish this conjecture for the first interval, $t \in \left[1,\, \tfrac{2n-6}{9}\right]$. 
    In particular, this improves upon a recent result of Lu--Luo--Ma~[arXiv:2506.07115] which established the case $t \le \tfrac{n - 57}{15}$. 
\end{abstract}

\renewcommand\thefootnote{} 
\footnotetext{\textit{Keywords:} anti-Ramsey problem, the Corr\'{a}di--Hajnal Theorem, rainbow triangle-tiling, stability method \\[0.15em]
\textit{MSC2020:} 05C35, 05C65, 05D05}
\addtocounter{footnote}{-1} 
\renewcommand\thefootnote{\arabic{footnote}} 
\section{Introduction}\label{SEC:Introduction}
\subsection{Definitions and main result}
For a graph $G$, let $V(G)$ and $E(G)$ denote its vertex set and edge set, respectively.
The size of $V(G)$ is denoted by $v(G)$. 
We use $|G|$ to denote the number of edges in $G$. 
The complete graph on $n$ vertices is denoted by $K_{n}$. 
An edge-colored graph $G$ is is said to be \emph{rainbow} if all its edges have distinct colors. 

Given a graph $F$ and an integer $t \ge 1$, let $t F$ denote the graph consisting of $t$ vertex-disjoint copies of $F$.
Motivated by the classical \emph{anti-Ramsey problem} initiated by Erd\H{o}s--Simonovits--S\'{o}s~\cite{ESS75}, as well as results on rainbow matchings by Schiermeyer~\cite{Sch04} and \"{O}zkahya--Young~\cite{OY13}, the following problem was recently proposed in~\cite{DHLY25} and systematically investigated for small values of $t$. 
\begin{problem}\label{PROB:antiRmasey-tF}
    Given a graph $F$ and a positive integer $n$. 
    Determine, for each $t \in \left[1, n/v(F)\right]$, the minimum value $N$, denoted by $\mathrm{ar}(n,tF)$, such that every edge-coloring of the complete graph $K_{n}$ using exactly $N$ colors contains a rainbow copy of $tF$. 
\end{problem}

Problem~\ref{PROB:antiRmasey-tF} is closely related to the classical Tur\'{a}n problem. 
Given a family $\mathcal{F}$ of graphs, we say a graph $H$ is \emph{$\mathcal{F}$-free} if it does not contain any member of $\mathcal{F}$ as a subgraph. 
The \emph{Tur\'{a}n number} $\mathrm{ex}(n, \mathcal{F})$ of $\mathcal{F}$ is the maximum number of edges in an $\mathcal{F}$-free graph on $n$ vertices. 
Determine the value of $\mathrm{ex}(n, \mathcal{F})$ for various families has been a central topic in Combinatorics since the seminal work of Tur\'{a}n~\cite{TU41}. 
Given a graph $F$, let $F_-$ denote the family of graphs obtained from $F$ by removing exactly one edge. 
For the case $t = 1$, it was observed by Erd\H{o}s--Simonovits--S\'{o}s~\cite{ESS75} that 
\begin{align*}
    \mathrm{ex}(n, F_-)+2 
    \le \mathrm{ar}(n, F) 
    \le \mathrm{ex}(n, F_-)+o(n^2).  
\end{align*}
%
%
The exact value of $\mathrm{ar}(n, F)$ is known only for a few classes of graph including the complete graph~\cite{ESS75,MN02}, paths~\cite{SS84}, cycles~\cite{ESS75,A83,MN05}, and trees~\cite{JW04}. 
We refer the reader to the survey~\cite{FMO10} for more results on this topic.

For $t \ge 2$, there are only a few results on Problem~\ref{PROB:antiRmasey-tF}. 
The observation of Erd\H{o}s--Simonovits--S\'{o}s was extended to general $t$ in~{\cite[Fact~2.1]{DHLY25}} that, for every $t \in \left[1, \tfrac{n}{v(F)}\right]$,  
\begin{align*}
    \mathrm{ex}(n, tF) + 2 
    \le \mathrm{ar}(n, (t+1)F) 
    \le \mathrm{ex}(n, (t+1)F) + 1.  
\end{align*}
A theorem of Jiang--Pikhurko~\cite{JP09} on doubly edge-critical graphs determines $\mathrm{ar}(n,2F)$ for large $n$ when $F$ is an edge-critical graph, where edge-critical means there exists an edge whose removal will decrease the chromatic number. 
The value of $\mathrm{ar}(n, tF)$ for relative small $t$ was systematically studied in~\cite{DHLY25} and determined for several classes of graphs and hypergraphs. 

Determining $\mathrm{ar}(n,(t + 1)F)$, even asymptotically, for the entire interval $[0, n/v(F)]$ appears to be a challenging problem in general. 
Partly due to the difficulty of determining the asymptotic value of $\mathrm{ex}(n, tF)$, which is known for very limited examples (see e.g.~\cite{GH12,ABHP15,HHLZ25,CLSW25}). 
It appears that $K_2$ stands as the only example for which Problem~\ref{PROB:antiRmasey-tF} has been completely solved~\cite{OY13,Sch04,CLT09,FKSS09}. 
A natural next step toward Problem~\ref{PROB:antiRmasey-tF} is naturally the case $F = K_{3}$.  
In this work, we investigate Problem~\ref{PROB:antiRmasey-tF} for the case when $F = K_3$. 
Below we provide a conjecture\footnote[1]{While preparing the draft, we learned that Liu--Ning--Tian~\cite{LNT25} independently proposed a similar conjecture. We include a discussion of the differences between the two conjectures in Section~\ref{SEC:remark}.} for the value of $\mathrm{ar}(n,tK_3)$ for large $n$.

\begin{conjecture}\label{Conj:Our-conjecture}
    Suppose that $n$ is sufficiently large and $t$ lies in the interval $\left[0, n/3\right]$. Then
    \begin{align*}
        \mathrm{ar}(n,(t+2)K_3) = \Xi(n,t)+2,
    \end{align*}
    where
    \begin{align*}
        \Xi(n,t) \coloneqq 
        \begin{cases}
            \binom{t}{2}+t(n-t)+\lceil\frac{n-t}{2}\rceil\lfloor\frac{n-t}{2}\rfloor &\text{if}\quad t\in\left[0,\frac{2n-6}{9}\right],\\[2ex]
            \binom{2t+1}{2}+\lceil\frac{n}{2}\rceil\lfloor\frac{n}{2}\rfloor &\text{if}\quad t\in\left[\frac{2n-6}{9},\frac{n-\sqrt{2n-3}-3}{4}\right],\\[2ex]
            \binom{2t+2}{2}+(2t + 2)(n-2t-2) &\text{if}\quad t\in\left[\frac{n-\sqrt{2n-3}-3}{4},\frac{ 5n+\sqrt{ 3n^2-2n+4}-20}{22}\right],\\[2ex]
            \binom{6t-n+6}{2}+(n-3t-3)(3t + 3) &\text{if}\quad t\in\left[\frac{5n+\sqrt{ 3n^2-2n+4}-20}{22},\frac{ 2n-\sqrt{16n-7}-3}{6}\right],\\[2ex]
            \binom{3t + 5}{2}+(n-3t-6) &\text{if}\quad t\in\left[\frac{ 2n-\sqrt{16n-7}-3}{6},\frac{n}{3}\right].
        \end{cases}
    \end{align*}
\end{conjecture}

There are several partial results on Conjecture~\ref{Conj:Our-conjecture}. 
Yuan--Zhang~\cite{YZ19} determined the value of $\mathrm{ar}(n,tK_3)$ for the case when $t$ is a constant and $n$ is sufficiently large. 
Wu--Zhang--Li--Xiao~\cite{WZLX23} extended their result to $t \le (1+o(1))\sqrt{n/2}$ and removed the restriction that $n$ is large. 
They also provided some lower bound and upper bound for $\mathrm{ar}(n,tK_3)$ in general. 
A general result from~\cite{DHLY25} determines $\mathrm{ar}(n,tK_3)$ for $t$ in the region $[0, \delta n]$, where $\delta > 0$ is some small constant. 
Very recently, their result was improved by Lu--Luo--Ma~\cite{LLM25} to the interval $\left[0, \tfrac{n-57}{15}\right]$ and removed the restriction that $n$ is large. 
For our purpose in this work, the following result from~\cite{DHLY25} suffices.  

\begin{theorem}[\cite{DHLY25}]~\label{THM:Anti-ramsey-(t+2)K3-t-small}
    There exist constants $\delta>0$ and $N_{\ref{THM:Anti-ramsey-(t+2)K3-t-small}}=N_{\ref{THM:Anti-ramsey-(t+2)K3-t-small}}(\delta)$ such that for every $n\ge N_{\ref{THM:Anti-ramsey-(t+2)K3-t-small}}$ and $t\in [0,\delta n]$,
    \begin{align*}
        \mathrm{ar}(n,(t+2)K_3)
        = \binom{t}{2}+t(n-t)+ \left\lfloor\frac{(n-t)^2}{4} \right\rfloor+2.
    \end{align*}
\end{theorem}

The main result of work is a confirmation of Conjecture~\ref{Conj:Our-conjecture} for the first interval. 
In particular, it improves the previous result by Lu--Luo--Ma~\cite{LLM25} for large $n$. 
\begin{theorem}\label{THM:main-first-interval}
    There exists $N_{0}$ such that the following holds for all $n \ge N_0$. 
    For every positive integer $t\le\frac{1}{9}(2n-6)-2$, 
    \begin{align*}
        \mathrm{ar}(n,(t+2)K_3)
        = \binom{t}{2}+t(n-t)+\left\lceil \frac{n-t}{2}\right\rceil \left\lfloor\frac{n-t}{2}\right \rfloor+2.
    \end{align*}
\end{theorem}

We remark that around the critical point $\tfrac{2n-6}{9}$, there are two near-extremal constructions, $E_{1}(n,t)$ and $E_{2}(n,t)$ defined below. This makes the stability approach used in the present work somewhat more involved to extend Theorem~\ref{THM:main-first-interval} to the entire first interval in Conjecture~\ref{Conj:Our-conjecture}. Nevertheless, we will address this remaining part, together with the second interval, in a forthcoming work.

\subsection{Extremal constructions}\label{SUBSEC:Extremal-construction}
We discuss the extremal constructions for Conjecture~\ref{Conj:Our-conjecture} in this subsection. 

\begin{figure}[H]
    \centering
\tikzset{every picture/.style={line width=0.75pt}} 

\begin{tikzpicture}[x=0.75pt,y=0.75pt,yscale=-1,xscale=1,scale=0.8]

\draw  [fill={rgb, 255:red, 255; green, 255; blue, 255 } ,fill opacity=1 ] (484.16,168.98).. controls (478.97,180.54) and (465.38,185.71).. (453.81,180.51).. controls (442.25,175.32) and (437.08,161.74).. (442.27,150.17).. controls (447.47,138.6) and (461.05,133.44).. (472.62,138.63).. controls (484.19,143.82) and (489.35,157.41).. (484.16,168.98) -- cycle ;
\draw  [fill={rgb, 255:red, 0; green, 0; blue, 0 } ,fill opacity=1 ] (487.12,101.97).. controls (481.09,115.4) and (465.31,121.4).. (451.88,115.37).. controls (438.45,109.34) and (432.45,93.56).. (438.48,80.13).. controls (444.51,66.7) and (460.29,60.7).. (473.72,66.73).. controls (487.15,72.76) and (493.15,88.54).. (487.12,101.97) -- cycle ;
\draw    (457.98,136.97) -- (454.65,116.9) ;
\draw    (465.8,136.24) -- (462.87,117.56) ;
\draw    (457.98,136.97) -- (462.87,117.56) ;
\draw    (465.8,136.24) -- (470.44,116.41) ;
\draw    (475.26,139.54) -- (470.44,116.41) ;
\draw    (449.24,141.13) -- (454.65,116.9) ;
\draw    (449.24,141.13) -- (445.7,111.32) ;
\draw    (475.26,139.54) -- (479.99,111.35) ;
\draw  [fill={rgb, 255:red, 255; green, 255; blue, 255 } ,fill opacity=1 ] (116.04,150.68).. controls (111.42,160.98) and (99.32,165.58).. (89.01,160.96).. controls (78.71,156.33) and (74.11,144.23).. (78.74,133.93).. controls (83.36,123.63) and (95.46,119.03).. (105.76,123.65).. controls (116.07,128.28) and (120.67,140.38).. (116.04,150.68) -- cycle ;
\draw  [fill={rgb, 255:red, 0; green, 0; blue, 0 } ,fill opacity=1 ] (154.99,98.18).. controls (151.1,106.86) and (140.9,110.73).. (132.22,106.84).. controls (123.54,102.94) and (119.66,92.74).. (123.56,84.06).. controls (127.45,75.38) and (137.65,71.5).. (146.33,75.4).. controls (155.01,79.3) and (158.89,89.49).. (154.99,98.18) -- cycle ;
\draw    (168.03,123.67) -- (152.6,101.95) ;
\draw    (97.16,121.94) -- (125.74,100.52) ;
\draw    (97.16,121.94) -- (122.32,92.81) ;
\draw    (105.76,123.65) -- (125.74,100.52) ;
\draw    (105.76,123.65) -- (132.03,105.38) ;
\draw    (168.03,123.67) -- (146.89,105.93) ;
\draw    (162.46,128.23) -- (146.89,105.93) ;
\draw    (112.31,128.23) -- (130.6,105.95) ;
\draw  [fill={rgb, 255:red, 255; green, 255; blue, 255 } ,fill opacity=1 ] (196.04,150.68).. controls (191.42,160.98) and (179.32,165.58).. (169.01,160.96).. controls (158.71,156.33) and (154.11,144.23).. (158.74,133.93).. controls (163.36,123.63) and (175.46,119.03).. (185.76,123.65).. controls (196.07,128.28) and (200.67,140.38).. (196.04,150.68) -- cycle ;
\draw    (175.64,121.95) -- (152.6,101.95) ;
\draw    (175.64,121.95) -- (156.24,95.35) ;
\draw    (157.42,146) -- (117.67,144.75) ;
\draw    (157.42,146) -- (115.42,152.25) ;
\draw    (159.92,153.5) -- (115.42,152.25) ;
\draw    (157.74,137.93) -- (117.67,144.75) ;
\draw    (157.74,137.93) -- (117.17,136.5) ;
\draw  [fill={rgb, 255:red, 255; green, 255; blue, 255 } ,fill opacity=1 ] (344.82,135.89).. controls (338.47,150.04) and (321.85,156.36).. (307.71,150).. controls (293.57,143.65) and (287.25,127.04).. (293.6,112.9).. controls (299.95,98.75) and (316.56,92.44).. (330.71,98.79).. controls (344.85,105.14) and (351.17,121.75).. (344.82,135.89) -- cycle ;
\draw  [fill={rgb, 255:red, 0; green, 0; blue, 0 } ,fill opacity=1 ] (338.76,64.66).. controls (333.78,75.76) and (320.75,80.71).. (309.65,75.73).. controls (298.56,70.75) and (293.6,57.71).. (298.58,46.62).. controls (303.56,35.52) and (316.6,30.56).. (327.69,35.55).. controls (338.79,40.53) and (343.75,53.56).. (338.76,64.66) -- cycle ;
\draw    (312.75,97.58) -- (309.65,75.73) ;
\draw    (322.94,96.41) -- (317.87,76.4) ;
\draw    (312.7,97.02) -- (317.87,76.4) ;
\draw    (322.94,96.41) -- (326.06,76.42) ;
\draw    (333.34,100.09) -- (326.06,76.42) ;
\draw    (302.25,101.9) -- (309.65,75.73) ;
\draw    (308.75,176.92) -- (314.08,151.92) ;
\draw    (319.42,174.25) -- (314.08,151.92) ;
\draw  [fill={rgb, 255:red, 255; green, 255; blue, 255 } ,fill opacity=1 ] (338.54,202.68).. controls (333.92,212.98) and (321.82,217.58).. (311.51,212.96).. controls (301.21,208.33) and (296.61,196.23).. (301.24,185.93).. controls (305.86,175.63) and (317.96,171.03).. (328.26,175.65).. controls (338.57,180.28) and (343.17,192.38).. (338.54,202.68) -- cycle ;
\draw    (308.75,176.92) -- (301.08,145.58) ;
\draw    (328.26,175.65) -- (325.08,152.25) ;
\draw    (319.42,174.25) -- (325.08,152.25) ;
\draw    (328.26,175.65) -- (336.08,146.92) ;
\draw    (157.41,448.86) -- (163.41,409.86) ;
\draw  [fill={rgb, 255:red, 255; green, 255; blue, 255 } ,fill opacity=1 ] (194.68,400.18).. controls (189.42,411.88) and (175.68,417.1).. (163.98,411.85).. controls (152.29,406.6) and (147.06,392.86).. (152.31,381.16).. controls (157.57,369.46) and (171.31,364.23).. (183.01,369.49).. controls (194.7,374.74) and (199.93,388.48).. (194.68,400.18) -- cycle ;
\draw  [fill={rgb, 255:red, 0; green, 0; blue, 0 } ,fill opacity=1 ] (202.01,326.06).. controls (194.63,342.48) and (175.35,349.81).. (158.93,342.44).. controls (142.51,335.07) and (135.18,315.79).. (142.55,299.37).. controls (149.92,282.95) and (169.21,275.62).. (185.63,282.99).. controls (202.04,290.36) and (209.38,309.65).. (202.01,326.06) -- cycle ;
\draw    (181.33,438.24) -- (174.06,413.7) ;
\draw    (181.33,438.24) -- (185.15,410.46) ;
\draw    (168.84,438.28) -- (163.04,411.48) ;
\draw    (168.84,438.28) -- (174.06,413.7) ;
\draw    (162.75,369.75) -- (168.08,344.75) ;
\draw    (173.42,367.08) -- (168.08,344.75) ;
\draw    (162.75,369.75) -- (155.08,338.42) ;
\draw    (182.26,368.49) -- (179.08,345.08) ;
\draw    (173.42,367.08) -- (179.08,345.08) ;
\draw    (182.26,368.49) -- (190.08,339.75) ;
\draw  [fill={rgb, 255:red, 255; green, 255; blue, 255 } ,fill opacity=1 ] (196.68,470.18).. controls (191.42,481.88) and (177.68,487.1).. (165.98,481.85).. controls (154.29,476.6) and (149.06,462.86).. (154.31,451.16).. controls (159.57,439.46) and (173.31,434.23).. (185.01,439.49).. controls (196.7,444.74) and (201.93,458.48).. (196.68,470.18) -- cycle ;
\draw    (191.87,443.7) -- (185.15,410.46) ;
\draw  [fill={rgb, 255:red, 0; green, 0; blue, 0 } ,fill opacity=1 ] (369.78,396.8).. controls (363.61,410.55) and (347.45,416.69).. (333.7,410.52).. controls (319.95,404.35) and (313.81,388.19).. (319.98,374.44).. controls (326.16,360.69) and (342.31,354.55).. (356.06,360.72).. controls (369.81,366.89) and (375.96,383.05).. (369.78,396.8) -- cycle ;
\draw    (428.5,318.5) -- (428.88,344.88) ;
\draw  [fill={rgb, 255:red, 0; green, 0; blue, 0 } ,fill opacity=1 ] (427.5,317.39).. controls (428.11,316.83) and (429.06,316.88).. (429.61,317.5).. controls (430.17,318.11) and (430.12,319.06).. (429.5,319.61).. controls (428.89,320.17) and (427.94,320.12).. (427.39,319.5).. controls (426.83,318.89) and (426.88,317.94).. (427.5,317.39) -- cycle ;
\draw  [fill={rgb, 255:red, 0; green, 0; blue, 0 } ,fill opacity=1 ] (427.88,343.77).. controls (428.49,343.21) and (429.44,343.26).. (430,343.88).. controls (430.55,344.49) and (430.5,345.44).. (429.89,346).. controls (429.27,346.55) and (428.32,346.5).. (427.77,345.89).. controls (427.21,345.27) and (427.26,344.32).. (427.88,343.77) -- cycle ;

\draw   (428.29,306.99).. controls (442.65,306.99) and (454.29,344.15).. (454.29,389.99).. controls (454.29,435.83) and (442.65,472.99).. (428.29,472.99).. controls (413.93,472.99) and (402.29,435.83).. (402.29,389.99).. controls (402.29,344.15) and (413.93,306.99).. (428.29,306.99) -- cycle ;
\draw    (428.88,344.88) -- (429.26,371.26) ;
\draw  [fill={rgb, 255:red, 0; green, 0; blue, 0 } ,fill opacity=1 ] (428.26,370.15).. controls (428.88,369.59) and (429.82,369.64).. (430.38,370.26).. controls (430.93,370.88) and (430.88,371.82).. (430.27,372.38).. controls (429.65,372.93) and (428.7,372.88).. (428.15,372.27).. controls (427.6,371.65) and (427.64,370.7).. (428.26,370.15) -- cycle ;
\draw    (429.5,434.21) -- (429.88,460.59) ;
\draw  [fill={rgb, 255:red, 0; green, 0; blue, 0 } ,fill opacity=1 ] (428.5,433.09).. controls (429.11,432.54) and (430.06,432.59).. (430.61,433.2).. controls (431.17,433.82) and (431.12,434.77).. (430.5,435.32).. controls (429.89,435.88) and (428.94,435.83).. (428.39,435.21).. controls (427.83,434.59) and (427.88,433.65).. (428.5,433.09) -- cycle ;
\draw  [fill={rgb, 255:red, 0; green, 0; blue, 0 } ,fill opacity=1 ] (428.88,459.47).. controls (429.49,458.92) and (430.44,458.97).. (431,459.58).. controls (431.55,460.2) and (431.5,461.15).. (430.89,461.7).. controls (430.27,462.26) and (429.32,462.21).. (428.77,461.59).. controls (428.21,460.98) and (428.26,460.03).. (428.88,459.47) -- cycle ;

\draw    (429.26,371.26) -- (429.3,391.37) ;
\draw    (429.46,414.1) -- (429.5,434.21) ;

\draw (112,224.4) node [anchor=north west][inner sep=0.75pt]    {$E_{1}( n,t)$};
\draw (145.5,491.4) node [anchor=north west][inner sep=0.75pt]    {$E_{4}( n,t)$};
\draw (441.9,224.4) node [anchor=north west][inner sep=0.75pt]    {$E_{3}( n,t)$};
\draw (294.5,224.4) node [anchor=north west][inner sep=0.75pt]    {$E_{2}( n,t)$};
\draw (432.1,392.25) node [anchor=north west][inner sep=0.75pt]  [rotate=-90]  {$\dotsc $};
\draw (109,67.8) node [anchor=north west][inner sep=0.75pt]    {$X$};
\draw (63,157.4) node [anchor=north west][inner sep=0.75pt]    {$Y_{1}$};
\draw (195,157.4) node [anchor=north west][inner sep=0.75pt]    {$Y_{2}$};
\draw (343,35.4) node [anchor=north west][inner sep=0.75pt]    {$X$};
\draw (349,105.8) node [anchor=north west][inner sep=0.75pt]    {$Y_{1}$};
\draw (343,178) node [anchor=north west][inner sep=0.75pt]    {$Y_{2}$};
\draw (493,72.2) node [anchor=north west][inner sep=0.75pt]    {$X$};
\draw (489,162.2) node [anchor=north west][inner sep=0.75pt]    {$Y$};
\draw (207,291.4) node [anchor=north west][inner sep=0.75pt]    {$X$};
\draw (199,374.4) node [anchor=north west][inner sep=0.75pt]    {$Y_{1}$};
\draw (201,449.4) node [anchor=north west][inner sep=0.75pt]    {$Y_{2}$};
\draw (350.5,491.4) node [anchor=north west][inner sep=0.75pt]    {$E_{5}( n,t)$};
\draw (299,381.4) node [anchor=north west][inner sep=0.75pt]    {$X$};
\draw (465,387.4) node [anchor=north west][inner sep=0.75pt]    {$Y$};
\end{tikzpicture}
\caption{Constructions $E_1(n,t),\dots,E_5(n,t)$.}
\label{FIG:Structure}
\end{figure}
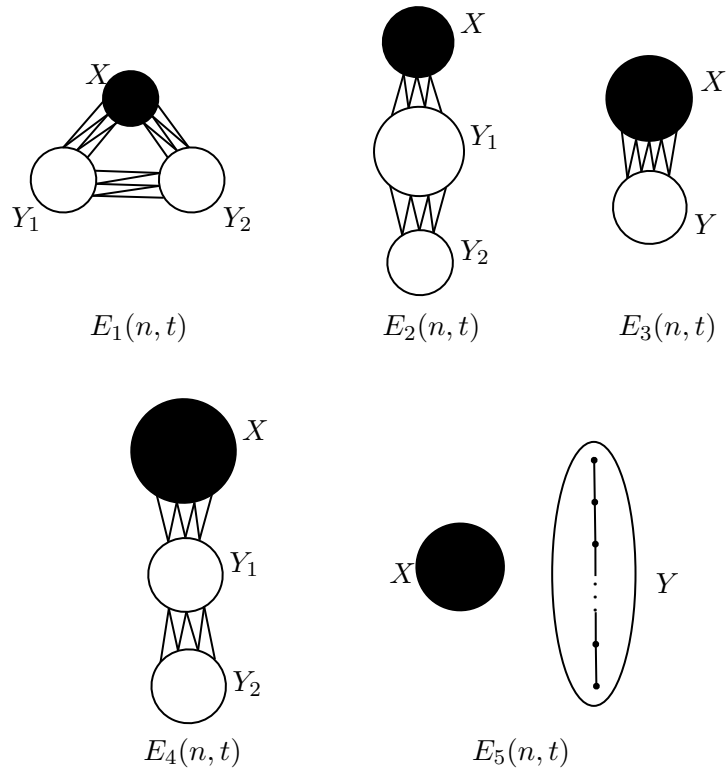

The function $\Xi(n,t)$ arises from the number of edges of the five classes of constructions $E_1(n,t), \cdots, E_5(n,t)$ (see Figure~\ref{FIG:Structure}) defined below. Specifically, 
\begin{align*}
    \Xi(n,t) 
    = \max\left\{|E_i(n,t)| \colon i \in [5] \right\} + 2. 
\end{align*}

Given $k$ pairwise disjoint sets $V_1, \cdots, V_k$, we use $K[V_1, \cdots, V_k]$ to denote the \emph{complete $k$-partite graph} with parts $V_1, \cdots, V_k$. 
Let $n$ and $t$ be integers satisfying $n \ge 3t \ge 0$. 
\begin{itemize}
    \item The vertex set of $E_1(n,t)$ consists of three parts $X,Y_1,Y_2$, with sizes given by
    \begin{align*}
        |X|=t,\quad|Y_1|=\left\lfloor\frac{n-t}{2}\right\rfloor,\quad \text{and}\quad|Y_2|=\left\lceil\frac{n-t}{2}\right\rceil.
    \end{align*}
    The edge set of $E_1(n,t)$ is define as
    \begin{align*}
        E_1(n,t)=\binom{X}{2}\cup K[X,Y_1,Y_2].
    \end{align*}

    \item The vertex set of $E_2(n,t)$ consists of three parts $X,Y_1,Y_2$, with sizes given by
    \begin{align*}
        |X|=2t+1,\quad|Y_1|=\left\lfloor\frac{n}{2}\right\rfloor,\quad \text{and}\quad|Y_2|=\left\lceil\frac{n}{2}\right\rceil-2t-1.
    \end{align*}
    The edge set of $E_2(n,t)$ is define as
    \begin{align*}
        E_2(n,t)=\binom{X}{2}\cup K[X,Y_1]\cup K[Y_1,Y_2].
    \end{align*}

    \item The vertex set of $E_3(n,t)$ consists of three parts $X,Y$, with sizes given by
    \begin{align*}
        |X|=2t+2\quad\text{and}\quad|Y|=n-2t-2.
    \end{align*}
    The edge set of $E_3(n,t)$ is define as
    \begin{align*}
        E_3(n,t)=\binom{X}{2}\cup K[X,Y].
    \end{align*}

    \item The vertex set of $E_4(n,t)$ consists of three parts $X,Y_1,Y_2$, with sizes given by
    \begin{align*}
        |X|=6t-n+6,\quad|Y_1|=n-3t-3,\quad\text{and}\quad|Y_2|=n-3t-3.
    \end{align*}
    The edge set of $E_4(n,t)$ is define as
    \begin{align*}
        E_4(n,t)=\binom{X}{2}\cup K[X,Y_1]\cup K[Y_1,Y_2].
    \end{align*}

    \item The vertex set of $E_5(n,t)$ consists of two parts $X,Y$, with sizes given by
    \begin{align*}
        |X|=3t+5\quad\text{and}\quad|Y|=n-3t-5.
    \end{align*}
    The edge set of $E_5(n,t)$ is define as
    \begin{align*}
        E_5(n,t)=\binom{X}{2}\cup P_{n-3t-6}. 
    \end{align*}
\end{itemize}

Allen--B\"{o}ttcher--Hladk\'{y}--Piguet~\cite{ABHP15}  established the exact value of $\mathrm{ex}(n,(t+1)K_3)$ for sufficiently large $n$ and all $0 \le t \le n/3$. 

\begin{theorem}[\cite{ABHP15}]\label{THM:Turan-number-(t+1)K_3}
    There exists $n_0$ such that for each $n > n_0$ and each $t\in [0,n/3]$, we have the following.
    \begin{align*}
        \mathrm{ex}(n,(t+1)K_3)=
    \begin{cases}
        \binom{t}{2}+t(n-t)+\lceil \frac{n-t}{2}\rceil \lfloor\frac{n-t}{2} \rfloor,    & 1\le t \le \frac{2n-6}{9}, \\[0.4em]
        \binom{2t+1}{2}+\lceil \frac{n}{2}\rceil \lfloor\frac{n}{2} \rfloor,    & \frac{2n-6}{9}\le t \le \frac{n-1}{4}, \\[0.4em]
        \binom{2t+1}{2}+(2t+1)(n-2t-1),   & \frac{n-1}{4}\le t \le \frac{5n-12+\sqrt{3n^2-10n+12} }{22}, \\[0.4em]
        \binom{6t-n+4}{2}+(3t+2)(n-3t-2),  & \frac{5n-12+\sqrt{3n^2-10n+12} }{22}\le t\le \frac{n}{3}.
    \end{cases}
    \end{align*}
\end{theorem}

\begin{figure}[h]
\centering

\tikzset{every picture/.style={line width=0.75pt}} 

\begin{tikzpicture}[x=0.7pt,y=0.7pt,yscale=-0.9,xscale=0.9, scale = 0.8]

\draw    (165.05,173.7) -- (138.5,129.89) ;
\draw    (165.05,173.7) -- (132.5,135.16) ;
\draw    (157.05,176.07) -- (132.5,135.16) ;
\draw  [fill={rgb, 255:red, 255; green, 255; blue, 255 }  ,fill opacity=1 ] (135.94,209.5) .. controls (135.94,189.34) and (152.28,173) .. (172.44,173) .. controls (192.6,173) and (208.94,189.34) .. (208.94,209.5) .. controls (208.94,229.66) and (192.6,246) .. (172.44,246) .. controls (152.28,246) and (135.94,229.66) .. (135.94,209.5) -- cycle ;
\draw    (174.86,172.98) -- (138.5,129.89) ;
\draw    (174.86,172.98) -- (142.86,121.7) ;
\draw    (135.55,208.28) -- (96,207.75) ;
\draw    (135.55,208.28) -- (95.77,216.05) ;
\draw    (136.88,216.94) -- (95.77,216.05) ;
\draw    (137.32,200.05) -- (96,207.75) ;
\draw    (137.32,200.05) -- (94.5,199.5) ;
\draw  [fill={rgb, 255:red, 255; green, 255; blue, 255 }  ,fill opacity=1 ] (324.09,157.15) .. controls (315.91,175.57) and (294.4,183.75) .. (276.05,175.41) .. controls (257.69,167.08) and (249.45,145.38) .. (257.63,126.96) .. controls (265.81,108.53) and (287.32,100.35) .. (305.67,108.69) .. controls (324.02,117.03) and (332.27,138.72) .. (324.09,157.15) -- cycle ;
\draw  [fill={rgb, 255:red, 0; green, 0; blue, 0 }  ,fill opacity=1 ] (263.36,48.05) .. controls (263.23,32.86) and (275.44,20.55) .. (290.62,20.55) .. controls (305.81,20.55) and (318.23,32.86) .. (318.36,48.05) .. controls (318.49,63.24) and (306.28,75.55) .. (291.09,75.55) .. controls (275.91,75.55) and (263.49,63.24) .. (263.36,48.05) -- cycle ;
\draw    (285.8,105.87) -- (281.48,73.48) ;
\draw    (297,106.38) -- (291.09,74.85) ;
\draw    (285.8,105.87) -- (291.09,74.85) ;
\draw    (307.2,109.37) -- (300.28,73.87) ;
\draw    (307.2,109.37) -- (308.59,69.62) ;
\draw    (297,106.38) -- (300.28,73.87) ;
\draw    (283.08,217.87) -- (287.25,177.89) ;
\draw    (283.08,217.87) -- (278.42,175.86) ;
\draw  [fill={rgb, 255:red, 255; green, 255; blue, 255 }  ,fill opacity=1 ] (272.86,233.05) .. controls (272.77,223.11) and (280.76,215.05) .. (290.71,215.05) .. controls (300.65,215.05) and (308.77,223.11) .. (308.86,233.05) .. controls (308.94,242.99) and (300.95,251.05) .. (291.01,251.05) .. controls (281.07,251.05) and (272.94,242.99) .. (272.86,233.05) -- cycle ;
\draw    (297.44,216.82) -- (296.58,178.01) ;
\draw    (290.03,216.09) -- (287.25,177.89) ;
\draw    (276.92,221.78) -- (278.42,175.86) ;
\draw    (290.03,216.09) -- (296.58,178.01) ;
\draw  [fill={rgb, 255:red, 0; green, 0; blue, 0 }  ,fill opacity=1 ] (561.33,54.85) .. controls (561.33,34.69) and (577.67,18.35) .. (597.83,18.35) .. controls (617.99,18.35) and (634.33,34.69) .. (634.33,54.85) .. controls (634.33,75.01) and (617.99,91.35) .. (597.83,91.35) .. controls (577.67,91.35) and (561.33,75.01) .. (561.33,54.85) -- cycle ;
\draw    (581.19,88.02) -- (568.98,122.69) ;
\draw    (574.15,83.26) -- (557.72,119.73) ;
\draw    (622.33,81.58) -- (640.18,109.39) ;
\draw    (629.06,73.6) -- (640.18,108.39) ;
\draw    (589.25,90.55) -- (568.98,122.69) ;
\draw    (581.19,88.02) -- (563.42,119.41) ;
\draw    (568.04,76.27) -- (557.72,119.73) ;
\draw    (574.15,83.26) -- (563.42,119.41) ;
\draw    (622.33,81.58) -- (633.08,112.43) ;
\draw    (614.33,85.41) -- (627.68,118.4) ;
\draw    (629.06,73.6) -- (649.73,111.77) ;
\draw    (614.33,85.41) -- (633.08,112.43) ;
\draw    (624.14,125.42) -- (573.39,130.88) ;
\draw    (564.21,161.54) -- (560.87,203.19) ;
\draw    (564.21,161.54) -- (566.35,204.78) ;
\draw    (559.18,161.85) -- (560.87,203.19) ;
\draw    (569.55,158.52) -- (566.35,204.78) ;
\draw    (569.55,158.52) -- (571.61,209.76) ;
\draw    (633.24,65.2) -- (649.73,111.77) ;
\draw    (641.47,174.63) -- (638.81,194.74) ;
\draw    (635.74,173.64) -- (638.81,194.74) ;
\draw    (647.27,173.51) -- (644.2,194.87) ;
\draw    (641.47,174.63) -- (644.2,194.87) ;
\draw    (622.9,133.82) -- (573.39,130.88) ;
\draw    (575.26,135.78) -- (622.94,133.82) ;
\draw   (561.24,119.29) .. controls (568.98,119.29) and (575.26,128.91) .. (575.26,140.78) .. controls (575.26,152.65) and (568.98,162.27) .. (561.24,162.27) .. controls (553.49,162.27) and (547.22,152.65) .. (547.22,140.78) .. controls (547.22,128.91) and (553.49,119.29) .. (561.24,119.29) -- cycle ;
\draw   (561.28,203.02) .. controls (569.02,203.02) and (575.3,212.64) .. (575.3,224.51) .. controls (575.3,236.38) and (569.02,246) .. (561.28,246) .. controls (553.54,246) and (547.26,236.38) .. (547.26,224.51) .. controls (547.26,212.64) and (553.54,203.02) .. (561.28,203.02) -- cycle ;
\draw   (641.59,108.69) .. controls (652.56,108.69) and (661.45,123.61) .. (661.45,142.01) .. controls (661.45,160.41) and (652.56,175.33) .. (641.59,175.33) .. controls (630.62,175.33) and (621.73,160.41) .. (621.73,142.01) .. controls (621.73,123.61) and (630.62,108.69) .. (641.59,108.69) -- cycle ;
\draw    (597.83,91.52) -- (572.51,127.37) ;
\draw    (589.25,90.35) -- (572.51,127.37) ;
\draw    (606.05,89.23) -- (627.68,118.4) ;
\draw   (641.64,193.69) .. controls (652.61,193.69) and (661.5,208.61) .. (661.5,227.01) .. controls (661.5,245.41) and (652.61,260.33) .. (641.64,260.33) .. controls (630.67,260.33) and (621.78,245.41) .. (621.78,227.01) .. controls (621.78,208.61) and (630.67,193.69) .. (641.64,193.69) -- cycle ;
\draw    (574.86,141.75) -- (622.52,150.2) ;
\draw    (574.86,141.75) -- (621.73,142.01) ;
\draw    (575.26,135.78) -- (621.73,142.01) ;
\draw    (574.22,147.08) -- (624.52,158.08) ;
\draw    (553.78,158.66) -- (555.68,204.78) ;
\draw    (574.22,147.08) -- (622.55,150.2) ;
\draw    (559.1,161.98) -- (555.68,204.78) ;
\draw    (553.78,158.66) -- (550.93,210.02) ;
\draw    (623.76,211.26) -- (574.7,221.06) ;
\draw    (621.77,220.89) -- (574.7,221.06) ;
\draw    (575.32,225.94) -- (621.77,220.89) ;
\draw    (575.31,231.92) -- (622.42,237.26) ;
\draw    (575.31,231.92) -- (622.09,228.92) ;
\draw    (575.32,225.94) -- (622.09,228.92) ;
\draw    (573.32,237.33) -- (624.93,244.99) ;
\draw    (573.32,237.33) -- (622.42,237.26) ;
\draw    (630.55,169.94) -- (633.74,196.65) ;
\draw    (652.96,169.05) -- (649.89,196.93) ;
\draw    (647.27,173.51) -- (649.89,196.93) ;
\draw    (635.74,173.64) -- (633.74,196.65) ;
\draw    (630.55,169.94) -- (628.96,201.1) ;
\draw    (652.96,169.05) -- (654.29,201.1) ;
\draw    (275.92,108.59) -- (274.14,70.08) ;
\draw    (275.92,108.59) -- (281.48,73.48) ;
\draw    (267.35,113.91) -- (274.14,70.08) ;
\draw    (315.35,114.91) -- (308.59,69.62) ;
\draw    (303.61,220.79) -- (305.28,175.68) ;
\draw    (297.44,216.82) -- (305.28,175.68) ;
\draw    (276.92,221.78) -- (270.25,171.95) ;
\draw    (303.61,220.79) -- (312.92,171.49) ;
\draw  [fill={rgb, 255:red, 0; green, 0; blue, 0 }  ,fill opacity=1 ] (401.29,104.08) .. controls (409.47,85.66) and (430.98,77.48) .. (449.34,85.82) .. controls (467.69,94.16) and (475.94,115.85) .. (467.76,134.28) .. controls (459.58,152.7) and (438.07,160.88) .. (419.71,152.54) .. controls (401.36,144.2) and (393.11,122.51) .. (401.29,104.08) -- cycle ;
\draw  [fill={rgb, 255:red, 255; green, 255; blue, 255 }  ,fill opacity=1 ] (462.03,220.18) .. controls (462.16,235.37) and (449.95,247.68) .. (434.76,247.68) .. controls (419.57,247.68) and (407.16,235.37) .. (407.03,220.18) .. controls (406.9,204.99) and (419.1,192.68) .. (434.29,192.68) .. controls (449.48,192.68) and (461.9,204.99) .. (462.03,220.18) -- cycle ;
\draw    (439.58,155.88) -- (443.91,194.3) ;
\draw    (428.38,155.28) -- (434.29,192.68) ;
\draw    (439.58,155.88) -- (434.29,192.68) ;
\draw    (418.18,151.73) -- (425.11,193.84) ;
\draw    (418.18,151.73) -- (416.79,198.88) ;
\draw    (428.38,155.28) -- (425.11,193.84) ;
\draw    (449.46,152.66) -- (451.24,198.34) ;
\draw    (449.46,152.66) -- (443.91,194.3) ;
\draw    (457.94,147.45) -- (451.24,198.34) ;
\draw    (410.08,146.19) -- (416.79,198.88) ;
\draw  [fill={rgb, 255:red, 255; green, 255; blue, 255 }  ,fill opacity=1 ] (23.28,209.5) .. controls (23.28,189.34) and (39.62,173) .. (59.78,173) .. controls (79.93,173) and (96.28,189.34) .. (96.28,209.5) .. controls (96.28,229.66) and (79.93,246) .. (59.78,246) .. controls (39.62,246) and (23.28,229.66) .. (23.28,209.5) -- cycle ;
\draw  [fill={rgb, 255:red, 0; green, 0; blue, 0 }  ,fill opacity=1 ] (144.68,111.51) .. controls (144.82,127.25) and (132.17,140.01) .. (116.43,140.01) .. controls (100.69,140.01) and (87.82,127.25) .. (87.68,111.51) .. controls (87.55,95.77) and (100.2,83.01) .. (115.94,83.01) .. controls (131.68,83.01) and (144.55,95.77) .. (144.68,111.51) -- cycle ;
\draw    (140.77,191.38) -- (94.5,199.5) ;
\draw    (136.88,216.94) -- (93.1,224.5) ;
\draw    (139.55,225.61) -- (93.1,224.5) ;
\draw    (139.55,225.61) -- (88.75,232.08) ;
\draw    (145.34,233.9) -- (88.75,232.08) ;
\draw    (157.05,176.07) -- (125.59,138.25) ;
\draw    (186.14,175.52) -- (142.86,121.7) ;
\draw    (149.95,180.61) -- (125.59,138.25) ;
\draw    (67.69,173.7) -- (94.19,129.89) ;
\draw    (67.69,173.7) -- (100.18,135.16) ;
\draw    (75.68,176.07) -- (100.18,135.16) ;
\draw    (57.89,172.98) -- (94.19,129.89) ;
\draw    (57.89,172.98) -- (89.84,121.7) ;
\draw    (75.68,176.07) -- (107.08,138.25) ;
\draw    (46.64,175.52) -- (89.84,121.7) ;
\draw    (82.76,180.61) -- (107.08,138.25) ;

\draw (143.67,83.8) node [anchor=north west][inner sep=0.75pt]    {$X$};
\draw (16.67,158.4) node [anchor=north west][inner sep=0.75pt]    {$Y_{1}$};
\draw (192.67,158.4) node [anchor=north west][inner sep=0.75pt]    {$Y_{2}$};
\draw (320.67,28.73) node [anchor=north west][inner sep=0.75pt]    {$X$};
\draw (330.67,126.47) node [anchor=north west][inner sep=0.75pt]    {$Y_{1}$};
\draw (310.33,225) node [anchor=north west][inner sep=0.75pt]    {$Y_{2}$};
\draw (474,99.53) node [anchor=north west][inner sep=0.75pt]    {$X$};
\draw (463,213.53) node [anchor=north west][inner sep=0.75pt]    {$Y_{1}$};
\draw (574,263.4) node [anchor=north west][inner sep=0.75pt]    {$\Gamma _{4}( n,k)$};
\draw (522,127.82) node [anchor=north west][inner sep=0.75pt]    {$Y_{1}$};
\draw (662,127.19) node [anchor=north west][inner sep=0.75pt]    {$Y_{2}$};
\draw (522,216.19) node [anchor=north west][inner sep=0.75pt]    {$Y_{3}$};
\draw (662,229.74) node [anchor=north west][inner sep=0.75pt]    {$Y_{4}$};
\draw (636,28.2) node [anchor=north west][inner sep=0.75pt]    {$X$};
\draw (267.73,263.4) node [anchor=north west][inner sep=0.75pt]    {$\Gamma _{2}( n,k)$};
\draw (88.73,263.4) node [anchor=north west][inner sep=0.75pt]    {$\Gamma _{1}( n,k)$};
\draw (410.06,263.4) node [anchor=north west][inner sep=0.75pt]    {$\Gamma _{3}( n,k)$};

\end{tikzpicture}
\caption{The extremal graphs.}\label{FIG:extramal-graph}
\end{figure}
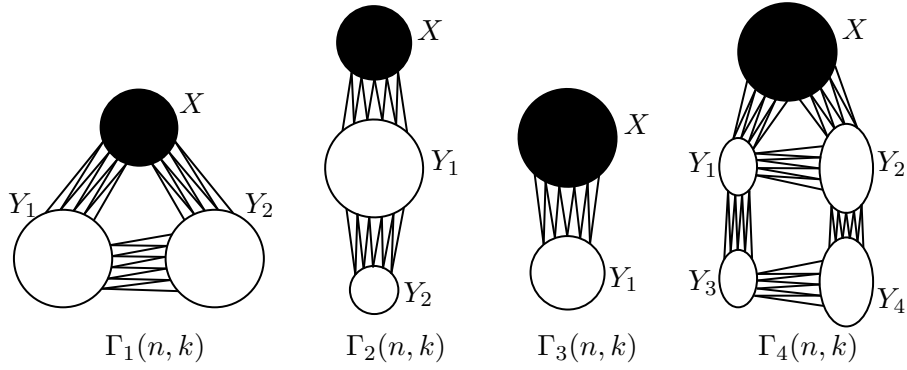

\textbf{Remark: }Figure~\ref{FIG:extramal-graph} illustrates the extremal graphs of Theorem~\ref{THM:Turan-number-(t+1)K_3}, where $ \Gamma_i(n,t) \cong E_i(n,t) $ for $ i \in \{1,2\} $. $ \Gamma_3(n,t) = \binom{X}{2} \cup K[X,Y_1] $ with $ |X| = 2t+1 $ and $ |Y_1| = n-2t-1 $. $ \Gamma_4(n,t) $ consists of disjoint vertex sets $ X, Y_1, Y_2, Y_3, Y_4 $ satisfying $ |Y_1| = |Y_3| $, $ |Y_2| = |Y_4| $, $ |Y_1| + |Y_2| = n-3t-2 $, and $ |X| = 6t - n + 4 $. All edges within $ X $, between $ X $ and $ Y_1 \cup Y_2 $, between $ Y_1 \cup Y_4 $, and between $ Y_2 \cup Y_3 $ are included in $\Gamma_4(n,t)$.

Define $e_i(n,t)\coloneqq |E_i(n,t)|$ for $i \in [5]$. To demonstrate the tightness of the bound $\mathrm{ar}(n,(t+2)K_3)$ established in Conjecture~\ref{Conj:Our-conjecture} and Theorem \ref{THM:main-first-interval}, we construct the following surjective example. Let $G$ be a copy of $E_i(n,t)$ in $K_n$ and $\chi\colon K_n\to\mathbb [e_i(n,t)+1]$ be an edge-coloring satisfying
\begin{itemize}
    \item  $\chi(G) = [e_i(n,t)]$, that is, $G$ is rainbow, and 
    \item  $\chi(e)= e_i(n,t) + 1$ for every $e \in K_n \setminus G$.
\end{itemize}
When $i\in[4]$, it is easy to see that $K_n$ does not contain a rainbow copy of $(t + 2)K_3$ under $\chi$. 

For $ i = 5 $, let $ G $ be a copy of $ K_{3t+5} $ in $ K_n $, with $ V(K_n) \setminus V(G) = \{v_1, v_2, \dots, v_{n-3t-5}\} $. Consider an edge-coloring $ \chi: K_n \to [e_5(n,t) + 1] $ satisfying the following conditions:
\begin{itemize}
  \item$ \chi(G) =\left[ \binom{3t+5}{2} \right]$, meaning $ G $ is a rainbow subgraph,
  \item$ \chi(xv_i) = \binom{3t+5}{2} + i $ for all $ xv_i \in V(G) \times \{v_1, v_2, \ldots, v_{n-3t-5}\} $, and
 \item $ \chi(v_iv_j) = \binom{3t+5}{2} + \min\{i,j\} $ for all $ v_iv_j \in E(H-G)$.
\end{itemize}
It is easy to see that $K_n$ does not contain a rainbow copy of $(t + 2)K_3$ under $\chi$.

In the remainder of this paper, we say a graph $H \subseteq K_n$ is a \emph{representative graph} of a surjective edge coloring $\chi \colon K_n \to [N]$  if  $H$ is rainbow under $\chi$ and $|H | = N$.  
\subsection{Setup}
In this subsection, we present the notations and assumptions that will be used throughout the remainder of this paper, unless otherwise stated.

For $t\ge 0$, let $H$ be an $n$-vertex $(t+1)K_3$-free graph. Given a graph $F$, a collection of vertex-disjoint copies of $F$ in $H$ is called an \emph{$F$-tiling} in $H$. A 3-tuple of families $(\mathcal{T},\mathcal{M},\mathcal{I})$ is a \emph{tiling triple} of $H$ if 
\begin{itemize}
    \item $\mathcal T$ is a $K_3$-tiling in $H$,
    \item $\mathcal M$ is a $K_2$-tiling (matching) in $H$,
    \item $\mathcal I$ is a $K_1$-tiling (independent set) in $H$, and
    \item $V(\mathcal T)\cup V(\mathcal M)\cup V(\mathcal I)=V(H)$ is a partition of $V(H)$.
\end{itemize}

As proposed in~\cite{ABHP15}, we additionally require that  $(\mathcal{T},\mathcal{M},\mathcal{I})$ is lexicographically maximal among all tiling triples of $H$ for the subsequent analysis. Such a tiling triple will be referred to as a \emph{maximal tiling triple} of $H$. 

\begin{figure}[h]
    \centering
\tikzset{every picture/.style={line width=0.5pt}} 

\begin{tikzpicture}[x=0.8pt,y=0.8pt,yscale=-1,xscale=1,scale=0.85]

\draw [color={rgb, 255:red, 0; green, 0; blue, 0 }  ,draw opacity=1 ][line width=1.5]  [dash pattern={on 1.69pt off 2.76pt}]  (565.36,60.51) -- (563.07,137.71) ;
\draw [color={rgb, 255:red, 0; green, 0; blue, 0 }  ,draw opacity=1 ][line width=1.5]  [dash pattern={on 1.69pt off 2.76pt}]  (565.36,60.51) -- (522.12,75.66) ;
\draw [color={rgb, 255:red, 0; green, 0; blue, 0 }  ,draw opacity=1 ][line width=1.5]  [dash pattern={on 1.69pt off 2.76pt}]  (445.07,81.69) -- (428.71,138.08) ;
\draw [color={rgb, 255:red, 0; green, 0; blue, 0 }  ,draw opacity=1 ][line width=1.5]  [dash pattern={on 1.69pt off 2.76pt}]  (340.17,61.09) -- (387.65,76.05) ;
\draw [color={rgb, 255:red, 0; green, 0; blue, 0 }  ,draw opacity=1 ][line width=1.5]  [dash pattern={on 1.69pt off 2.76pt}]  (371.86,31.18) -- (387.76,76.03) ;
\draw [color={rgb, 255:red, 0; green, 0; blue, 0 }  ,draw opacity=1 ][line width=1.5]  [dash pattern={on 1.69pt off 2.76pt}]  (445.07,81.69) -- (387.76,76.03) ;
\draw  [fill={rgb, 255:red, 0; green, 0; blue, 0 }  ,fill opacity=1 ][line width=1.5]  (390.69,75.66) .. controls (390.69,77.26) and (389.39,78.56) .. (387.79,78.56) .. controls (386.19,78.56) and (384.89,77.26) .. (384.89,75.66) .. controls (384.89,74.06) and (386.19,72.76) .. (387.79,72.76) .. controls (389.39,72.76) and (390.69,74.06) .. (390.69,75.66) -- cycle ;
\draw [line width=1.5]    (428.71,138.08) -- (387.76,76.03) ;
\draw  [fill={rgb, 255:red, 0; green, 0; blue, 0 }  ,fill opacity=1 ][line width=1.5]  (431.64,137.71) .. controls (431.64,139.31) and (430.34,140.61) .. (428.74,140.61) .. controls (427.14,140.61) and (425.84,139.31) .. (425.84,137.71) .. controls (425.84,136.11) and (427.14,134.81) .. (428.74,134.81) .. controls (430.34,134.81) and (431.64,136.11) .. (431.64,137.71) -- cycle ;
\draw  [fill={rgb, 255:red, 0; green, 0; blue, 0 }  ,fill opacity=1 ][line width=1.5]  (349.74,137.71) .. controls (349.74,139.31) and (348.44,140.61) .. (346.84,140.61) .. controls (345.23,140.61) and (343.94,139.31) .. (343.94,137.71) .. controls (343.94,136.11) and (345.23,134.81) .. (346.84,134.81) .. controls (348.44,134.81) and (349.74,136.11) .. (349.74,137.71) -- cycle ;
\draw [line width=1.5]    (346.81,138.08) -- (387.76,76.03) ;
\draw [line width=1.5]    (346.81,138.08) -- (428.71,138.08) ;
\draw  [color={rgb, 255:red, 45; green, 72; blue, 133 }  ,draw opacity=1 ][fill={rgb, 255:red, 45; green, 72; blue, 133 }  ,fill opacity=1 ][line width=1.5]  (443.73,80.13) .. controls (444.71,79.51) and (446.09,79.69) .. (446.83,80.55) .. controls (447.57,81.41) and (447.38,82.62) .. (446.41,83.24) .. controls (445.43,83.87) and (444.05,83.68) .. (443.31,82.82) .. controls (442.57,81.96) and (442.76,80.76) .. (443.73,80.13) -- cycle ;
\draw  [color={rgb, 255:red, 136; green, 4; blue, 22 }  ,draw opacity=1 ][fill={rgb, 255:red, 136; green, 4; blue, 22 }  ,fill opacity=1 ][line width=1.5]  (338.51,62.33) .. controls (337.7,61.53) and (337.79,60.31) .. (338.71,59.63) .. controls (339.62,58.94) and (341.02,59.04) .. (341.83,59.85) .. controls (342.64,60.66) and (342.55,61.87) .. (341.63,62.56) .. controls (340.72,63.24) and (339.32,63.14) .. (338.51,62.33) -- cycle ;
\draw  [color={rgb, 255:red, 136; green, 4; blue, 22 }  ,draw opacity=1 ][fill={rgb, 255:red, 136; green, 4; blue, 22 }  ,fill opacity=1 ][line width=1.5]  (370.2,32.42) .. controls (369.39,31.61) and (369.48,30.4) .. (370.39,29.71) .. controls (371.31,29.03) and (372.71,29.13) .. (373.51,29.93) .. controls (374.32,30.74) and (374.24,31.96) .. (373.32,32.64) .. controls (372.4,33.33) and (371,33.23) .. (370.2,32.42) -- cycle ;
\draw [color={rgb, 255:red, 136; green, 4; blue, 22 }  ,draw opacity=1 ][fill={rgb, 255:red, 136; green, 4; blue, 22 }  ,fill opacity=1 ][line width=1.5]    (371.86,31.18) -- (340.17,61.09) ;
\draw [color={rgb, 255:red, 0; green, 0; blue, 0 }  ,draw opacity=1 ][line width=1.5]  [dash pattern={on 1.69pt off 2.76pt}]  (167.17,58.73) -- (214.26,77.07) ;
\draw [color={rgb, 255:red, 0; green, 0; blue, 0 }  ,draw opacity=1 ][line width=1.5]  [dash pattern={on 1.69pt off 2.76pt}]  (197.17,31.86) -- (214.37,77.04) ;
\draw [color={rgb, 255:red, 0; green, 0; blue, 0 }  ,draw opacity=1 ][line width=1.5]  [dash pattern={on 1.69pt off 2.76pt}]  (241.65,36.86) -- (214.37,77.04) ;
\draw [color={rgb, 255:red, 0; green, 0; blue, 0 }  ,draw opacity=1 ][line width=1.5]  [dash pattern={on 1.69pt off 2.76pt}]  (265.86,68.06) -- (214.37,77.04) ;
\draw  [fill={rgb, 255:red, 0; green, 0; blue, 0 }  ,fill opacity=1 ] (217.27,76.64) .. controls (217.27,78.25) and (215.98,79.54) .. (214.37,79.54) .. controls (212.77,79.54) and (211.48,78.25) .. (211.48,76.64) .. controls (211.48,75.04) and (212.77,73.74) .. (214.37,73.74) .. controls (215.98,73.74) and (217.27,75.04) .. (217.27,76.64) -- cycle ;
\draw [line width=1.5]    (254.94,138.3) -- (214.37,77.04) ;
\draw  [fill={rgb, 255:red, 0; green, 0; blue, 0 }  ,fill opacity=1 ] (257.84,137.9) .. controls (257.84,139.5) and (256.55,140.8) .. (254.94,140.8) .. controls (253.34,140.8) and (252.05,139.5) .. (252.05,137.9) .. controls (252.05,136.3) and (253.34,135) .. (254.94,135) .. controls (256.55,135) and (257.84,136.3) .. (257.84,137.9) -- cycle ;
\draw  [fill={rgb, 255:red, 0; green, 0; blue, 0 }  ,fill opacity=1 ] (176.7,137.9) .. controls (176.7,139.5) and (175.41,140.8) .. (173.81,140.8) .. controls (172.2,140.8) and (170.91,139.5) .. (170.91,137.9) .. controls (170.91,136.3) and (172.2,135) .. (173.81,135) .. controls (175.41,135) and (176.7,136.3) .. (176.7,137.9) -- cycle ;
\draw [line width=1.5]    (173.81,138.3) -- (214.37,77.04) ;
\draw [line width=1.5]    (173.81,138.3) -- (254.94,138.3) ;
\draw  [color={rgb, 255:red, 45; green, 72; blue, 133 }  ,draw opacity=1 ][fill={rgb, 255:red, 45; green, 72; blue, 133 }  ,fill opacity=1 ][line width=1.5]  (240.31,35.3) .. controls (241.27,34.69) and (242.65,34.88) .. (243.39,35.74) .. controls (244.13,36.59) and (243.95,37.79) .. (242.99,38.41) .. controls (242.03,39.03) and (240.65,38.83) .. (239.91,37.98) .. controls (239.17,37.12) and (239.35,35.92) .. (240.31,35.3) -- cycle ;
\draw  [color={rgb, 255:red, 45; green, 72; blue, 133 }  ,draw opacity=1 ][fill={rgb, 255:red, 45; green, 72; blue, 133 }  ,fill opacity=1 ][line width=1.5]  (264.56,66.56) .. controls (265.53,65.94) and (266.91,66.13) .. (267.65,66.99) .. controls (268.38,67.85) and (268.2,69.04) .. (267.24,69.66) .. controls (266.28,70.28) and (264.9,70.09) .. (264.16,69.23) .. controls (263.42,68.37) and (263.6,67.18) .. (264.56,66.56) -- cycle ;
\draw [color={rgb, 255:red, 45; green, 72; blue, 133 }  ,draw opacity=1 ][fill={rgb, 255:red, 45; green, 72; blue, 133 }  ,fill opacity=1 ][line width=1.5]    (265.86,68.06) -- (241.65,36.86) ;
\draw  [color={rgb, 255:red, 136; green, 4; blue, 22 }  ,draw opacity=1 ][fill={rgb, 255:red, 136; green, 4; blue, 22 }  ,fill opacity=1 ][line width=1.5]  (165.58,60.01) .. controls (164.76,59.2) and (164.84,57.99) .. (165.75,57.32) .. controls (166.66,56.64) and (168.05,56.74) .. (168.86,57.55) .. controls (169.67,58.36) and (169.59,59.56) .. (168.69,60.24) .. controls (167.78,60.92) and (166.39,60.81) .. (165.58,60.01) -- cycle ;
\draw  [color={rgb, 255:red, 136; green, 4; blue, 22 }  ,draw opacity=1 ][fill={rgb, 255:red, 136; green, 4; blue, 22 }  ,fill opacity=1 ][line width=1.5]  (195.53,33.08) .. controls (194.72,32.28) and (194.8,31.07) .. (195.71,30.39) .. controls (196.61,29.72) and (198.01,29.82) .. (198.82,30.63) .. controls (199.63,31.44) and (199.55,32.64) .. (198.64,33.32) .. controls (197.74,34) and (196.34,33.89) .. (195.53,33.08) -- cycle ;
\draw [color={rgb, 255:red, 136; green, 4; blue, 22 }  ,draw opacity=1 ][fill={rgb, 255:red, 136; green, 4; blue, 22 }  ,fill opacity=1 ][line width=1.5]    (197.17,31.86) -- (167.22,58.78) ;
\draw [color={rgb, 255:red, 0; green, 0; blue, 0 }  ,draw opacity=1 ][line width=1.5]  [dash pattern={on 1.69pt off 2.76pt}]  (586.64,90.82) -- (563.07,137.71) ;
\draw [color={rgb, 255:red, 0; green, 0; blue, 0 }  ,draw opacity=1 ][line width=1.5]  [dash pattern={on 1.69pt off 2.76pt}]  (586.64,90.82) -- (522.1,76.03) ;
\draw  [fill={rgb, 255:red, 0; green, 0; blue, 0 }  ,fill opacity=1 ][line width=1.5]  (525.02,75.66) .. controls (525.02,77.26) and (523.72,78.56) .. (522.12,78.56) .. controls (520.52,78.56) and (519.22,77.26) .. (519.22,75.66) .. controls (519.22,74.06) and (520.52,72.76) .. (522.12,72.76) .. controls (523.72,72.76) and (525.02,74.06) .. (525.02,75.66) -- cycle ;
\draw [line width=1.5]    (563.05,138.08) -- (522.1,76.03) ;
\draw  [fill={rgb, 255:red, 0; green, 0; blue, 0 }  ,fill opacity=1 ][line width=1.5]  (565.97,137.71) .. controls (565.97,139.31) and (564.67,140.61) .. (563.07,140.61) .. controls (561.47,140.61) and (560.17,139.31) .. (560.17,137.71) .. controls (560.17,136.11) and (561.47,134.81) .. (563.07,134.81) .. controls (564.67,134.81) and (565.97,136.11) .. (565.97,137.71) -- cycle ;
\draw  [fill={rgb, 255:red, 0; green, 0; blue, 0 }  ,fill opacity=1 ][line width=1.5]  (484.07,137.71) .. controls (484.07,139.31) and (482.77,140.61) .. (481.17,140.61) .. controls (479.57,140.61) and (478.27,139.31) .. (478.27,137.71) .. controls (478.27,136.11) and (479.57,134.81) .. (481.17,134.81) .. controls (482.77,134.81) and (484.07,136.11) .. (484.07,137.71) -- cycle ;
\draw [line width=1.5]    (481.14,138.08) -- (522.1,76.03) ;
\draw [line width=1.5]    (481.14,138.08) -- (563.05,138.08) ;
\draw  [color={rgb, 255:red, 45; green, 72; blue, 133 }  ,draw opacity=1 ][fill={rgb, 255:red, 45; green, 72; blue, 133 }  ,fill opacity=1 ][line width=1.5]  (585.31,89.27) .. controls (586.28,88.64) and (587.67,88.83) .. (588.4,89.69) .. controls (589.14,90.55) and (588.95,91.75) .. (587.98,92.38) .. controls (587.01,93) and (585.62,92.81) .. (584.88,91.95) .. controls (584.15,91.1) and (584.34,89.89) .. (585.31,89.27) -- cycle ;
\draw  [color={rgb, 255:red, 136; green, 4; blue, 22 }  ,draw opacity=1 ][fill={rgb, 255:red, 136; green, 4; blue, 22 }  ,fill opacity=1 ][line width=1.5]  (563.7,61.75) .. controls (562.89,60.95) and (562.98,59.73) .. (563.89,59.05) .. controls (564.81,58.36) and (566.21,58.46) .. (567.01,59.27) .. controls (567.82,60.08) and (567.74,61.29) .. (566.82,61.98) .. controls (565.9,62.66) and (564.5,62.56) .. (563.7,61.75) -- cycle ;

\draw (429.33,154.05) node [anchor=north west][inner sep=0.75pt]    {$K_{3} \ in\ \mathcal{T}_{2}$};
\draw (184.33,154.07) node [anchor=north west][inner sep=0.75pt]    {$K_{3} \ in\ \mathcal{T}_{1}$};
%
\end{tikzpicture}
\caption{Members in $\mathcal T_1$ and $\mathcal T_2$.}\label{FIG:Members-T1-T2}
\end{figure}
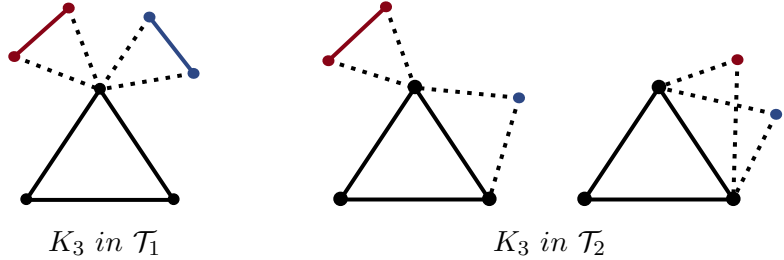

For a triangle $xyz$, an edge $uv$ and a vertex $w$ in $H$, we say
\begin{itemize}
    \item $uv$ \emph{sees} $xyz$ if there exists a vertex $v_1\in \{x,y,z\}$ such that $uv_1,vv_1\in H$. In this case, we also say $uv$ \emph{sees}  $v_1$ of $xyz$. 
    \item $w$ \emph{sees} $xyz$ if there exist two vertices $v_1,v_2\in \{x,y,z\}$ such that $wv_1,wv_2\in H$.
\end{itemize}
In what follows, we adopt the convention that if an edge in $\mathcal M$ sees a triangle $xyz\in \mathcal T_1$, then  it specifically sees the vertex $x$ of $xyz$. We refer to $x$ as the \emph{critical vertex} of $ xyz$, while $y$  and $z$ \emph{non-critical vertices} of $ xyz$. Let $V'$ be the set of all non-critical vertices of the triangles in $\mathcal{T}_1$, and $V''=V(\mathcal{T}_1)\setminus V'$.

Below, we partition $\mathcal T$ further into four subfamilies, as illustrated in Figure~\ref{FIG:Members-T1-T2} for $\mathcal T_1$ and $\mathcal T_2$.
\begin{itemize}
    \item $\mathcal T_1$ consists of members in $\mathcal T$ that are seen by at least two members in $\mathcal M$.
    \item $\mathcal T_2$ consists of members in $\mathcal T \setminus \mathcal{T}_{1}$ that are seen by at least two members in $\mathcal T$ or one member in $\mathcal M$ and at least one member in $\mathcal I$.
    \item $\mathcal T_3$ and $\mathcal T_4$ are obtained through the following process: We start with $D$ equal to the set of all triangles in $\mathcal{T}\setminus (\mathcal T_1\cup\mathcal  T_2)$, and $S = \emptyset$. If there is a triangle in $D$ that sends at most $8(|D| - 1)$ edges to the other triangles in $D$, we move it to $S$. 
    We repeat until $D$ contains no more such triangles. We then set $\mathcal T_3 \coloneqq S$, and $\mathcal T_4 \coloneqq D$. The construction of $\mathcal T_3$ ensures that $e(H[\mathcal T_3]) + e(H[\mathcal T_3, \mathcal T_4])$ is small enough. In $\mathcal T_4$, there is not only a high edge-density but also a “minimum-degree” condition between triangles. 
\end{itemize}

We say such a partition $(\mathcal{T}_1, \mathcal{T}_2, \mathcal{T}_3, \mathcal{T}_4)$ is an \emph{ideal partition} of $\mathcal{T}$. 

In the remainder of this paper, we set 
\begin{enumerate}[label=(\roman*)]
    \item $N \coloneqq \mathrm{ex}(n,(t+1)K_3)+2$, 
    \item $m \coloneqq |\mathcal{M}|$, 
    \item $i \coloneqq |\mathcal I|$, 
    \item $t_j \coloneqq |\mathcal{T}_j|$ and $\bar{t}_{j} \coloneqq |\mathcal{T}| - t_j$ for $j \in [4]$. 
\end{enumerate}

\textbf{Organization of the paper}: We present some preliminary results in the next section. In Section \ref{SEC:Stability}, 
we establish bounds on the cardinalities of the components in the maximal tiling triple $(\mathcal{T}, \mathcal{M}, \mathcal{I})$ for $(t+2)K_3$-free graphs $H$ with $n$ vertices and $N$ edges. In Section~\ref{SEC:H-Structure}, we prove that any representative graph $H$ under an edge-coloring of $K_n$ with $N$ colors avoiding rainbow $(t+2)K_3$ must contain a large rainbow complete tripartite subgraph. In Section~\ref{SEC:PF-Main-result}, we provide a complete proof of Theorem~\ref{THM:main-first-interval}.


\section{Preliminaries}\label{SEC:prelim}
Let $H$ be a graph. For every vertex set $S \subseteq V(H)$, let $H[S]$ denote the \emph{induced subgraph} of $H$ on $S$. The number of edges in $H[S]$ is denoted by $e_H(S)$ for simplicity. We use $H - S$ to denote the induced subgraph of $H$ on $V(H)\setminus S$. 
Given two vertex sets $S, T \subseteq V(H)$, let $e_{H}(S,T)$ be the number of edges with one end in $S$ and the other in $T$. For  subgraphs $H_1, H_2$ in $H$, we  we define $e(H_1, H_2)\coloneqq e(V(H_1), V(H_2))$ for notational convenience.  Given  vertex-disjoint subsets $S_1,\dots,S_k \subseteq V(H)$, let $H[S_1,\dots,S_k]=H\cap K[S_1,\dots,S_k]$ denote the \emph{induced $k$-partite subgraph} of $H$. 

In their proof of Theorem~\ref{THM:Turan-number-(t+1)K_3}, Allen-Böttcher-Hladký-Piguet~\cite{ABHP15} considered the maximal tiling triple $(\mathcal T,\mathcal M,\mathcal I)$ of  a $(t+1)K_3$-free graph $H$ with an ideal partition $(\mathcal T_1,\mathcal T_2,\mathcal T_3,\mathcal T_4)$ of $\mathcal T$, and completely characterized the edge structure both between and within all components of $(\mathcal T_1,\mathcal T_2,\mathcal T_3,\mathcal T_4,\mathcal M,\mathcal I)$ (see Appendix).

Define
\begin{align*}
     f\coloneqq &f\left(\tau_{1}, \tau_{2}, \tau_{3}, \tau_{4}, \mu, \iota\right)\\ =& 4 \mu \tau_{1}+2 \iota \tau_{1}+7\binom{\tau_1}{2}+3 \tau_{1}+2 \iota \tau_{2}+8\binom{\tau_2}{2}+3 \tau_{2} +8\binom{\tau_3}{2}+8 \tau_{3} \tau_{4} \\ &+3 \tau_{3}  +7 \tau_{1} \tau_{2}+(2+3 \mu) \tau_{2}+7 \tau_{1}\left(\tau_{3}+\tau_{4}\right) +(3+3 \mu) \tau_{3} \\ &+8 \tau_{2}\left(\tau_{3}+\tau_{4}\right)+(2+\iota) \tau_{3}, 
\end{align*}
 and 
\begin{align*}
    h \coloneqq h(\tau_1,\tau_2,\tau_3,\tau_4,\mu,\iota)=f+\iota\mu +\mu^2+(3\mu+3)\tau_4+(\iota+2)\tau_4+\binom{3\tau_4}{2}.
\end{align*}
Remark that $h$ represents an upper bound on the number of edges within or between certain parts of the maximal tiling triple in $H$. This bound heavily depends on Lemmas~\ref{LEM:upper-bound-each-part-1},~\ref{LEM:upper-bound-each-part-2} and \ref{LEM:upper-bound-each-part-3}. For instance, the term 
$4 \mu \tau_{1}$ in $f$ corresponds to the maximum number of edges between $\mathcal T_1$ and $\mathcal M$ with $|\mathcal T_1|=\tau_1$ and $|\mathcal M|=\mu$, as ensured by Lemma~\ref{LEM:upper-bound-each-part-1}~\ref{itm:mt1}. In our proof, some bounds provided by those lemmas  can be further refined, allowing for stronger results. Our proof relies crucially on the following lemma, established by  Allen-Böttcher-Hladký-Piguet~\cite{ABHP15}. 

\begin{lemma}[\cite{ABHP15}]\label{LEMMA:upper-bound-H-h1-ABHP}
    Let $H$ be a $(t+1)K_3$-free graph  with $n$ vertices, $(\mathcal T, \mathcal M, \mathcal I)$ be a maximal tiling triple of $H$ and $(\mathcal T_1,\mathcal T_2,\mathcal T_3,\mathcal T_4)$  be an ideal partition of $\mathcal T$. Then 
    \begin{align*}
        |H |\le h(t_1,t_2,t_3,t_4,m,i). 
    \end{align*}
\end{lemma}
Back to $h$, the following equalities can be derived through straightforward calculations. These results demonstrate that $h$ attains its maximum value at the boundary.
\begin{align}
    h(\tau_1-x,\tau_2+x,\tau_3,\tau_4,\mu,\iota)-h&=\frac{1}{2} x (x+2 \tau_2+2 \tau_3+2 \tau_4-2 \mu+3)\label{Eq:h1-2-1};\\
    h(\tau_1+x,\tau_2,\tau_3-x,\tau_4,\mu,\iota)-h&=\frac{1}{2} x (x-2 \tau_2-2\tau_3-2\tau_4+2 \mu+2 \iota-9)\label{Eq:h1-3-1};\\
    h(\tau_1,\tau_2+x,\tau_3-x,\tau_4,\mu,\iota)-h&=(\iota-3) x\label{Eq:h1-3-2};\\
    h(\tau_1+x,\tau_2,\tau_3,\tau_4-x,\mu,\iota)-h&=x (x-\tau_2-\tau_3-2 \tau_4+\mu+\iota-4)\label{Eq:h1-4-1};\\
    h(\tau_1,\tau_2+x,\tau_3,\tau_4-x,\mu,\iota)-h&=\frac{1}{2} x (x-2 \tau_4+2\iota-5)\label{Eq:h1-4-2};\\
    h(\tau_1,\tau_2,\tau_3-x,\tau_4+x,\mu,\iota)-h&=\frac{1}{2} x (x+2 \tau_4-1)\label{Eq:h1-4-3}.
\end{align}

A straightforward calculation demonstrates the following two facts. 
\begin{fact}\label{FACT:upper-bound-f+im+m^2}
    For positive integers $n,t,m,i$ with $2m+i=n-3(t+1)$, $t\le\frac{1}{9}(2n-6)-2$ and $n$ large enough. The following hold:
   \begin{enumerate}[label=\alph*)]
       \item\label{eq:f1+im+m2-t2=t+1} $h(0,t+1,0,0,m,i)=f(0,t+1,0,0,m,i)+im+m^2\le e_1(n,t)-\frac{n}{9}-\frac{49}{12}$;
       \item\label{eq:f1+im+m2-t1=t+1} $h(t+1,0,0,0,m,i)=f(t+1,0,0,0,m,i)+im+m^2\le e_1(n,t)+\frac14 (-i^2+2n-2 t-2)$.
   \end{enumerate}
\end{fact}

\begin{fact}\label{LEM:EQ-f1-max-t1-t2}
    Given non-negative integers $\tau_{1}, \tau_{2}, \tau_{3},\tau_{4},\mu,\iota$, then
    \begin{align*}
        f\left(\tau_{1}, \tau_{2}, \tau_{3}, \tau_{4}, \mu, \iota\right) \leq \max \left\{f\left(\tau_{1}+\tau_{2}, 0, \tau_{3}, \tau_{4}, \mu, \iota\right), f\left(0, \tau_{1}+\tau_{2}, \tau_{3}, \tau_{4}, \mu, \iota\right)\right\}, 
    \end{align*}
    with the equality holds only if  $\tau_{1}\tau_{2}=0$. Moreover, 
    \begin{align*}
        h(\tau_1,\tau_2,\tau_3,\tau_4,\mu,\iota)\le h(\tau_1,\tau_2,0,\tau_3+\tau_4,\mu,\iota).
    \end{align*}
\end{fact}
Let
\begin{align*}
    g(\tau_1,\tau_2,\tau_3,\tau_4,\mu,\iota)\coloneqq  f+\iota\mu +\mu^2+(3\mu+3)\tau_4+(\iota+2)\tau_4+8\binom{\tau_4}{2}+10\tau_4-28. 
\end{align*}
Then 
\begin{align}
    g(\tau_1,\tau_2+x,\tau_3,\tau_4-x,\mu,\iota)- g&=(\iota-10) x\label{Eq:gs1-4-2};\\
    g(\tau_1,\tau_2,\tau_3-x,\tau_4+x,\mu,\iota)- g&=7 x\label{Eq:gs1-3-4}.
\end{align}
Moreover, Allen et al.~\cite{ABHP15} established a sharper upper bound on the number of edges in $H$ by analyzing the structural properties of $\mathcal{T}_4$.
\begin{lemma}[{\cite[Lemma~6.1]{ABHP15}}]\label{LEM:T4-sparse}
    Let $H$ be a $(t+1)K_3$-free graph  with $n$ vertices, $(\mathcal T, \mathcal M, \mathcal I)$ be a maximal tiling triple of $H$ and $(\mathcal T_1,\mathcal T_2,\mathcal T_3,\mathcal T_4)$  be an ideal partition of $\mathcal T$. If  $e(\mathcal T_{4}) \le8\binom{t_4}{2}+10 t_{4}-28$, then 
    \begin{align*}
        |H |\le  g\left(t_{1}, t_{2}, t_{3}, t_{4}, m, i\right).
    \end{align*}
\end{lemma}
The next lemma can be derived from Lemmas~6.1 and~6.2 and the proof of Lemma~6.2 in~\cite{ABHP15}.
\begin{lemma}[\cite{ABHP15}]\label{LEM:T4-dense}
    Let $H$ be a $(t+1)K_3$-free graph  with $n$ vertices, $(\mathcal T, \mathcal M, \mathcal I)$ be a maximal tiling triple of $H$ and $(\mathcal T_1,\mathcal T_2,\mathcal T_3,\mathcal T_4)$  be an ideal partition of $\mathcal T$. Suppose that $\sum_{i=1}^4t_i=t$ and $e(\mathcal T_{4}) \ge8\binom{t_4}{2}+10 t_{4}-27$. The following statements hold.  
    \begin{enumerate}[label=(\roman*)]
        \item\label{itm:t4-dense-few} If $t_{4}<\frac{2 m+i}{3}$, then  
            \begin{align*}
            |H |\le  h\left(t_{1}, t_{2}, t_{3}, t_{4}, m, i\right).
            \end{align*}
         \item\label{itm:t4-dense-more-t1=1} If $t_{4} =\frac{2 m+i}{3}$ and $t\le \frac{2n-6}{9}-2$, then  
            \begin{align*}
            |H |\le e_1(n,t)-\frac{1}{2000}n^2+O(n).
            \end{align*}
         \item\label{itm:t4-dense-more-t1>1} If $t_{4} > \frac{2 m+i}{3}$ and  $n/5\le t\le3n/10$, then  
        \begin{align*}
            |H |\le \binom{2t+1}{2}+(2t+1)(n-2t-1)-\frac{n^2}{100}+O(n).
        \end{align*}
     \end{enumerate}
\end{lemma}

We remark that in the original paper \cite{ABHP15}, the bound on $|H |$ involves a more complex expression, requiring the introduction of additional auxiliary functions. However, the simplified version presented in Lemma \ref{LEM:T4-dense} above suffices for our purposes.

We end this section with the classical Andr\'{a}sfai--Erd\H{o}s--S\'{o}s Theorem for triangles. 

\begin{theorem}[\cite{AES74}]\label{THM:AES-Thm}
    Every $n$-vertex triangle-free graph with minimum degree greater than $2n/5$ must be bipartite.
\end{theorem}

\section{A stability result}\label{SEC:Stability}
In this section, we prove the following stability theorem that characterizes the approximate structure of a $(t+2)K_3$-free graph $H$ with $n$ vertices and $N$ edges. 
\begin{theorem}\label{THM:STABILITY}
    For every $\delta > 0$ there exists $N_{\ref{THM:STABILITY}} = N_{\ref{THM:STABILITY}}(\delta)$ such that the following holds for every $n \ge N_{\ref{THM:STABILITY}}$. 
    Let $H$ be a $(t+2)K_{3}$-free graph on $n$ vertices with $N$ edges, where $t$ is an integer in the interval $\left[ \delta n,~\frac{2n-6}{9} - 2 \right]$. 
    Suppose that $(\mathcal T, \mathcal M, \mathcal I)$ is a maximal tiling triple of $H$ and $(\mathcal T_1,\mathcal T_2,\mathcal T_3,\mathcal T_4)$ is an ideal partition of $\mathcal T$. Then 
    \begin{enumerate}[label=(\roman*)]
        \item\label{THM:STABILITY-a} $t_1 \in [t-1,~t+1]$, 
        \item\label{THM:STABILITY-b} $m \in \left(\frac{n-3(t+1)}2 - \frac{\sqrt{2n}}{2},~\frac{n-3(t+1)}2\right]$, and 
        \item\label{THM:STABILITY-c} $i < \sqrt{2 n}$. 
    \end{enumerate}
\end{theorem}

\begin{proof}[Proof of Theorem~\ref{THM:STABILITY}]
Let $n$ be a sufficiently large integer. Let $H$ be a $(t + 2)K_3$-free graph with $n$ vertices and $N$ edges, $(\mathcal T, \mathcal M, \mathcal I)$ be a maximal tiling triple of $H$,  and $(\mathcal T_1,\mathcal T_2,\mathcal T_3,\mathcal T_4)$ be an ideal partition of $\mathcal T$. Then $|\mathcal T|=t+1$. It follows from $\sum_{j=1}^4 t_j=t+1$ and $n =3(t + 1)+2m+i$ that 
\begin{align}\label{EQ:m=f(n,t,i)}
    m=\frac{n-3(t + 1)-i}{2}\le \frac{n-3(t + 1)}{2}\quad \text{ and }\quad i=n-3(t+1)-2m.
\end{align}
By Lemma \ref{LEMMA:upper-bound-H-h1-ABHP}, we have $|H |\le h(t_1,t_2,t_3,t_4,m,i)$.

\textbf{Case 1.} $\delta n\le t\le (n-18)/5$.

It follows from Lemma \ref{LEMMA:upper-bound-H-h1-ABHP}, \eqref{Eq:h1-4-1} and \eqref{Eq:h1-4-3} that 
\begin{align*}
    |H|&\le h(t_1,t_2,0,t_3+t_4,m,i) \\
    &\le h(t_1+t_3+t_4,t_2,0,0,m,i)-(t_3+t_4)(m+i-t_2-t_3-t_4-4).
\end{align*}

By the fact $t\le (n-18)/5$ and \eqref{EQ:m=f(n,t,i)}, we have  
\begin{align*}
    m+i-t_2-t_3-t_4-4\ge \frac{n-5(t+1)-8}2\ge 0, 
\end{align*}
which together with Fact~\ref{LEM:EQ-f1-max-t1-t2} gives 
\begin{align}\label{EQ:upper-bound-|H |-samll-t}
    |H | &\le f(t_1+t_3+t_4,t_2,0,0,m,i)+im+m^2 \notag \\
    &\le\max\{f(t+1,0,0,0,m,i),f(0,t+1,0,0,m,i)\}+im+m^2. 
\end{align}

Using the fact that $|H |=N$ and Fact~\ref{FACT:upper-bound-f+im+m^2}~\ref{eq:f1+im+m2-t2=t+1}, we have $f(t+1,0,0,0,m,i)\ge f(0,t+1,0,0,m,i)$. Thus, by \eqref{EQ:upper-bound-|H |-samll-t} and Fact~\ref{FACT:upper-bound-f+im+m^2}~\ref{eq:f1+im+m2-t1=t+1}, we have 
\begin{align*}
    |H |\le  f(t+1,0,0,0,m,i)+im+m^2\le e_1(n,t)+\frac14 (-i^2 + 2 n-2 t-2).
\end{align*}

Applying the assumption $|H |=N$ once again, we obtain  $\frac14 (-i^2 + 2 n-2 t-2)\ge 2$. It follows from \eqref{EQ:m=f(n,t,i)} that 
\begin{align*}
     i<\sqrt{2n} \text{ and thus } \frac{n-3(t+1)-\sqrt{2n}}2<m\le\frac{n-3(t+1)}2. 
\end{align*}

In order to bound $t_1$, we give another upper bound on $|H |$. Recall that $|H |\le h(t_1,t_2,0,t_3+t_4,m,i)$ by Lemma \ref{LEMMA:upper-bound-H-h1-ABHP} and  Fact~\ref{LEM:EQ-f1-max-t1-t2}. 
It follows from \eqref{Eq:h1-4-2} and $\overline{t}_1=t+1-t_1$ that 
\begin{align*}
    |H |&\le  \max\{h(t_1,0,0,\overline{t}_1,m,i),h(t_1,\overline{t}_1,0,0,m,i)\}. 
\end{align*}

We claim that $t_1\ge t-1$. Suppose to the contrary that $\overline{t}_1\ge 3$. If $h(t_1,0,0,\overline{t}_1,m,i)\le h(t_1,\overline{t}_1,0,0,m,i)$, then 
\begin{align}\label{EQ:upper-bound-|H |-first-case-t1>=3}
    |H |\le h(t_1,\overline{t}_1,0,0,m,i)=\frac{1}{4} \left(-i^2+n^2+2 n t+2 n-t^2-4 t-3\right)+\alpha_1(\overline{t}_1),
\end{align}
where $\alpha_1(\overline{t}_1)=\frac12\overline{t}_1^2-\frac{1}{2} \overline{t}_1(  n-i-3 t-6)$. Recall that $\overline{t}_1\le t+1\le  (n-13)/5$. So, $\alpha_1(\overline{t}_1)$ reaches its  maximum when $\overline{t}_1=3$. So substituting $\alpha_1(3)$ into \eqref{EQ:upper-bound-|H |-first-case-t1>=3} yields 
\begin{align*}
    |H |&{\le}
    \frac{1}{4} \left(-(i-3)^2+n^2+2 n (t-2)-t^2+14 t+60\right)\\
    &\le \frac{1}{4} \left(n^2+2 n (t-2)-t^2+14 t+60\right)\\
    &\le e_1(n,t)-n+4t+O(1)<N,
\end{align*}
a contradiction.  Otherwise, 
\begin{align*}
    |H |\le h(t_1,0,0,\overline{t}_1,m,i)
    =\frac{ 1 }4\left(-i^2 + 2 n + n^2-4 t + 2 n t-t^2-3\right)+\alpha_2(\overline{t}_1),
\end{align*}
where $\alpha_2(\overline{t}_1)=\overline{t}_1^2-\frac{1}{2}\overline{t}_1(n+i-3t-11)$. Using $3\le \overline{t}_1\le t+1\le(n-13)/5$,  $\alpha_2(\overline{t}_1)$ reaches its  maximum when $\overline{t}_1=3$. Thus,  
\begin{align*}
    |H |
    &\le \frac{1}4\left(-i(i+6)+99 +n^2+2n(t-2)+14 t-t^2\right)\\
    &\le\frac{1}4\left(99 + n^2 + 2 n (t-2) + 14 t-t^2\right)\\
    &\le e_1(n,t)-n+4 t+O(1)<N,
\end{align*}
also a contradiction. 
Thus, we have $t_1\ge t-1.$\\

It remains to analyze the case when $t > (n-18)/5$. We proceed by dividing the remaining proof into two cases based on the number of edges in  $H[V(\mathcal{T}_4)]$.

{\bf {Case 2}}:  $t>(n-18)/5$ and $e(\mathcal T_4)\le8\binom{t_4}{2}+10t_4-28$.

By Lemma~\ref{LEM:T4-sparse}, we have 
    \begin{align*}
        |H |\le g(t_1,t_2,t_3,t_4,m,i).
    \end{align*}
Combining \eqref{Eq:gs1-4-2} with  \eqref{Eq:gs1-3-4},  one can  verify that
    \begin{align}\label{EQ:differ-gs-t2=t+1-t1}
        g\left(t_{1}, t_{2}, t_{3},t_{4}, m, i\right)-g(t_1,\overline{t}_1,0,0,m,i)=7t_4-(t_3+t_4)(i-3).
    \end{align}
Thus, we have
   \begin{align}
       |H | &\le g(t_1,\overline{t}_1,0,0,m,i)+7t_4-(t_3+t_4)(i-3)\notag\\
       &=f\left(t_{1},\overline{t}_1,0,0,m, i\right)+i m+m^2+7t_4-(t_3+t_4)(i-3)-28\label{EQ:upper-bound-|H |-gs1-f}\\
       &\le \max\{f\left(t+1,0,0,0,m, i\right),f\left(0,t+1,0,0,m, i\right)\}\notag\\
       &\quad+i m+m^2+7t_4-(t_3+t_4)(i-3)-28\label{EQ:upper-bound-|H |-gs1-max-f},
   \end{align}
where the last inequality holds by Fact~\ref{LEM:EQ-f1-max-t1-t2}. 

\textbf{Subcase 2.1.} $i\le12$.

In this case, it suffices to prove $t_1\ge t-1$. Suppose to the contrary that $\overline{t}_1\ge 3$. It follows from \eqref{EQ:m=f(n,t,i)} and $i\le12$ that 
\begin{align*}
    m\ge\frac{1}{2}(n-3(t+1))-6 \text{ and } 7t_4-(t_3+t_4)(i-3)\le 3t_3+10t_4\le 10\overline{t}_1
\end{align*}
By \eqref{EQ:upper-bound-|H |-gs1-f}, we have 
\begin{align*}
    |H | &\le f(t_1,\overline{t}_1,0,0,m,i)+i m+m^2+10\overline{t}_1-28\notag\\
    &\le \frac14(n^2-i^2-t^2+ 2 n t + 2 n -4 t -115)+\alpha_3(\overline{t}_1),
\end{align*}
where $\alpha_3\left(\overline{t}_1\right)=\frac{1}{2}\left(\overline{t}_1^2+(3t-n+i+26)\overline{t}_1\right)$.
Recall that $3\le \overline{t}_1\le t+1\le(2n-6)/9-1$, which means  $\alpha_3(\overline{t}_1)$ attains its maximum when  $\overline{t}_1=3$. Thus, 
\begin{align*}
    |H |
    &\le\frac{1}{4} \left(-(i-3)^2+n^2+2 n t-t^2+14t-4n+68\right)\\
    &\le\frac{1}{4} \left(n^2+2 n t-t^2+14t-4n+68\right)\\
    &\le e_1(n,t)-n + 4 t+O(1)
    <N,
\end{align*}
a contradiction.  

\textbf{Subcase 2.2.} $i\ge13$.

It follows from  $i\ge13$ that $7t_4-(t_3+t_4)(i-3)\le-10t_3-3t_4\le0$. Observe that $f\left(t+1,0,0,0,m, i\right)\ge f\left(0,t+1,0,0,m, i\right)$ by \eqref{EQ:upper-bound-|H |-gs1-f} and Fact~\ref{FACT:upper-bound-f+im+m^2}~\ref{eq:f1+im+m2-t2=t+1}. Thus, by \eqref{EQ:upper-bound-|H |-gs1-max-f} and Fact~\ref{FACT:upper-bound-f+im+m^2}~\ref{eq:f1+im+m2-t2=t+1}, we have
\begin{align*}
    |H |\le  f\left(t+1,0,0,0,m, i\right)+i m+m^2
    \le  e_1(n,t)+\frac14 (-i^2+2n-2 t-2).
\end{align*}
Using $|H|=N$, we have $\frac14(-i^2+2n-2 t-2)\ge2$, which together with \eqref{EQ:m=f(n,t,i)} gives
\begin{align*}
    i<\sqrt{2n} \quad \text{and then}\quad \frac{n-3(t+1)-\sqrt{2n}}2<m\le\frac{n-3(t+1)}2. 
\end{align*}
To derive a lower bound for  $t_1$, assume for a contradiction that  $\overline{t}_1\ge 3$. Then  by \eqref{EQ:differ-gs-t2=t+1-t1}, 
    \begin{align*}
       |H |
       &\le g(t_1,\overline{t}_1,0,0,m,i)\le\frac{1}{4} \left(-i^2+n^2+2 n t+2 n-t^2-4 t-115\right)+\alpha_4(\overline{t}_1),
    \end{align*}
where $\alpha_4(\overline{t}_1)=\frac{1}{2}\overline{t}_1^2-\frac{1}{2} \overline{t}_1(n-i-3 t-6)$ which attains its maximum when $\overline{t}_1=3$. Thus,
    \begin{align*}
        |H |&\le-\frac{1}{4}(i-3)^2+\frac{1}{4} \left(n^2+2 n (t-2)-t^2+14 t-52\right)\\
        &\le\frac{1}{4} \left(n^2+2 n (t-2)-t^2+14 t-52\right)\\
        &=e_1(n,t)-n + 4t+O(1)<N,
    \end{align*}
a contradiction.

       
\textbf{Case 3.} $t>(n-18)/5$ and  $e(\mathcal T_4)\ge8\binom{t_4}{2}+10t_4-27$.

Here, our strategy remains to bound the number of edges of $H$ via Lemma~\ref{LEM:T4-dense}. By the continuity of the functions involved, along with the condition $|H | = N$ and  Lemma~\ref{LEM:T4-dense} \ref{itm:t4-dense-more-t1=1}, we have  $3t_4\neq 2m+i.$
If $3t_4> 2m+i$ and $n/5\le t\le 3n/10$, then by Lemma~\ref{LEM:T4-dense} \ref{itm:t4-dense-more-t1>1}, we have 
\begin{align*}
     |H | &\le \binom{2(t+1)+1}{2}+(2(t+1)+1)(n-2(t+1)-1)-\frac1{100}n^2+O(n)\\
     &\le e_1(n,t)-\frac{7}{4}t^2+\frac{1}{2}t(3n-13)-\frac{1}{4}(n^2-12n+23)-\frac1{100}n^2+O(n)\\
     &=e_1(n,t)-\frac{7}{4}t^2+\frac{3}{2}nt-\frac{1}{4}n^2-\frac{1}{100}n^2+O(n)\\
     &\le e_1(n,t)-\frac{53}{4045}n^2+O(n)<N,
 \end{align*}
where the last inequality follows from $t\le\frac{1}{9}(2n-6)-2$. Note that while we have $t>(n-18)/5$,  Lemma~\ref{LEM:T4-dense} \ref{itm:t4-dense-more-t1>1} requires the stronger condition  $t\ge n/5$. Nevertheless, we can still apply the lemma by deleting at most four vertices from either $\mathcal M$ or $\mathcal I$. This modification affects only  $O(n)$ in calculating $|H |$, which is negligible for our purposes.

Thus, we have $3t_4< 2m+i$, which together with Lemma~\ref{LEM:T4-dense}~\ref{itm:t4-dense-few}  gives 
\begin{align}
     |H |\le h(t_1, t_{2}, t_{3}, t_{4}, m, i).\label{EQ:upper-bound-|H |-gl1-subcase2}
\end{align}
By \eqref{Eq:h1-4-1},\eqref{Eq:h1-4-2} and \eqref{Eq:h1-4-3},  we have that the function $ h(\tau_1,\tau_2,\tau_3,\tau_4,\mu,\iota)$ attains its maximum when either $\tau_{4}=0$ or $\tau_{4}=\frac{1}{3}(n-3(t+1)-1)$ provided   $\sum_{i=1}^4\tau_i=t+1$. 


\begin{claim}\label{CLAIM:maximum-t4=0-case-large-T4}
    The function $h$ attains its maximum when $\tau_{4}=0$.  
\end{claim}

\begin{proof}
By contradiction, it follows from \eqref{EQ:upper-bound-|H |-gl1-subcase2} that 
\begin{align*}
    |H |\le h(t_1, t_{2}, t_{3}, \frac13\left( n-3(t+1)-1\right), m, i). 
\end{align*}

If $h$ attains its maximum when $\tau_1>0$, then by \eqref{Eq:h1-2-1} and \eqref{Eq:h1-3-1}, we have  $\tau_1=2(t+1)-(n-1)/3$, $\tau_2=\tau_3=0$. Thus, 
\begin{align*}
    |H |&\le h\left(2(t+1)-\frac {n-1}3,0,0,\frac13\left( n-3(t+1)-1\right),m,i\right)\notag\\
    &= h(t+1,0,0,0,m,i)-\frac13\left( n-3(t+1)-1\right)\left(m+i-\frac13\left( n-3(t+1)-1\right)-4\right)\notag\\
    &\le h(t+1,0,0,0,m,i)-\frac13\left( n-3(t+1)-1\right)\left(\frac16(n-3(t+1))-\frac{11}3\right)\\
    &=\frac{1}{4} \left(-i^2+n^2+2 n (t+1)-t^2-4 t-3\right)-\frac{(n-3t)^2}{18}+O(n)\\
    &\le e_1(n,t)-\frac{(n-3t)^2}{18}+O(n)<N,
\end{align*}
a contradiction.  This implies that $h$ attains its maximum when $\tau_1=0$. Combining \eqref{Eq:h1-3-2}, \eqref{EQ:m=f(n,t,i)} and \eqref{EQ:upper-bound-|H |-gl1-subcase2}, we conclude that 
\begin{align*}
    |H |
    &\le h\left(0,2(t+1)-\frac {n-1}3,0,\frac13(n-3(t+1)-1),m,i\right)+3(t+1)\notag\\
    &= -\frac{1}{4} \left(i-\frac{1}{3} (9 t-2 n+11)\right)^2+\frac{1}{12} \left(5n^2+54t^2+142t-2n(8t+5)+80\right)\notag\\
    &\le\frac{1}{12} \left(5n^2+54t^2+142t-2n(8t+5)+80\right)\notag\\
    &\le e_1(n,t)+\alpha_5(t),
\end{align*}
where 
\begin{align*}
    \alpha_5(t)=\frac{19}{4}\left(t-\left(\frac{11}{57}n-\frac{74}{57}\right)\right)^2-\frac{1}{684}(7 n^2-1058n+745).
\end{align*}
Since $\frac{1}{5}(n-18)\le t\le \frac{1}{9}(2n-6)-2$, $\alpha_5(t)$ reaches its maximum when $t=\frac{1}{9}(2n-6)-2$. Thus, 
\begin{align*}
    |H|\le e_1(n,t)-\frac{1}{162} n^2+\frac{7}{6} n+\frac{281}{36}<N,
\end{align*}
a contradiction. This completes the proof of Claim \ref{CLAIM:maximum-t4=0-case-large-T4}. 
\end{proof}

By Claim \ref{CLAIM:maximum-t4=0-case-large-T4}, there exists $t'_{i}\ge t_{i}$ for $i\in [3]$ and $t'_{1}+t'_{2}+t'_{3}=t+1$ such that 
\begin{align*}
    |H |\le   h\left(t'_{1},t'_2,t'_{3},  0,  m,i\right),
\end{align*}
which together with \eqref{Eq:h1-3-2} and Fact~\ref{LEM:EQ-f1-max-t1-t2} yields
\begin{align}
    |H |&\le h\left(t'_{1},t'_2+t'_{3}, 0, 0, m,i\right)-(i-3)t'_3\label{EQ:upper-bound-|H |-gs1-t1'-t2'+t3'}\\
    &=f(t'_{1},t'_2+t'_{3}, 0, 0, m,i)+im+m^2-(i-3)t'_3\notag\\
    &\le\max\{f(t+1,0,0,0,m,i),f(0,t+1,0,0,m,i)\}+im+m^2-(i-3)t'_3\label{EQ:upper-bound-|H |-gl1-max-f-i}.
\end{align}

\textbf{Subcase 3.1.} $i\le 2$.

Now we claim that $t_1'\ge t-1$. By contradiction, $\overline{t}_1'=t_2'+t_3'\ge 3$. It follows from \eqref{EQ:upper-bound-|H |-gs1-t1'-t2'+t3'} that 
\begin{align*}
    |H |&\le  h(t_1',t_2'+t_3',0,0,m,i)+(3-i)(t_2'+t_3')\notag\\
    &=\frac14(-3-i^2 + 2 n + n^2-4 t + 2 n t-t^2)+\alpha_6(\overline{t}_1'),
\end{align*}
where $\alpha_6(\overline{t}_1')=\frac12\overline{t}_1'^2- \frac12 ( i+n -3t-12) \overline{t}_1'$ which reaches its maximum when $\overline{t}_1'=3$. Thus, 
\begin{align*}
    |H |&\le\frac{1}{4} \left(-i^2-6 i+n^2+2 n (t-2)-t^2+14 t+87\right)\\
    &\le\frac{1}{4} \left(n^2+2 n (t-2)-t^2+14 t+87\right)\\
    &\le e_1(n,t)-n+4 t+O(1)<N,
\end{align*}
a contradiction.

Recall that $3t_4<2m+i$. In the following, we have 
\begin{align*}
    t_2+t_3\le t_2'+t_3'\le2,
    \text{ and }t_2+t_3+t_4\le\frac{n-3(t+1)}3+2.
\end{align*} 
Note that by \eqref{Eq:h1-3-2} and \eqref{Eq:h1-4-2}, we have
\begin{align*}
    |H |&\le h(t_1,t_2+t_3,0,t_4,m,i)+(3-i)t_3\\&\le\max\{h(t_1,\overline{t}_1,0,0,m,i),h(t_1,0,0,\overline{t}_1,m,i)\}+(3-i)\overline{t}_1.
\end{align*}
If $h(t_1,\overline{t}_1,0,0,m,i)\ge h(t_1,0,0,\overline{t}_1,m,i)$, then  by an argument analogous to the proof of $t_1'\ge t-1$ above, we have $t_1\ge t-1$. Otherwise,  
\begin{align*}
    |H |&\le  h(t_1,0,0,\overline{t}_1,m,i)+(3-i)\overline{t}_1\notag\\
    &=\frac{1}{4} \left(-i^2+n^2+2 n t+2 n-t^2-4 t-3\right)+\alpha_{7}(\overline{t}_1),
\end{align*}
where $\alpha_{7}(\overline{t}_1)=\overline{t}_1^2-\frac{1}{2}\overline{t}_1 (3 i+n-3 t-17)$. In this case, if $\overline{t}_1\ge 3$, then $\alpha_{7}(\overline{t}_1)$  attains its maximum when $\overline{t}_1=3$; and so 
\begin{align*}
    |H |&\le\frac{1}{4} \left(-i^2-18 i+n^2+2 n (t-2)-t^2+14 t+135\right)\\
    &\le\frac{1}{4} \left(n^2+2 n (t-2)-t^2+14 t+135\right)\\
    &\le e_1(n,t)-n+4 t+O(1)<N,
\end{align*}
 a contradiction.


{\bf{Subcase 3.2}}: $i\ge3$.

It follows from $i\ge3$ that $(i-3)t_3'\ge0$. By \eqref{EQ:upper-bound-|H |-gl1-max-f-i} and Fact~\ref{FACT:upper-bound-f+im+m^2}~\ref{eq:f1+im+m2-t2=t+1}, we have  $f(t+1,0,0,0,m,i)>f(0,t+1,0,0,m,i)$, and then 
\begin{align*}
    |H |&\le f(t+1,0,0,0,m,i)+im+m^2\le e_1(n,t)+\frac14 (-i^2+2n-2 t-2).
\end{align*}
Using $|H |=N $, we have $\frac14(-i^2+2n-2 t-2)\ge2$ and then 
\begin{align*}
    i<\sqrt{2n} \text{ and then }\frac{n-3(t+1)-\sqrt{2n}}2<m\le\frac{n-3(t+1)}2. 
\end{align*}

\begin{claim}\label{CLAIM:t2+t3-le-2-case-large-T4-i3}
   $t_2+t_3\le 2$. 
\end{claim}

\begin{proof} 
By contradiction, we have $\overline{t}_1'=t_2'+t_3'\ge t_2+t_3\ge3$. By  \eqref{EQ:upper-bound-|H |-gs1-t1'-t2'+t3'},
\begin{align*}
    |H |\le  h(t_{1}',\overline{t}_1', 0, 0, m,i)=\frac14(-i^2+n^2+2 n t+2 n-t^2-4 t-3) + \alpha_8(\overline{t}_1'),
\end{align*}
where $\alpha_8(\overline{t}_1')=\frac{1}{2}\overline{t}_1'^2-\frac{1}{2} \overline{t}_1'(n-i-3 t-6)$ which reaches its  maximum when $\overline{t}_1'=3$. Thus, 
\begin{align*}
    |H |&\le\frac{1}{4} \left(-(i-3)^2+n^2+2 n (t-2)-t^2+14 t+60\right)\\
    &\le\frac{1}{4} \left(n^2+2 n (t-2)-t^2+14 t+60\right)\\
    &\le e_1(n,t)- n + 4t+O(1)<N,
\end{align*}
a contradiction. 
\end{proof}

Recall that $3t_4<2m+i=n-3(t+1)$. In the following, we have 
\begin{align}
    t_2+t_3\le t_2'+t_3'\le2\text{ and }t_2+t_3+t_4\le\frac{n-3(t+1)}{3}+2.\label{EQ:tau_4=0-t_1>0-t2+t3+t4}
\end{align}
It suffices to show $t_1\ge t-1$. Suppose to the contrary, $\overline{t}_1\ge3$.  Combining \eqref{Eq:h1-4-2} and Fact~\ref{LEM:EQ-f1-max-t1-t2}, we have 
\begin{align*}
    |H |
    \le h(t_1,t_2,0,t_4+t_3,m,i)
    \le \max\{h(t_1,\overline{t}_1,0,0,m,i),h(t_1,0,0,\overline{t}_1,m,i)\}.
\end{align*}
If $h(t_1,\overline{t}_1,0,0,m,i)\ge h(t_1,0,0,\overline{t}_1,m,i)$, then we can get a contradiction by arguments analogous to the proof of Claim \ref{CLAIM:t2+t3-le-2-case-large-T4-i3}. Otherwise,  
\begin{align*}
     |H |\le h(t_1,0,0,\overline{t}_1,m,i)
     =\frac14(-i^2+n^2+2 n t+2 n-t^2-4 t-3) + \alpha_{9}(\overline{t}_1),
\end{align*}
where $\alpha_{9}(\overline{t}_1)= \overline{t}_1^2-(i+n-3 t-11) \overline{t}_1/2$.  By \eqref{EQ:tau_4=0-t_1>0-t2+t3+t4}, we have $3\le \overline{t}_1\le n/3-(t+1)+2$ and so $\alpha_{9}(\overline{t}_1)$ attains its maximum when $\overline{t}_1=3$. Thus, 
\begin{align*}
     |H |&\le\frac{1}{4} \left(-i^2-6 i+n^2+2 n (t-2)-t^2+14 t+99\right)\\
     &\le  \frac{1}{4} \left(n^2+2 n (t-2)-t^2+14 t+99\right)\\
     &\le e_1(n,t)-n+4t+O(1)<N,
\end{align*}
 a contradiction.  This completes the proof of Theorem \ref{THM:STABILITY}. 
\end{proof}


\section{Large rainbow complete tripartite  subgraph}\label{SEC:H-Structure}
 
Our main purpose in this section is to prove the following result. 

\begin{theorem}\label{THM:tripartite-graph}
    For every $\delta > 0$, there exists $N_{\ref{THM:tripartite-graph}} = N_{\ref{THM:tripartite-graph}}(\delta)$ such that the following holds for every $n \ge N_{\ref{THM:tripartite-graph}}$. 
    Let $t \in \left[\delta n,~\frac{2n-6}{9} + 1 \right]$ be an integer and $\chi \colon K_n \rightarrow [N]$ be a surjective edge coloring that admits no rainbow copy of $(t+2)K_3$. 
    Let $H \subseteq K_{n}$ be a representative graph of $\chi$. 
    Suppose that $(\mathcal{T}, \mathcal{M}, \mathcal{I})$ is a maximal tiling triple of $H$ and $(\mathcal{T}_1, \mathcal{T}_2, \mathcal{T}_3, \mathcal{T}_4)$ is an ideal partition  of $\mathcal{T}$ that satisfy~\ref{THM:STABILITY-a},~\ref{THM:STABILITY-b}, and~\ref{THM:STABILITY-c} in Theorem~\ref{THM:STABILITY}. Then 
    \begin{align*}
        K_{t_1-10,~m+t_1-10,~m+t_1-10}
        \subseteq H\left[V(\mathcal{T}_1) \cup V(\mathcal{M})\right]. 
    \end{align*}    
\end{theorem}
Under the condition~\ref{THM:STABILITY-a},~\ref{THM:STABILITY-b}, and~\ref{THM:STABILITY-c} in Theorem~\ref{THM:STABILITY}, the following facts is easy to check. 

\begin{fact}\label{FACT:function-q}
Let 
\begin{align*}
    q_1(t_1,t_2,t_3,t_4,m,i)\coloneqq h(t_1,t_2,t_3,t_4,m,i)-m-t_1.
\end{align*}
Then $q_1(t_1,t_2,t_3,t_4,m,i)\le e_1(n,t)+1/4$.
\end{fact}

\begin{proof}
It is easy to check that $q_1(t_1,t_2,t_3,t_4,m,i)\le q_1(t+1,0,0,0,m,i)$. Therefore,
\begin{align*}
    q_1(t_1,t_2,t_3,t_4,m,i)
    &\le \frac{1}{4} \left(-(i - 1)^2+n^2+2 n t-t (t+2)\right)\notag\\
    &\le\frac{1}{4} \left(n^2+2 n t-t (t+2)\right)
    \le e_1(n,t)+\frac{1}{4},
\end{align*}
completing the proof. 
\end{proof}

\begin{fact}\label{fact1}
    We have $|H |\le  e_1(n,t)+\frac{1}{2} ( n- t-1)$.  Moreover, if  $\overline{t}_1>0$, then 
    \begin{align*}
        h(t_1,t_2,t_3,t_4,m,i)\le e_1(n,t)+t+6. 
    \end{align*}
\end{fact}
\begin{proof}
By \eqref{Eq:h1-2-1}, \eqref{Eq:h1-4-1} and \eqref{Eq:h1-4-3}, we have $h(t_1,t_2,t_3,t_4,m,i)\le h(t+1,0,0,0,m,i)$ under the condition~\ref{THM:STABILITY-a},~\ref{THM:STABILITY-b}, and~\ref{THM:STABILITY-c} in Theorem~\ref{THM:STABILITY}. Combining this with Fact~\ref{FACT:upper-bound-f+im+m^2}~\ref{eq:f1+im+m2-t1=t+1}, we conclude that $|H|\le e_1(n,t)+\frac{1}{2} ( n- t-1)$. 
  
For the moreover part, if $t_2 > 0$ and $t_4 > 0$, then by \eqref{Eq:h1-4-1}and \eqref{Eq:h1-4-3}, we have
\begin{align}
    h(t_1+t_2-1,1,t_3,t_4,m,i)-h&=\frac{1}{2} (t_2-1) (2 m-(t_2-1)+2t_3+2t_4+3)\ge0;\label{EQ:h1-half-2-1}\\
    h(t_1+t_4-1,t_2,t_3,1,m,i)-h&=(t_4 -1)(i+m-t_2-t_3-(t_4-1)-4)\ge0.\label{EQ:h1-half-1-4}
\end{align}
Thus,
\begin{align*}
    &h(t_1,t_2,t_3,t_4,m,i)\le\max\{ h(t,1,0,0,m,i),h(t,0,0,1,m,i)\}\\
    &=\max\left\{\frac{1}{4} \left(-(i-1)^2+n^2+2 n t-t^2+2 t+12\right),\frac{1}{4} \left(-i^2-2 i+n^2+2 n t-t^2+2 t+23\right)\right\}\\
    &\le\max\left\{\frac{1}{4} \left(n^2+2 n t-t^2+2 t+12\right),\frac{1}{4} \left(n^2+2 n t-t^2+2 t+23\right)\right\}\\
    &\le e_1(n,t)+\max\left\{t+\frac{13}{4},t+6\right\}=e_1(n,t)+t+6.
\end{align*}
Here we complete the proof.
\end{proof}

\begin{proof}[Proof of Theorem \ref{THM:tripartite-graph}]

Suppose that $H$ is the representative graph of a surjective edge-coloring $\chi: K_n \rightarrow [N]$ that admits no rainbow copy of $(t+2)K_3$. Let $(\mathcal{T}, \mathcal{M}, \mathcal{I})$  be a maximal tiling triple of $H$ with an ideal partition $(\mathcal{T}_1, \mathcal{T}_2, \mathcal{T}_3, \mathcal{T}_4)$ of $\mathcal{T}$ which satisfying~\ref{THM:STABILITY-a},~\ref{THM:STABILITY-b}, and~\ref{THM:STABILITY-c} in Theorem~\ref{THM:STABILITY}. It follows from~\ref{THM:STABILITY-b} in Theorem~\ref{THM:STABILITY} and $ t < \frac{2n-6}{9} + 1$ that 
    \begin{align}\label{EQ:lower-bound-m-n/6-tripartite}
        m\ge \frac{n}{6}-o(n). 
    \end{align}
Before proceeding, we state some necessary properties of $H$. Note that by Lemma \ref{LEMMA:upper-bound-H-h1-ABHP}, 
    \begin{align*}
        |H |\le h(t_1,t_2,t_3,t_4,m,i). 
    \end{align*}

\begin{claim}\label{CLM:two-edge-contian-C4}
    There exist $u_1v_1,u_2v_2\in\mathcal M$ such that $H[\{u_1,v_1,u_2,v_2\}]$ contains a cycle of length 4. 
\end{claim}
\begin{proof}
Otherwise, there are at most one edges between pair of edges in $\mathcal M$ which implies that $e(\mathcal{M})\leq\binom{m}{2}+m$. We update the bound involving  $e(\mathcal{M})$ in $h$ to $e(\mathcal{M})\leq\binom{m}{2}+m$, replacing the earlier estimate $e(\mathcal{M})\leq m^2$ given by Lemma \ref{LEM:upper-bound-each-part-1}, which together with \eqref{EQ:lower-bound-m-n/6-tripartite} and Fact \ref{fact1} yields that 
\begin{align*}
    |H |&\le f+i m+m^2+(3+3m)t_4+(2+i)t_4+\binom{3t_4}{2}-\left(m^2-\left(\binom{m}{2}+m\right)\right)\\
    &\le e_1(n,t)-\frac{1}{2}m^2+O(n)<N,
\end{align*}
a contradiction. 
\end{proof}

Recall that each triangle $xyz\in \mathcal T_1$ are seen by at least two edges in $\mathcal M$.
\begin{claim}\label{CLM:matching-sees-same-vertex}
For $ xyz\in\mathcal{T}_1$, if an $\mathcal{M}$-edge $uv$ sees the vertex $x$ of $ xyz$, then for any of the remaining $\mathcal{M}$-edge that sees $xyz$, these $\mathcal M$-edges are restricted to seeing only $x$ of $xyz$.
\end{claim}
\begin{proof}
By contradiction, if there is a matching edge $u'v'\in \mathcal M\setminus\{uv\}$ that sees  $y$ of  $ xyz$. Then $\mathcal{T}\setminus\{ xyz\}\cup\{ uvx, u'v'y\}$ is a copy of $(t + 2)K_3 $ in $H$, a contradiction. 
\end{proof}

By Claim \ref{CLM:matching-sees-same-vertex},  for  $ xyz\in\mathcal T_1$ and $uv\in\mathcal M$, then 
\begin{align}\label{EQ:bound-edge-triangle-edge-T1-M}
    e(uv, xyz)\le4,
\end{align}
and the equality holds only if $uv$ sees $ xyz$. 

\begin{claim}\label{CLM:at-most-two-triangle}
    For $\varepsilon \le 10^{-3}$, there are at most two triangles in $\mathcal{T}_1$ such that each of them is seen by less than $\varepsilon n$ $\mathcal{M}$-edges.
\end{claim}
\begin{proof}
We consider the number of edges between $\mathcal{T}_1$ and $\mathcal{M}$. For $xyz\in \mathcal{T}_1$, by \eqref{EQ:bound-edge-triangle-edge-T1-M}, then $e( xyz,\mathcal{M})\leq 4m$. Moreover, if there are fewer than $\varepsilon n$ $\mathcal{M}$-edges that see $xyz$, then $e(xyz,\mathcal{M})\leq 4m-(m-\varepsilon n)$. Thus, 
if there are at least three triangles in $\mathcal{T}_1$ are each seen by less than $\varepsilon n$ $\mathcal{M}$-edges, then $$e(\mathcal{T}_1,\mathcal{M})\leq 4mt_1-3(m -\varepsilon n).$$ 
We update the bound on $e(\mathcal{T}_1,\mathcal{M})$ in $h$, replacing the earlier estimate $e(\mathcal{T}_1,\mathcal{M})\leq 4mt_1$ given by Lemma \ref{LEM:upper-bound-each-part-1}, yields that 
\begin{align*}
    |H |&\le f(t_1,t_2,t_3,t_4,m,i)+i m+m^2+(3+3m)t_4+(2+i)t_4+\binom{3t_4}{2}-3(m-\varepsilon n)\\
    &\le e_1(n,t) +\frac{1}{2} ( n- t-1)-3\left(\frac{1}2(n-3(t+1))-2\varepsilon n\right)\\
    &\le e_1(n,t)+4t-n+6\varepsilon n<N,
\end{align*}
 a contradiction. Note that the second inequality follows from  Fact \ref{fact1}. 
\end{proof}

\begin{claim}\label{CLM:structure-edge-between-T1}
    For any two triangles $ x_1y_1z_1, x_2y_2z_2\in\mathcal T_1$, we have that $z_1y_2$ and $z_1z_2$ (similarly, $y_1z_2$ and $y_1y_2$ ) cannot appear in $H$ simultaneously.  Consequently, $e(y_1, x_2y_2z_2)\le2$ and $e(z_1, x_2y_2z_2)\le2$.
\end{claim}

\begin{proof}
Suppose there exist two triangles $x_1y_1z_1, x_2y_2z_2\in \mathcal{T}_1$, along with two $\mathcal{M}$-edges $u_1v_1$ and $u_2v_2$, such that $u_1v_1$ sees $ x_1y_1z_1$ and $u_2v_2$ sees $ x_2y_2z_2$. 
If $z_1y_2,z_1z_2\in E(H)$, then 
\begin{align*}
    \left( \mathcal T\setminus\{ x_1y_1z_1, x_2y_2z_2\} \right) \cup\{u_1v_1x_1,u_2v_2x_2,z_1z_2y_2\}
\end{align*}
forms a copy of $(t+2)K_3$ in $H$ (see Figure~\ref{fig:z_1y_2,z_1z_2-|H |}), a contradiction. Similarly, at most one of $y_1y_2,y_1z_2$ can appear in $E(H)$. 
\end{proof}

\begin{figure}[htbp]
\begin{minipage}[t]{0.35\linewidth}
\centering

\tikzset{every picture/.style={line width=0.75pt}} 

\begin{tikzpicture}[x=0.75pt,y=0.75pt,yscale=-1,xscale=1]

\draw [color={rgb, 255:red, 136; green, 4; blue, 22 }  ,draw opacity=1 ][line width=1.5]  [dash pattern={on 1.69pt off 2.76pt}]  (53.72,105.17) -- (84.78,138.17) ;
\draw [color={rgb, 255:red, 136; green, 4; blue, 22 }  ,draw opacity=1 ][line width=1.5]    (53.72,105.17) -- (53.72,171.18) ;
\draw [color={rgb, 255:red, 136; green, 4; blue, 22 }  ,draw opacity=1 ][line width=1.5]  [dash pattern={on 1.69pt off 2.76pt}]  (53.72,171.18) -- (84.78,138.17) ;
\draw  [fill={rgb, 255:red, 0; green, 0; blue, 0 }  ,fill opacity=1 ] (51.99,107.49) .. controls (53.28,108.45) and (55.1,108.18) .. (56.05,106.89) .. controls (57.01,105.61) and (56.74,103.79) .. (55.45,102.84) .. controls (54.16,101.88) and (52.35,102.15) .. (51.39,103.44) .. controls (50.43,104.72) and (50.7,106.54) .. (51.99,107.49) -- cycle ;
\draw  [fill={rgb, 255:red, 0; green, 0; blue, 0 }  ,fill opacity=1 ] (51.99,173.5) .. controls (53.28,174.46) and (55.1,174.19) .. (56.05,172.91) .. controls (57.01,171.62) and (56.74,169.81) .. (55.45,168.85) .. controls (54.16,167.9) and (52.35,168.16) .. (51.39,169.45) .. controls (50.43,170.73) and (50.7,172.55) .. (51.99,173.5) -- cycle ;

\draw [color={rgb, 255:red, 136; green, 4; blue, 22 }  ,draw opacity=1 ][line width=1.5]  [dash pattern={on 1.69pt off 2.76pt}]  (115.29,82.84) -- (154.51,163.44) ;
\draw [color={rgb, 255:red, 136; green, 4; blue, 22 }  ,draw opacity=1 ][line width=1.5]  [dash pattern={on 1.69pt off 2.76pt}]  (115.29,82.84) -- (154.18,82.14) ;
\draw [color={rgb, 255:red, 0; green, 0; blue, 0 }  ,draw opacity=1 ][line width=1.5]    (115.18,164.11) -- (83.46,138.84) ;
\draw [color={rgb, 255:red, 0; green, 0; blue, 0 }  ,draw opacity=1 ][line width=1.5]    (115.51,82.8) -- (83.46,138.84) ;
\draw [color={rgb, 255:red, 0; green, 0; blue, 0 }  ,draw opacity=1 ][line width=1.5]    (115.18,164.11) -- (115.51,82.8) ;
\draw  [fill={rgb, 255:red, 0; green, 0; blue, 0 }  ,fill opacity=1 ] (118.16,82.45) .. controls (118.38,84.03) and (117.27,85.5) .. (115.68,85.71) .. controls (114.1,85.93) and (112.64,84.82) .. (112.42,83.23) .. controls (112.2,81.65) and (113.31,80.18) .. (114.9,79.97) .. controls (116.48,79.75) and (117.95,80.86) .. (118.16,82.45) -- cycle ;
\draw  [fill={rgb, 255:red, 0; green, 0; blue, 0 }  ,fill opacity=1 ] (86.33,137.45) .. controls (86.55,139.03) and (85.44,140.49) .. (83.85,140.71) .. controls (82.27,140.93) and (80.81,139.82) .. (80.59,138.23) .. controls (80.37,136.64) and (81.48,135.18) .. (83.07,134.97) .. controls (84.65,134.75) and (86.11,135.86) .. (86.33,137.45) -- cycle ;
\draw [color={rgb, 255:red, 0; green, 0; blue, 0 }  ,draw opacity=1 ][line width=1.5]    (154.51,163.44) -- (185.92,138.17) ;
\draw [color={rgb, 255:red, 0; green, 0; blue, 0 }  ,draw opacity=1 ][line width=1.5]    (154.18,82.14) -- (185.92,138.17) ;
\draw [color={rgb, 255:red, 0; green, 0; blue, 0 }  ,draw opacity=1 ][line width=1.5]    (154.51,163.44) -- (154.18,82.14) ;
\draw  [fill={rgb, 255:red, 0; green, 0; blue, 0 }  ,fill opacity=1 ] (151.55,81.78) .. controls (151.33,83.37) and (152.43,84.83) .. (154,85.05) .. controls (155.57,85.26) and (157.02,84.15) .. (157.24,82.57) .. controls (157.45,80.98) and (156.35,79.52) .. (154.78,79.3) .. controls (153.21,79.08) and (151.77,80.19) .. (151.55,81.78) -- cycle ;
\draw [color={rgb, 255:red, 136; green, 4; blue, 22 }  ,draw opacity=1 ][line width=1.5]  [dash pattern={on 1.69pt off 2.76pt}]  (216.92,105.17) -- (185.92,138.17) ;
\draw [color={rgb, 255:red, 136; green, 4; blue, 22 }  ,draw opacity=1 ][line width=1.5]    (216.92,105.17) -- (216.92,171.18) ;
\draw [color={rgb, 255:red, 136; green, 4; blue, 22 }  ,draw opacity=1 ][line width=1.5]  [dash pattern={on 1.69pt off 2.76pt}]  (216.92,171.18) -- (185.92,138.17) ;
\draw  [fill={rgb, 255:red, 0; green, 0; blue, 0 }  ,fill opacity=1 ] (218.65,107.49) .. controls (217.36,108.45) and (215.54,108.18) .. (214.59,106.89) .. controls (213.64,105.61) and (213.9,103.79) .. (215.19,102.84) .. controls (216.48,101.88) and (218.29,102.15) .. (219.25,103.44) .. controls (220.2,104.72) and (219.93,106.54) .. (218.65,107.49) -- cycle ;
\draw  [fill={rgb, 255:red, 0; green, 0; blue, 0 }  ,fill opacity=1 ] (218.65,173.5) .. controls (217.36,174.46) and (215.54,174.19) .. (214.59,172.91) .. controls (213.64,171.62) and (213.9,169.81) .. (215.19,168.85) .. controls (216.48,167.9) and (218.29,168.16) .. (219.25,169.45) .. controls (220.2,170.73) and (219.93,172.55) .. (218.65,173.5) -- cycle ;

\draw  [fill={rgb, 255:red, 0; green, 0; blue, 0 }  ,fill opacity=1 ] (187.65,140.5) .. controls (186.36,141.45) and (184.54,141.19) .. (183.59,139.9) .. controls (182.64,138.62) and (182.9,136.8) .. (184.19,135.84) .. controls (185.48,134.89) and (187.29,135.16) .. (188.25,136.44) .. controls (189.2,137.73) and (188.93,139.54) .. (187.65,140.5) -- cycle ;
\draw  [fill={rgb, 255:red, 0; green, 0; blue, 0 }  ,fill opacity=1 ] (156.23,165.77) .. controls (154.95,166.72) and (153.13,166.45) .. (152.18,165.17) .. controls (151.22,163.88) and (151.49,162.07) .. (152.78,161.11) .. controls (154.07,160.16) and (155.88,160.43) .. (156.84,161.71) .. controls (157.79,163) and (157.52,164.81) .. (156.23,165.77) -- cycle ;
\draw  [fill={rgb, 255:red, 0; green, 0; blue, 0 }  ,fill opacity=1 ] (116.9,166.43) .. controls (115.62,167.39) and (113.8,167.12) .. (112.85,165.83) .. controls (111.89,164.55) and (112.16,162.73) .. (113.45,161.78) .. controls (114.74,160.82) and (116.55,161.09) .. (117.51,162.38) .. controls (118.46,163.66) and (118.19,165.48) .. (116.9,166.43) -- cycle ;

\draw (54,92) node [anchor=north west][inner sep=0.75pt]    {$u_{1}$};
\draw (100.67,164) node [anchor=north west][inner sep=0.75pt]    {$y_{1}$};
\draw (218.77,92) node [anchor=north west][inner sep=0.75pt]    {$u_{2}$};
\draw (220.19,166.4) node [anchor=north west][inner sep=0.75pt]    {$v_{2}$};
\draw (35.43,166.4) node [anchor=north west][inner sep=0.75pt]    {$v_{1}$};
\draw (156.43,164) node [anchor=north west][inner sep=0.75pt]    {$y_{2}$};
\draw (176,145) node [anchor=north west][inner sep=0.75pt]    {$x_{2}$};
\draw (109,67) node [anchor=north west][inner sep=0.75pt]    {$z_{1}$};
\draw (149.1,67) node [anchor=north west][inner sep=0.75pt]    {$z_{2}$};
\draw (77.1,145) node [anchor=north west][inner sep=0.75pt]    {$x_{1}$};
\draw (122,199.4) node [anchor=north west][inner sep=0.75pt]    {$\ \ \ \ \ \ \ \ \ $};
\end{tikzpicture}

\caption{$z_1y_2,z_1z_2 \in |H |$.}
     \label{fig:z_1y_2,z_1z_2-|H |}
\end{minipage}%
\begin{minipage}[t]{0.75\linewidth}

\centering

\tikzset{every picture/.style={line width=0.75pt}} 

\begin{tikzpicture}[x=0.7pt,y=0.7pt,yscale=-0.9,xscale=0.9]

\draw   (152.89,80.39) .. controls (152.89,68.6) and (188.42,59.05) .. (232.25,59.05) .. controls (276.09,59.05) and (311.62,68.6) .. (311.62,80.39) .. controls (311.62,92.18) and (276.09,101.73) .. (232.25,101.73) .. controls (188.42,101.73) and (152.89,92.18) .. (152.89,80.39) -- cycle ;
\draw [color={rgb, 255:red, 136; green, 4; blue, 22 }  ,draw opacity=1 ][line width=1.5]  [dash pattern={on 1.69pt off 2.76pt}]  (290.51,167.11) -- (240.51,149.11) ;
\draw [color={rgb, 255:red, 136; green, 4; blue, 22 }  ,draw opacity=1 ][line width=1.5]  [dash pattern={on 1.69pt off 2.76pt}]  (188.51,167.11) -- (290.51,167.11) ;
\draw [color={rgb, 255:red, 136; green, 4; blue, 22 }  ,draw opacity=1 ][line width=1.5]  [dash pattern={on 1.69pt off 2.76pt}]  (188.51,167.11) -- (240.51,149.11) ;
\draw [color={rgb, 255:red, 0; green, 0; blue, 0 }  ,draw opacity=1 ][line width=1.5]    (188.51,167.11) -- (156.79,141.84) ;
\draw [color={rgb, 255:red, 0; green, 0; blue, 0 }  ,draw opacity=1 ][line width=1.5]    (188.84,85.8) -- (156.79,141.84) ;
\draw [color={rgb, 255:red, 0; green, 0; blue, 0 }  ,draw opacity=1 ][line width=1.5]    (188.51,167.11) -- (188.84,85.8) ;
\draw  [color={rgb, 255:red, 136; green, 4; blue, 22 }  ,draw opacity=1 ][fill={rgb, 255:red, 136; green, 4; blue, 22 }  ,fill opacity=1 ][line width=1.5]  (186.37,169.21) .. controls (185.21,168.03) and (185.22,166.13) .. (186.41,164.97) .. controls (187.59,163.81) and (189.49,163.82) .. (190.65,165) .. controls (191.81,166.18) and (191.8,168.08) .. (190.61,169.24) .. controls (189.43,170.41) and (187.53,170.39) .. (186.37,169.21) -- cycle ;
\draw  [fill={rgb, 255:red, 0; green, 0; blue, 0 }  ,fill opacity=1 ] (191.5,85.45) .. controls (191.71,87.03) and (190.6,88.5) .. (189.02,88.71) .. controls (187.43,88.93) and (185.97,87.82) .. (185.75,86.23) .. controls (185.54,84.65) and (186.65,83.18) .. (188.23,82.97) .. controls (189.82,82.75) and (191.28,83.86) .. (191.5,85.45) -- cycle ;
\draw  [fill={rgb, 255:red, 0; green, 0; blue, 0 }  ,fill opacity=1 ] (159.66,141.45) .. controls (159.88,143.03) and (158.77,144.49) .. (157.19,144.71) .. controls (155.6,144.93) and (154.14,143.82) .. (153.92,142.23) .. controls (153.7,140.64) and (154.81,139.18) .. (156.4,138.97) .. controls (157.99,138.75) and (159.45,139.86) .. (159.66,141.45) -- cycle ;

\draw   (142.89,127.43) .. controls (142.89,119.66) and (149.18,113.37) .. (156.95,113.37) -- (308.33,113.37) .. controls (316.1,113.37) and (322.39,119.66) .. (322.39,127.43) -- (322.39,169.6) .. controls (322.39,177.36) and (316.1,183.66) .. (308.33,183.66) -- (156.95,183.66) .. controls (149.18,183.66) and (142.89,177.36) .. (142.89,169.6) -- cycle ;
\draw [color={rgb, 255:red, 0; green, 0; blue, 0 }  ,draw opacity=1 ][line width=1.5]    (290.51,167.11) -- (258.79,141.84) ;
\draw [color={rgb, 255:red, 0; green, 0; blue, 0 }  ,draw opacity=1 ][line width=1.5]    (290.84,85.8) -- (258.79,141.84) ;
\draw [color={rgb, 255:red, 0; green, 0; blue, 0 }  ,draw opacity=1 ][line width=1.5]    (290.51,167.11) -- (290.84,85.8) ;
\draw  [color={rgb, 255:red, 136; green, 4; blue, 22 }  ,draw opacity=1 ][fill={rgb, 255:red, 136; green, 4; blue, 22 }  ,fill opacity=1 ][line width=1.5]  (288.37,169.21) .. controls (287.21,168.03) and (287.22,166.13) .. (288.41,164.97) .. controls (289.59,163.81) and (291.49,163.82) .. (292.65,165) .. controls (293.81,166.18) and (293.8,168.08) .. (292.61,169.24) .. controls (291.43,170.41) and (289.53,170.39) .. (288.37,169.21) -- cycle ;
\draw  [fill={rgb, 255:red, 0; green, 0; blue, 0 }  ,fill opacity=1 ] (293.5,85.45) .. controls (293.71,87.03) and (292.6,88.5) .. (291.02,88.71) .. controls (289.43,88.93) and (287.97,87.82) .. (287.75,86.23) .. controls (287.54,84.65) and (288.65,83.18) .. (290.23,82.97) .. controls (291.82,82.75) and (293.28,83.86) .. (293.5,85.45) -- cycle ;
\draw  [fill={rgb, 255:red, 0; green, 0; blue, 0 }  ,fill opacity=1 ] (261.66,141.45) .. controls (261.88,143.03) and (260.77,144.49) .. (259.19,144.71) .. controls (257.6,144.93) and (256.14,143.82) .. (255.92,142.23) .. controls (255.7,140.64) and (256.81,139.18) .. (258.4,138.97) .. controls (259.99,138.75) and (261.45,139.86) .. (261.66,141.45) -- cycle ;

\draw [color={rgb, 255:red, 0; green, 0; blue, 0 }  ,draw opacity=1 ][line width=1.5]    (240.51,149.11) -- (208.79,123.84) ;
\draw [color={rgb, 255:red, 0; green, 0; blue, 0 }  ,draw opacity=1 ][line width=1.5]    (240.84,67.8) -- (208.79,123.84) ;
\draw [color={rgb, 255:red, 0; green, 0; blue, 0 }  ,draw opacity=1 ][line width=1.5]    (240.51,149.11) -- (240.84,67.8) ;
\draw  [color={rgb, 255:red, 136; green, 4; blue, 22 }  ,draw opacity=1 ][fill={rgb, 255:red, 136; green, 4; blue, 22 }  ,fill opacity=1 ][line width=1.5]  (238.37,151.21) .. controls (237.21,150.03) and (237.22,148.13) .. (238.41,146.97) .. controls (239.59,145.81) and (241.49,145.82) .. (242.65,147) .. controls (243.81,148.18) and (243.8,150.08) .. (242.61,151.24) .. controls (241.43,152.41) and (239.53,152.39) .. (238.37,151.21) -- cycle ;
\draw  [fill={rgb, 255:red, 0; green, 0; blue, 0 }  ,fill opacity=1 ] (243.5,67.45) .. controls (243.71,69.03) and (242.6,70.5) .. (241.02,70.71) .. controls (239.43,70.93) and (237.97,69.82) .. (237.75,68.23) .. controls (237.54,66.65) and (238.65,65.18) .. (240.23,64.97) .. controls (241.82,64.75) and (243.28,65.86) .. (243.5,67.45) -- cycle ;
\draw  [fill={rgb, 255:red, 0; green, 0; blue, 0 }  ,fill opacity=1 ] (211.66,123.45) .. controls (211.88,125.03) and (210.77,126.49) .. (209.19,126.71) .. controls (207.6,126.93) and (206.14,125.82) .. (205.92,124.23) .. controls (205.7,122.64) and (206.81,121.18) .. (208.4,120.97) .. controls (209.99,120.75) and (211.45,121.86) .. (211.66,123.45) -- cycle ;

\draw [color={rgb, 255:red, 136; green, 4; blue, 22 }  ,draw opacity=1 ][line width=1.5]  [dash pattern={on 1.69pt off 2.76pt}]  (380.26,43.71) -- (402.2,85.71) ;
\draw [color={rgb, 255:red, 136; green, 4; blue, 22 }  ,draw opacity=1 ][line width=1.5]    (380.26,43.71) -- (424.15,43.7) ;
\draw [color={rgb, 255:red, 136; green, 4; blue, 22 }  ,draw opacity=1 ][line width=1.5]  [dash pattern={on 1.69pt off 2.76pt}]  (424.15,43.7) -- (402.2,85.71) ;
\draw  [fill={rgb, 255:red, 0; green, 0; blue, 0 }  ,fill opacity=1 ] (383.06,42.94) .. controls (383.48,44.49) and (382.57,46.08) .. (381.03,46.51) .. controls (379.49,46.93) and (377.89,46.02) .. (377.47,44.48) .. controls (377.05,42.93) and (377.95,41.34) .. (379.5,40.91) .. controls (381.04,40.49) and (382.64,41.4) .. (383.06,42.94) -- cycle ;
\draw  [fill={rgb, 255:red, 0; green, 0; blue, 0 }  ,fill opacity=1 ] (426.94,42.93) .. controls (427.36,44.48) and (426.45,46.07) .. (424.91,46.5) .. controls (423.37,46.92) and (421.77,46.01) .. (421.35,44.47) .. controls (420.93,42.92) and (421.84,41.33) .. (423.38,40.9) .. controls (424.92,40.48) and (426.52,41.39) .. (426.94,42.93) -- cycle ;

\draw   (362.56,87.63) .. controls (362.56,79.28) and (390.59,72.52) .. (425.17,72.52) .. controls (459.75,72.52) and (487.78,79.28) .. (487.78,87.63) .. controls (487.78,95.97) and (459.75,102.73) .. (425.17,102.73) .. controls (390.59,102.73) and (362.56,95.97) .. (362.56,87.63) -- cycle ;
\draw [color={rgb, 255:red, 136; green, 4; blue, 22 }  ,draw opacity=1 ][line width=1.5]  [dash pattern={on 1.69pt off 2.76pt}]  (402.18,167.11) -- (432.79,142.17) ;
\draw [color={rgb, 255:red, 136; green, 4; blue, 22 }  ,draw opacity=1 ][line width=1.5]  [dash pattern={on 1.69pt off 2.76pt}]  (402.18,167.11) -- (464.51,167.44) ;
\draw [color={rgb, 255:red, 0; green, 0; blue, 0 }  ,draw opacity=1 ][line width=1.5]    (402.18,167.11) -- (370.46,141.84) ;
\draw [color={rgb, 255:red, 0; green, 0; blue, 0 }  ,draw opacity=1 ][line width=1.5]    (402.51,85.8) -- (370.46,141.84) ;
\draw [color={rgb, 255:red, 0; green, 0; blue, 0 }  ,draw opacity=1 ][line width=1.5]    (402.18,167.11) -- (402.51,85.8) ;
\draw  [color={rgb, 255:red, 136; green, 4; blue, 22 }  ,draw opacity=1 ][fill={rgb, 255:red, 136; green, 4; blue, 22 }  ,fill opacity=1 ][line width=1.5]  (400.04,169.21) .. controls (398.88,168.03) and (398.89,166.13) .. (400.07,164.97) .. controls (401.25,163.81) and (403.15,163.82) .. (404.31,165) .. controls (405.48,166.18) and (405.46,168.08) .. (404.28,169.24) .. controls (403.1,170.41) and (401.2,170.39) .. (400.04,169.21) -- cycle ;
\draw  [fill={rgb, 255:red, 0; green, 0; blue, 0 }  ,fill opacity=1 ] (405.16,85.45) .. controls (405.38,87.03) and (404.27,88.5) .. (402.68,88.71) .. controls (401.1,88.93) and (399.64,87.82) .. (399.42,86.23) .. controls (399.2,84.65) and (400.31,83.18) .. (401.9,82.97) .. controls (403.48,82.75) and (404.95,83.86) .. (405.16,85.45) -- cycle ;
\draw  [fill={rgb, 255:red, 0; green, 0; blue, 0 }  ,fill opacity=1 ] (373.33,141.45) .. controls (373.55,143.03) and (372.44,144.49) .. (370.85,144.71) .. controls (369.27,144.93) and (367.81,143.82) .. (367.59,142.23) .. controls (367.37,140.64) and (368.48,139.18) .. (370.07,138.97) .. controls (371.65,138.75) and (373.11,139.86) .. (373.33,141.45) -- cycle ;

\draw   (350.78,127.43) .. controls (350.78,119.66) and (357.07,113.37) .. (364.84,113.37) -- (488.72,113.37) .. controls (496.48,113.37) and (502.78,119.66) .. (502.78,127.43) -- (502.78,169.6) .. controls (502.78,177.36) and (496.48,183.66) .. (488.72,183.66) -- (364.84,183.66) .. controls (357.07,183.66) and (350.78,177.36) .. (350.78,169.6) -- cycle ;
\draw [color={rgb, 255:red, 136; green, 4; blue, 22 }  ,draw opacity=1 ][line width=1.5]    (464.51,167.44) -- (432.79,142.17) ;
\draw [color={rgb, 255:red, 0; green, 0; blue, 0 }  ,draw opacity=1 ][line width=1.5]    (464.84,86.14) -- (432.79,142.17) ;
\draw [color={rgb, 255:red, 0; green, 0; blue, 0 }  ,draw opacity=1 ][line width=1.5]    (464.51,167.44) -- (464.84,86.14) ;
\draw  [color={rgb, 255:red, 136; green, 4; blue, 22 }  ,draw opacity=1 ][fill={rgb, 255:red, 136; green, 4; blue, 22 }  ,fill opacity=1 ][line width=1.5]  (462.37,169.54) .. controls (461.21,168.36) and (461.22,166.46) .. (462.41,165.3) .. controls (463.59,164.14) and (465.49,164.15) .. (466.65,165.33) .. controls (467.81,166.52) and (467.8,168.42) .. (466.61,169.58) .. controls (465.43,170.74) and (463.53,170.72) .. (462.37,169.54) -- cycle ;
\draw  [fill={rgb, 255:red, 0; green, 0; blue, 0 }  ,fill opacity=1 ] (467.5,85.78) .. controls (467.71,87.37) and (466.6,88.83) .. (465.02,89.05) .. controls (463.43,89.26) and (461.97,88.15) .. (461.75,86.57) .. controls (461.54,84.98) and (462.65,83.52) .. (464.23,83.3) .. controls (465.82,83.08) and (467.28,84.19) .. (467.5,85.78) -- cycle ;
\draw  [color={rgb, 255:red, 136; green, 4; blue, 22 }  ,draw opacity=1 ][fill={rgb, 255:red, 136; green, 4; blue, 22 }  ,fill opacity=1 ][line width=1.5]  (430.65,144.28) .. controls (429.49,143.09) and (429.51,141.2) .. (430.69,140.03) .. controls (431.87,138.87) and (433.77,138.89) .. (434.93,140.07) .. controls (436.09,141.25) and (436.08,143.15) .. (434.9,144.31) .. controls (433.72,145.47) and (431.82,145.46) .. (430.65,144.28) -- cycle ;
\draw [color={rgb, 255:red, 136; green, 4; blue, 22 }  ,draw opacity=1 ][line width=1.5]  [dash pattern={on 1.69pt off 2.76pt}]  (442.93,43.38) -- (464.87,85.37) ;
\draw [color={rgb, 255:red, 136; green, 4; blue, 22 }  ,draw opacity=1 ][line width=1.5]    (442.93,43.38) -- (486.81,43.37) ;
\draw [color={rgb, 255:red, 136; green, 4; blue, 22 }  ,draw opacity=1 ][line width=1.5]  [dash pattern={on 1.69pt off 2.76pt}]  (486.81,43.37) -- (464.87,85.37) ;
\draw  [fill={rgb, 255:red, 0; green, 0; blue, 0 }  ,fill opacity=1 ] (445.73,42.61) .. controls (446.15,44.16) and (445.24,45.75) .. (443.7,46.17) .. controls (442.15,46.6) and (440.56,45.69) .. (440.14,44.14) .. controls (439.71,42.6) and (440.62,41) .. (442.17,40.58) .. controls (443.71,40.16) and (445.3,41.07) .. (445.73,42.61) -- cycle ;
\draw  [fill={rgb, 255:red, 0; green, 0; blue, 0 }  ,fill opacity=1 ] (489.61,42.6) .. controls (490.03,44.15) and (489.12,45.74) .. (487.58,46.16) .. controls (486.03,46.59) and (484.44,45.68) .. (484.02,44.13) .. controls (483.59,42.59) and (484.5,40.99) .. (486.05,40.57) .. controls (487.59,40.15) and (489.18,41.06) .. (489.61,42.6) -- cycle ;

\draw [color={rgb, 255:red, 136; green, 4; blue, 22 }  ,draw opacity=1 ][line width=1.5]  [dash pattern={on 1.69pt off 2.76pt}]  (166.6,43.38) -- (188.54,85.37) ;
\draw [color={rgb, 255:red, 136; green, 4; blue, 22 }  ,draw opacity=1 ][line width=1.5]    (166.6,43.38) -- (210.48,43.37) ;
\draw [color={rgb, 255:red, 136; green, 4; blue, 22 }  ,draw opacity=1 ][line width=1.5]  [dash pattern={on 1.69pt off 2.76pt}]  (210.48,43.37) -- (188.54,85.37) ;
\draw  [fill={rgb, 255:red, 0; green, 0; blue, 0 }  ,fill opacity=1 ] (169.39,42.61) .. controls (169.82,44.16) and (168.91,45.75) .. (167.36,46.17) .. controls (165.82,46.6) and (164.23,45.69) .. (163.8,44.14) .. controls (163.38,42.6) and (164.29,41) .. (165.83,40.58) .. controls (167.38,40.16) and (168.97,41.07) .. (169.39,42.61) -- cycle ;
\draw  [fill={rgb, 255:red, 0; green, 0; blue, 0 }  ,fill opacity=1 ] (213.27,42.6) .. controls (213.7,44.15) and (212.79,45.74) .. (211.24,46.16) .. controls (209.7,46.59) and (208.11,45.68) .. (207.68,44.13) .. controls (207.26,42.59) and (208.17,40.99) .. (209.71,40.57) .. controls (211.26,40.15) and (212.85,41.06) .. (213.27,42.6) -- cycle ;

\draw [color={rgb, 255:red, 136; green, 4; blue, 22 }  ,draw opacity=1 ][line width=1.5]  [dash pattern={on 1.69pt off 2.76pt}]  (218.93,25.71) -- (240.87,67.71) ;
\draw [color={rgb, 255:red, 136; green, 4; blue, 22 }  ,draw opacity=1 ][line width=1.5]    (218.93,25.71) -- (262.81,25.7) ;
\draw [color={rgb, 255:red, 136; green, 4; blue, 22 }  ,draw opacity=1 ][line width=1.5]  [dash pattern={on 1.69pt off 2.76pt}]  (262.81,25.7) -- (240.87,67.71) ;
\draw  [fill={rgb, 255:red, 0; green, 0; blue, 0 }  ,fill opacity=1 ] (221.73,24.94) .. controls (222.15,26.49) and (221.24,28.08) .. (219.7,28.51) .. controls (218.15,28.93) and (216.56,28.02) .. (216.14,26.48) .. controls (215.71,24.93) and (216.62,23.34) .. (218.17,22.91) .. controls (219.71,22.49) and (221.3,23.4) .. (221.73,24.94) -- cycle ;
\draw  [fill={rgb, 255:red, 0; green, 0; blue, 0 }  ,fill opacity=1 ] (265.61,24.93) .. controls (266.03,26.48) and (265.12,28.07) .. (263.58,28.5) .. controls (262.03,28.92) and (260.44,28.01) .. (260.02,26.47) .. controls (259.59,24.92) and (260.5,23.33) .. (262.05,22.9) .. controls (263.59,22.48) and (265.18,23.39) .. (265.61,24.93) -- cycle ;

\draw [color={rgb, 255:red, 136; green, 4; blue, 22 }  ,draw opacity=1 ][line width=1.5]  [dash pattern={on 1.69pt off 2.76pt}]  (268.93,43.04) -- (290.87,85.04) ;
\draw [color={rgb, 255:red, 136; green, 4; blue, 22 }  ,draw opacity=1 ][line width=1.5]    (268.93,43.04) -- (312.81,43.03) ;
\draw [color={rgb, 255:red, 136; green, 4; blue, 22 }  ,draw opacity=1 ][line width=1.5]  [dash pattern={on 1.69pt off 2.76pt}]  (312.81,43.03) -- (290.87,85.04) ;
\draw  [fill={rgb, 255:red, 0; green, 0; blue, 0 }  ,fill opacity=1 ] (271.73,42.28) .. controls (272.15,43.82) and (271.24,45.42) .. (269.7,45.84) .. controls (268.15,46.26) and (266.56,45.35) .. (266.14,43.81) .. controls (265.71,42.26) and (266.62,40.67) .. (268.17,40.25) .. controls (269.71,39.82) and (271.3,40.73) .. (271.73,42.28) -- cycle ;
\draw  [fill={rgb, 255:red, 0; green, 0; blue, 0 }  ,fill opacity=1 ] (315.61,42.27) .. controls (316.03,43.81) and (315.12,45.41) .. (313.58,45.83) .. controls (312.03,46.25) and (310.44,45.34) .. (310.02,43.8) .. controls (309.59,42.25) and (310.5,40.66) .. (312.05,40.24) .. controls (313.59,39.81) and (315.18,40.72) .. (315.61,42.27) -- cycle ;

\draw (164.1,161) node [anchor=north west][inner sep=0.75pt]    {$\omega _{1}$};
\draw (212.5,139) node [anchor=north west][inner sep=0.75pt]    {$\omega _{2}$};
\draw (294,161) node [anchor=north west][inner sep=0.75pt]    {$\omega _{3}$};
\draw (298,121.4) node [anchor=north west][inner sep=0.75pt]    {$V'$};
\draw (217.01,193) node [anchor=north west][inner sep=0.75pt]    {$J=3$};
\draw (415,125) node [anchor=north west][inner sep=0.75pt]    {$\omega _{1}$};
\draw (468,161) node [anchor=north west][inner sep=0.75pt]    {$\omega _{2}$};
\draw (376,161) node [anchor=north west][inner sep=0.75pt]    {$\omega _{3}$};
\draw (478,121.4) node [anchor=north west][inner sep=0.75pt]    {$V'$};
\draw (407.68,193) node [anchor=north west][inner sep=0.75pt]    {$J=2$};

\end{tikzpicture}

\caption{$|\{\omega_1,\omega_2,\omega_3\}\cap V(\mathcal{M})|=0$.}
    \label{fig:cap=0}
\end{minipage}
\end{figure}

By Claim \ref{CLM:structure-edge-between-T1}, for each pair of triangles $x_1y_1z_1, x_2y_2z_2\in \mathcal{T}_1$, we have
\begin{align}\label{EQ:edge-between-triangles-T1}
    e(x_1y_1z_1, x_2y_2z_2)\leq 7. 
\end{align}

For simplicity, let $H'$ be the subgraph of $H$ induced by $V(\mathcal M)\cup V'$ and recall $V'$ is the set of all non-critical vertices of the triangles in $\mathcal{T}_1$. Then $H'$ has $2t_1+2m$ vertices. The crucial point is to demonstrate  that $H'$ is bipartite. 

\begin{claim}\label{CLAIM:K3-free-V'-M}
    The graph $H'$ is triangle-free.
\end{claim}

\begin{proof}
By contradiction, assume that there exists a triangle $\omega_1\omega_2\omega_3$  in $H[V(\mathcal M)\cup V']$. Recall that  $H[V(\mathcal M)]$ is triangle-free by the assumption. It follows that $0\le |\{\omega_1,\omega_2,\omega_3\}\cap V(\mathcal{M})|\le 2$.

{\bf Case 1.} $|\{\omega_1,\omega_2,\omega_3\}\cap V(\mathcal{M})|=0$.

Suppose that $\omega_1\omega_2\omega_3$ intersects with $x_jy_jz_j\in \mathcal{T}_1$ for each $1\le j\le J$. Then $2\le J\le 3$. Without loss of generality, we focus on the case $J = 3$, as the argument for $J = 2$ follows similarly. By Claim \ref{CLM:at-most-two-triangle}, there exist an $\mathcal M$-edge $u_jv_j$ for $j\in [3]$ such that $u_jv_j$ sees $x_jy_jz_j$. Then $\mathcal T\setminus\{ x_1y_1z_1, x_2y_2z_2, x_3y_3z_3\}\cup\{u_1v_1x_1,u_2v_2x_2,u_3v_3x_3,\omega_1\omega_2\omega_3\}$ (see Figure~\ref{fig:cap=0}) is a copy of $(t+2)K_3$ in $H$,  a contradiction.

{\bf Case 2.} $|\{\omega_1,\omega_2,\omega_3\}\cap V(\mathcal{M})|=1$.

Assume that $\{\omega_1,\omega_2,\omega_3\}\cap V(\mathcal{M})=\{\omega_1\}$. Then $\omega_2$ and $\omega_3$ must belong to two distinct triangles in  $\mathcal{T}_1$. To see why, suppose for contradiction that $\omega_2$ and $\omega_3$ lie in the same triangle, say $\omega\omega_2\omega_3 \in \mathcal{T}_1$. Then there is  an $\mathcal M$-edge $uv\in \mathcal{M}\setminus\{\omega_1\}$ sees $ \omega\omega_2\omega_3$. However, $\mathcal T\setminus \{ \omega\omega_2\omega_3\}\cup\{ \omega_1\omega_2\omega_3, \omega uv\}$ (see Figure~\ref{fig:cap=1}) is a copy of $(t+2)K_3$ of $H$.

Suppose that $\omega_2$ and $\omega_3$ come from $ x_2y_2\omega_2$ and $ x_3y_3\omega_3$ in $\mathcal{T}_1$. If there exist two $\mathcal M$-edges $u_2v_2$ and $u_3v_3$ in $\mathcal{M}\setminus\{\omega_1\}$ such that $u_iv_i$ sees $ x_iy_i\omega_i$ for $i=2,3$, then  $\mathcal T=\mathcal T\setminus\{ x_2y_2\omega_2, x_3y_3\omega_3\}\cup\{u_2v_2x_2,u_3v_3x_3,\omega_1\omega_2\omega_3\}$ (see Figure~\ref{fig:cap=1}) is a copy of $(t + 2)K_3$ in $H$, a contradiction. Thus, there are exactly two $\mathcal{M}$-edges, say $uv, \omega_1\omega$, that see both $ x_2y_2\omega_2$ and $ x_3y_3\omega_3$. By \eqref{EQ:bound-edge-triangle-edge-T1-M}, we have $e(x_iy_i\omega_i, \mathcal{M})\le 3(m-2)+2\times 4=4m-(m-2)$ for $i\in\{2,3\}$, and then 
\begin{align}\label{EQ:edge-T1-M-w1-cap-M=1}
    e(\mathcal{T}_1,\mathcal{M})\leq 4mt_1-2(m-2). 
\end{align}

We can update the number of edges in $\mathcal{T}_1$. For each $xyz\in\mathcal{T}_1\setminus\{ x_2y_2\omega_2, x_3y_3\omega_3\}$, we claim that $e(x_i, xyz)\le 2$ for $i\in\{2,3\}$. Suppose to the contrary that $e(x_2, xyz)=3$. By Claim \ref{CLM:at-most-two-triangle}, there exists $u'v'\in\mathcal{M}\setminus\{uv,\omega_1\omega\}$ that sees $ xyz$. Therefore, if $yz$ sees $ x_2y_2\omega_2$, then $yz$ must see $x_2$ of $ x_2y_2w_2$ by Claim~\ref{CLM:structure-edge-between-T1} and we can find a $(t + 2)K_3$ in $H$ (see Figure~\ref{fig:cap=1}):
\begin{align*}
    \mathcal T\setminus\{ xyz, x_2y_2\omega_2, x_3y_3\omega_3\}\cup\{ x_2yz, x_3uv, xu'v',\omega_1\omega_2\omega_3\}, 
\end{align*}
a contradiction. Combining Claim \ref{CLM:structure-edge-between-T1} and $e(x_i, xyz)\le 2$,  we conclude that  $e(xyz, x_iy_iz_i)\le 6$ for $i=2,3$. This together with \eqref{EQ:edge-between-triangles-T1} shows 
\begin{align}\label{EQ:edge-T1-w1-cap-M=1}
    e(\mathcal{T}_1)\leq7\binom{t_1}{2}+3t_1-2(t_1-2). 
\end{align}

\begin{figure}[H]
    \centering

\tikzset{every picture/.style={line width=0.75pt}} 

\begin{tikzpicture}[x=0.75pt,y=0.75pt,yscale=-0.89,xscale=0.89]

\draw [color={rgb, 255:red, 136; green, 4; blue, 22 }  ,draw opacity=1 ][line width=1.5]  [dash pattern={on 1.69pt off 2.76pt}]  (218.51,110.44) -- (276.31,110.77) ;
\draw [color={rgb, 255:red, 136; green, 4; blue, 22 }  ,draw opacity=1 ][line width=1.5]  [dash pattern={on 1.69pt off 2.76pt}]  (518.51,110.44) -- (564.31,110.77) ;
\draw [color={rgb, 255:red, 136; green, 4; blue, 22 }  ,draw opacity=1 ][line width=1.5]  [dash pattern={on 1.69pt off 2.76pt}]  (486.79,84.17) -- (457.84,29.14) ;
\draw [color={rgb, 255:red, 136; green, 4; blue, 22 }  ,draw opacity=1 ][line width=1.5]  [dash pattern={on 1.69pt off 2.76pt}]  (457.51,110.44) -- (486.79,84.17) ;
\draw [color={rgb, 255:red, 136; green, 4; blue, 22 }  ,draw opacity=1 ][line width=1.5]  [dash pattern={on 1.69pt off 2.76pt}]  (564.31,110.77) -- (540.73,46.02) ;
\draw [color={rgb, 255:red, 136; green, 4; blue, 22 }  ,draw opacity=1 ][line width=1.5]  [dash pattern={on 1.69pt off 2.76pt}]  (307.99,84.5) -- (337.73,109.18) ;
\draw [color={rgb, 255:red, 0; green, 0; blue, 0 }  ,draw opacity=1 ][line width=1.5]    (115.83,32.68) -- (115.39,82.68) ;
\draw [color={rgb, 255:red, 136; green, 4; blue, 22 }  ,draw opacity=1 ][line width=1.5]  [dash pattern={on 1.69pt off 2.76pt}]  (22.72,60.17) -- (52.46,84.84) ;
\draw [color={rgb, 255:red, 136; green, 4; blue, 22 }  ,draw opacity=1 ][line width=1.5]    (22.72,60.17) -- (22.72,110.17) ;
\draw [color={rgb, 255:red, 136; green, 4; blue, 22 }  ,draw opacity=1 ][line width=1.5]  [dash pattern={on 1.69pt off 2.76pt}]  (22.72,110.17) -- (52.46,84.84) ;
\draw  [fill={rgb, 255:red, 0; green, 0; blue, 0 }  ,fill opacity=1 ] (20.99,62.49) .. controls (22.28,63.45) and (24.1,63.18) .. (25.05,61.89) .. controls (26.01,60.61) and (25.74,58.79) .. (24.45,57.84) .. controls (23.16,56.88) and (21.35,57.15) .. (20.39,58.44) .. controls (19.43,59.72) and (19.7,61.54) .. (20.99,62.49) -- cycle ;
\draw  [fill={rgb, 255:red, 0; green, 0; blue, 0 }  ,fill opacity=1 ] (20.99,112.49) .. controls (22.28,113.45) and (24.1,113.18) .. (25.05,111.89) .. controls (26.01,110.61) and (25.74,108.79) .. (24.45,107.84) .. controls (23.16,106.88) and (21.35,107.15) .. (20.39,108.44) .. controls (19.43,109.72) and (19.7,111.54) .. (20.99,112.49) -- cycle ;
\draw [color={rgb, 255:red, 136; green, 4; blue, 22 }  ,draw opacity=1 ][line width=1.5]  [dash pattern={on 1.69pt off 2.76pt}]  (84.18,111.11) -- (115.39,82.68) ;
\draw [color={rgb, 255:red, 136; green, 4; blue, 22 }  ,draw opacity=1 ][line width=1.5]  [dash pattern={on 1.69pt off 2.76pt}]  (84.29,29.84) -- (115.39,82.68) ;
\draw [color={rgb, 255:red, 0; green, 0; blue, 0 }  ,draw opacity=1 ][line width=1.5]    (84.18,111.11) -- (52.46,85.84) ;
\draw [color={rgb, 255:red, 0; green, 0; blue, 0 }  ,draw opacity=1 ][line width=1.5]    (84.51,29.8) -- (52.46,85.84) ;
\draw [color={rgb, 255:red, 136; green, 4; blue, 22 }  ,draw opacity=1 ][line width=1.5]    (84.18,111.11) -- (84.51,29.8) ;
\draw  [fill={rgb, 255:red, 0; green, 0; blue, 0 }  ,fill opacity=1 ] (87.16,29.45) .. controls (87.38,31.03) and (86.27,32.5) .. (84.68,32.71) .. controls (83.1,32.93) and (81.64,31.82) .. (81.42,30.23) .. controls (81.2,28.65) and (82.31,27.18) .. (83.9,26.97) .. controls (85.48,26.75) and (86.95,27.86) .. (87.16,29.45) -- cycle ;
\draw  [fill={rgb, 255:red, 0; green, 0; blue, 0 }  ,fill opacity=1 ] (55.33,84.45) .. controls (55.55,86.03) and (54.44,87.49) .. (52.85,87.71) .. controls (51.27,87.93) and (49.81,86.82) .. (49.59,85.23) .. controls (49.37,83.64) and (50.48,82.18) .. (52.07,81.97) .. controls (53.65,81.75) and (55.11,82.86) .. (55.33,84.45) -- cycle ;
\draw  [fill={rgb, 255:red, 0; green, 0; blue, 0 }  ,fill opacity=1 ] (112.99,32.29) .. controls (112.77,33.88) and (113.87,35.34) .. (115.44,35.56) .. controls (117.01,35.77) and (118.46,34.66) .. (118.68,33.08) .. controls (118.89,31.49) and (117.79,30.03) .. (116.22,29.81) .. controls (114.65,29.59) and (113.2,30.71) .. (112.99,32.29) -- cycle ;
\draw  [fill={rgb, 255:red, 0; green, 0; blue, 0 }  ,fill opacity=1 ] (117.12,85.01) .. controls (115.83,85.97) and (114.02,85.7) .. (113.07,84.41) .. controls (112.11,83.13) and (112.38,81.31) .. (113.67,80.36) .. controls (114.95,79.4) and (116.77,79.67) .. (117.72,80.96) .. controls (118.68,82.24) and (118.41,84.06) .. (117.12,85.01) -- cycle ;
\draw  [fill={rgb, 255:red, 0; green, 0; blue, 0 }  ,fill opacity=1 ] (85.9,113.43) .. controls (84.62,114.39) and (82.8,114.12) .. (81.85,112.83) .. controls (80.89,111.55) and (81.16,109.73) .. (82.45,108.78) .. controls (83.74,107.82) and (85.55,108.09) .. (86.51,109.38) .. controls (87.46,110.66) and (87.19,112.48) .. (85.9,113.43) -- cycle ;
\draw [color={rgb, 255:red, 136; green, 4; blue, 22 }  ,draw opacity=1 ][line width=1.5]  [dash pattern={on 1.69pt off 2.76pt}]  (307.99,84.5) -- (338.06,59.17) ;
\draw [color={rgb, 255:red, 0; green, 0; blue, 0 }  ,draw opacity=1 ][line width=1.5]    (248.17,32.02) -- (247.73,82.02) ;
\draw [color={rgb, 255:red, 136; green, 4; blue, 22 }  ,draw opacity=1 ][line width=1.5]  [dash pattern={on 1.69pt off 2.76pt}]  (157.06,59.5) -- (186.79,84.17) ;
\draw [color={rgb, 255:red, 136; green, 4; blue, 22 }  ,draw opacity=1 ][line width=1.5]    (157.06,59.5) -- (157.06,109.5) ;
\draw [color={rgb, 255:red, 136; green, 4; blue, 22 }  ,draw opacity=1 ][line width=1.5]  [dash pattern={on 1.69pt off 2.76pt}]  (157.06,109.5) -- (186.79,84.17) ;
\draw  [fill={rgb, 255:red, 0; green, 0; blue, 0 }  ,fill opacity=1 ] (155.32,61.83) .. controls (156.61,62.78) and (158.43,62.51) .. (159.39,61.23) .. controls (160.34,59.94) and (160.07,58.13) .. (158.79,57.17) .. controls (157.5,56.22) and (155.68,56.49) .. (154.72,57.77) .. controls (153.77,59.06) and (154.04,60.87) .. (155.32,61.83) -- cycle ;
\draw  [fill={rgb, 255:red, 0; green, 0; blue, 0 }  ,fill opacity=1 ] (155.32,111.83) .. controls (156.61,112.78) and (158.43,112.51) .. (159.39,111.23) .. controls (160.34,109.94) and (160.07,108.13) .. (158.79,107.17) .. controls (157.5,106.22) and (155.68,106.49) .. (154.72,107.77) .. controls (153.77,109.06) and (154.04,110.87) .. (155.32,111.83) -- cycle ;
\draw [color={rgb, 255:red, 136; green, 4; blue, 22 }  ,draw opacity=1 ][line width=1.5]  [dash pattern={on 1.69pt off 2.76pt}]  (218.51,110.44) -- (247.73,82.02) ;
\draw [color={rgb, 255:red, 136; green, 4; blue, 22 }  ,draw opacity=1 ][line width=1.5]  [dash pattern={on 1.69pt off 2.76pt}]  (276.31,110.77) -- (247.73,82.02) ;
\draw [color={rgb, 255:red, 0; green, 0; blue, 0 }  ,draw opacity=1 ][line width=1.5]    (218.51,110.44) -- (186.79,85.17) ;
\draw [color={rgb, 255:red, 0; green, 0; blue, 0 }  ,draw opacity=1 ][line width=1.5]    (218.84,29.14) -- (186.79,85.17) ;
\draw [color={rgb, 255:red, 0; green, 0; blue, 0 }  ,draw opacity=1 ][line width=1.5]    (218.51,110.44) -- (218.84,29.14) ;
\draw  [fill={rgb, 255:red, 0; green, 0; blue, 0 }  ,fill opacity=1 ] (221.5,28.78) .. controls (221.71,30.37) and (220.6,31.83) .. (219.02,32.05) .. controls (217.43,32.26) and (215.97,31.15) .. (215.75,29.57) .. controls (215.54,27.98) and (216.65,26.52) .. (218.23,26.3) .. controls (219.82,26.08) and (221.28,27.19) .. (221.5,28.78) -- cycle ;
\draw  [fill={rgb, 255:red, 0; green, 0; blue, 0 }  ,fill opacity=1 ] (189.66,83.78) .. controls (189.88,85.37) and (188.77,86.83) .. (187.19,87.04) .. controls (185.6,87.26) and (184.14,86.15) .. (183.92,84.56) .. controls (183.7,82.98) and (184.81,81.52) .. (186.4,81.3) .. controls (187.99,81.08) and (189.45,82.19) .. (189.66,83.78) -- cycle ;
\draw  [fill={rgb, 255:red, 0; green, 0; blue, 0 }  ,fill opacity=1 ] (245.32,31.63) .. controls (245.11,33.21) and (246.21,34.67) .. (247.78,34.89) .. controls (249.35,35.11) and (250.8,34) .. (251.01,32.41) .. controls (251.23,30.82) and (250.13,29.36) .. (248.56,29.15) .. controls (246.99,28.93) and (245.54,30.04) .. (245.32,31.63) -- cycle ;
\draw  [fill={rgb, 255:red, 0; green, 0; blue, 0 }  ,fill opacity=1 ] (249.45,84.34) .. controls (248.17,85.3) and (246.35,85.03) .. (245.4,83.75) .. controls (244.44,82.46) and (244.71,80.65) .. (246,79.69) .. controls (247.29,78.74) and (249.1,79) .. (250.06,80.29) .. controls (251.01,81.57) and (250.74,83.39) .. (249.45,84.34) -- cycle ;
\draw  [fill={rgb, 255:red, 0; green, 0; blue, 0 }  ,fill opacity=1 ] (220.24,112.77) .. controls (218.95,113.72) and (217.14,113.45) .. (216.18,112.17) .. controls (215.23,110.88) and (215.5,109.07) .. (216.78,108.11) .. controls (218.07,107.16) and (219.88,107.43) .. (220.84,108.71) .. controls (221.79,110) and (221.52,111.81) .. (220.24,112.77) -- cycle ;
\draw [color={rgb, 255:red, 0; green, 0; blue, 0 }  ,draw opacity=1 ][line width=1.5]    (276.31,110.77) -- (307.99,85.5) ;
\draw [color={rgb, 255:red, 0; green, 0; blue, 0 }  ,draw opacity=1 ][line width=1.5]    (275.98,29.47) -- (307.99,85.5) ;
\draw [color={rgb, 255:red, 0; green, 0; blue, 0 }  ,draw opacity=1 ][line width=1.5]    (276.31,110.77) -- (275.98,29.47) ;
\draw  [fill={rgb, 255:red, 0; green, 0; blue, 0 }  ,fill opacity=1 ] (305.13,84.11) .. controls (304.91,85.7) and (306.02,87.16) .. (307.6,87.38) .. controls (309.18,87.6) and (310.64,86.48) .. (310.86,84.9) .. controls (311.08,83.31) and (309.97,81.85) .. (308.39,81.63) .. controls (306.8,81.41) and (305.34,82.53) .. (305.13,84.11) -- cycle ;
\draw  [fill={rgb, 255:red, 0; green, 0; blue, 0 }  ,fill opacity=1 ] (274.58,113.1) .. controls (275.87,114.05) and (277.68,113.79) .. (278.63,112.5) .. controls (279.59,111.22) and (279.32,109.4) .. (278.03,108.45) .. controls (276.75,107.49) and (274.93,107.76) .. (273.98,109.04) .. controls (273.03,110.33) and (273.3,112.14) .. (274.58,113.1) -- cycle ;
\draw  [fill={rgb, 255:red, 0; green, 0; blue, 0 }  ,fill opacity=1 ] (278.85,29.08) .. controls (279.06,30.66) and (277.95,32.13) .. (276.37,32.34) .. controls (274.78,32.56) and (273.32,31.45) .. (273.1,29.86) .. controls (272.89,28.28) and (274,26.81) .. (275.58,26.6) .. controls (277.17,26.38) and (278.63,27.49) .. (278.85,29.08) -- cycle ;
\draw [color={rgb, 255:red, 136; green, 4; blue, 22 }  ,draw opacity=1 ][line width=1.5]    (338.06,59.17) -- (338.06,109.17) ;
\draw  [fill={rgb, 255:red, 0; green, 0; blue, 0 }  ,fill opacity=1 ] (336.32,61.49) .. controls (337.61,62.45) and (339.43,62.18) .. (340.39,60.89) .. controls (341.34,59.61) and (341.07,57.79) .. (339.79,56.84) .. controls (338.5,55.88) and (336.68,56.15) .. (335.72,57.44) .. controls (334.77,58.72) and (335.04,60.54) .. (336.32,61.49) -- cycle ;
\draw  [fill={rgb, 255:red, 0; green, 0; blue, 0 }  ,fill opacity=1 ] (336.32,111.49) .. controls (337.61,112.45) and (339.43,112.18) .. (340.39,110.89) .. controls (341.34,109.61) and (341.07,107.79) .. (339.79,106.84) .. controls (338.5,105.88) and (336.68,106.15) .. (335.72,107.44) .. controls (334.77,108.72) and (335.04,110.54) .. (336.32,111.49) -- cycle ;
\draw [color={rgb, 255:red, 136; green, 4; blue, 22 }  ,draw opacity=1 ][line width=1.5]  [dash pattern={on 1.69pt off 2.76pt}]  (595.99,84.5) -- (625.73,109.18) ;
\draw [color={rgb, 255:red, 136; green, 4; blue, 22 }  ,draw opacity=1 ][line width=1.5]  [dash pattern={on 1.69pt off 2.76pt}]  (595.99,84.5) -- (626.06,59.17) ;
\draw [color={rgb, 255:red, 136; green, 4; blue, 22 }  ,draw opacity=1 ][line width=1.5]  [dash pattern={on 1.69pt off 2.76pt}]  (518.51,110.44) -- (540.73,46.02) ;
\draw [color={rgb, 255:red, 0; green, 0; blue, 0 }  ,draw opacity=1 ][line width=1.5]    (518.51,110.44) -- (486.79,85.17) ;
\draw [color={rgb, 255:red, 0; green, 0; blue, 0 }  ,draw opacity=1 ][line width=1.5]    (518.84,29.14) -- (486.79,85.17) ;
\draw [color={rgb, 255:red, 0; green, 0; blue, 0 }  ,draw opacity=1 ][line width=1.5]    (518.51,110.44) -- (518.84,29.14) ;
\draw  [fill={rgb, 255:red, 0; green, 0; blue, 0 }  ,fill opacity=1 ] (521.5,28.78) .. controls (521.71,30.37) and (520.6,31.83) .. (519.02,32.05) .. controls (517.43,32.26) and (515.97,31.15) .. (515.75,29.57) .. controls (515.54,27.98) and (516.65,26.52) .. (518.23,26.3) .. controls (519.82,26.08) and (521.28,27.19) .. (521.5,28.78) -- cycle ;
\draw  [fill={rgb, 255:red, 0; green, 0; blue, 0 }  ,fill opacity=1 ] (489.66,83.78) .. controls (489.88,85.37) and (488.77,86.83) .. (487.19,87.04) .. controls (485.6,87.26) and (484.14,86.15) .. (483.92,84.56) .. controls (483.7,82.98) and (484.81,81.52) .. (486.4,81.3) .. controls (487.99,81.08) and (489.45,82.19) .. (489.66,83.78) -- cycle ;
\draw  [fill={rgb, 255:red, 0; green, 0; blue, 0 }  ,fill opacity=1 ] (542.45,48.34) .. controls (541.17,49.3) and (539.35,49.03) .. (538.4,47.75) .. controls (537.44,46.46) and (537.71,44.65) .. (539,43.69) .. controls (540.29,42.74) and (542.1,43) .. (543.06,44.29) .. controls (544.01,45.57) and (543.74,47.39) .. (542.45,48.34) -- cycle ;
\draw  [fill={rgb, 255:red, 0; green, 0; blue, 0 }  ,fill opacity=1 ] (520.24,112.77) .. controls (518.95,113.72) and (517.14,113.45) .. (516.18,112.17) .. controls (515.23,110.88) and (515.5,109.07) .. (516.78,108.11) .. controls (518.07,107.16) and (519.88,107.43) .. (520.84,108.71) .. controls (521.79,110) and (521.52,111.81) .. (520.24,112.77) -- cycle ;
\draw [color={rgb, 255:red, 0; green, 0; blue, 0 }  ,draw opacity=1 ][line width=1.5]    (564.31,110.77) -- (595.99,85.5) ;
\draw [color={rgb, 255:red, 0; green, 0; blue, 0 }  ,draw opacity=1 ][line width=1.5]    (563.98,29.47) -- (595.99,85.5) ;
\draw [color={rgb, 255:red, 0; green, 0; blue, 0 }  ,draw opacity=1 ][line width=1.5]    (564.31,110.77) -- (563.98,29.47) ;
\draw  [fill={rgb, 255:red, 0; green, 0; blue, 0 }  ,fill opacity=1 ] (593.13,84.11) .. controls (592.91,85.7) and (594.02,87.16) .. (595.6,87.38) .. controls (597.18,87.6) and (598.64,86.48) .. (598.86,84.9) .. controls (599.08,83.31) and (597.97,81.85) .. (596.39,81.63) .. controls (594.8,81.41) and (593.34,82.53) .. (593.13,84.11) -- cycle ;
\draw  [fill={rgb, 255:red, 0; green, 0; blue, 0 }  ,fill opacity=1 ] (562.58,113.1) .. controls (563.87,114.05) and (565.68,113.79) .. (566.63,112.5) .. controls (567.59,111.22) and (567.32,109.4) .. (566.03,108.45) .. controls (564.75,107.49) and (562.93,107.76) .. (561.98,109.04) .. controls (561.03,110.33) and (561.3,112.14) .. (562.58,113.1) -- cycle ;
\draw  [fill={rgb, 255:red, 0; green, 0; blue, 0 }  ,fill opacity=1 ] (566.85,29.08) .. controls (567.06,30.66) and (565.95,32.13) .. (564.37,32.34) .. controls (562.78,32.56) and (561.32,31.45) .. (561.1,29.86) .. controls (560.89,28.28) and (562,26.81) .. (563.58,26.6) .. controls (565.17,26.38) and (566.63,27.49) .. (566.85,29.08) -- cycle ;
\draw [color={rgb, 255:red, 136; green, 4; blue, 22 }  ,draw opacity=1 ][line width=1.5]    (626.06,59.17) -- (626.06,109.17) ;
\draw  [fill={rgb, 255:red, 0; green, 0; blue, 0 }  ,fill opacity=1 ] (624.32,61.49) .. controls (625.61,62.45) and (627.43,62.18) .. (628.39,60.89) .. controls (629.34,59.61) and (629.07,57.79) .. (627.79,56.84) .. controls (626.5,55.88) and (624.68,56.15) .. (623.72,57.44) .. controls (622.77,58.72) and (623.04,60.54) .. (624.32,61.49) -- cycle ;
\draw  [fill={rgb, 255:red, 0; green, 0; blue, 0 }  ,fill opacity=1 ] (624.32,111.49) .. controls (625.61,112.45) and (627.43,112.18) .. (628.39,110.89) .. controls (629.34,109.61) and (629.07,107.79) .. (627.79,106.84) .. controls (626.5,105.88) and (624.68,106.15) .. (623.72,107.44) .. controls (622.77,108.72) and (623.04,110.54) .. (624.32,111.49) -- cycle ;
\draw [color={rgb, 255:red, 136; green, 4; blue, 22 }  ,draw opacity=1 ][line width=1.5]  [dash pattern={on 1.69pt off 2.76pt}]  (396.06,59.5) -- (425.79,84.17) ;
\draw [color={rgb, 255:red, 136; green, 4; blue, 22 }  ,draw opacity=1 ][line width=1.5]    (396.06,59.5) -- (396.06,109.5) ;
\draw [color={rgb, 255:red, 136; green, 4; blue, 22 }  ,draw opacity=1 ][line width=1.5]  [dash pattern={on 1.69pt off 2.76pt}]  (396.06,109.5) -- (425.79,84.17) ;
\draw  [fill={rgb, 255:red, 0; green, 0; blue, 0 }  ,fill opacity=1 ] (394.32,61.83) .. controls (395.61,62.78) and (397.43,62.51) .. (398.39,61.23) .. controls (399.34,59.94) and (399.07,58.13) .. (397.79,57.17) .. controls (396.5,56.22) and (394.68,56.49) .. (393.72,57.77) .. controls (392.77,59.06) and (393.04,60.87) .. (394.32,61.83) -- cycle ;
\draw  [fill={rgb, 255:red, 0; green, 0; blue, 0 }  ,fill opacity=1 ] (394.32,111.83) .. controls (395.61,112.78) and (397.43,112.51) .. (398.39,111.23) .. controls (399.34,109.94) and (399.07,108.13) .. (397.79,107.17) .. controls (396.5,106.22) and (394.68,106.49) .. (393.72,107.77) .. controls (392.77,109.06) and (393.04,110.87) .. (394.32,111.83) -- cycle ;
\draw [color={rgb, 255:red, 0; green, 0; blue, 0 }  ,draw opacity=1 ][line width=1.5]    (457.51,110.44) -- (425.79,85.17) ;
\draw [color={rgb, 255:red, 0; green, 0; blue, 0 }  ,draw opacity=1 ][line width=1.5]    (457.84,29.14) -- (425.79,85.17) ;
\draw [color={rgb, 255:red, 136; green, 4; blue, 22 }  ,draw opacity=1 ][line width=1.5]    (457.51,110.44) -- (457.84,29.14) ;
\draw  [fill={rgb, 255:red, 0; green, 0; blue, 0 }  ,fill opacity=1 ] (460.5,28.78) .. controls (460.71,30.37) and (459.6,31.83) .. (458.02,32.05) .. controls (456.43,32.26) and (454.97,31.15) .. (454.75,29.57) .. controls (454.54,27.98) and (455.65,26.52) .. (457.23,26.3) .. controls (458.82,26.08) and (460.28,27.19) .. (460.5,28.78) -- cycle ;
\draw  [fill={rgb, 255:red, 0; green, 0; blue, 0 }  ,fill opacity=1 ] (428.66,83.78) .. controls (428.88,85.37) and (427.77,86.83) .. (426.19,87.04) .. controls (424.6,87.26) and (423.14,86.15) .. (422.92,84.56) .. controls (422.7,82.98) and (423.81,81.52) .. (425.4,81.3) .. controls (426.99,81.08) and (428.45,82.19) .. (428.66,83.78) -- cycle ;
\draw  [fill={rgb, 255:red, 0; green, 0; blue, 0 }  ,fill opacity=1 ] (459.24,112.77) .. controls (457.95,113.72) and (456.14,113.45) .. (455.18,112.17) .. controls (454.23,110.88) and (454.5,109.07) .. (455.78,108.11) .. controls (457.07,107.16) and (458.88,107.43) .. (459.84,108.71) .. controls (460.79,110) and (460.52,111.81) .. (459.24,112.77) -- cycle ;

\draw (8,46.4) node [anchor=north west][inner sep=0.75pt]    {$u$};
\draw (7.43,103.4) node [anchor=north west][inner sep=0.75pt]    {$v$};
\draw (74.77,115) node [anchor=north west][inner sep=0.75pt]    {$\omega _{2}$};
\draw (111.33,86.4) node [anchor=north west][inner sep=0.75pt]    {$\omega _{1}$};
\draw (61,17) node [anchor=north west][inner sep=0.75pt]    {$\omega _{3}$};
\draw (46.1,92) node [anchor=north west][inner sep=0.75pt]    {$\omega $};
\draw (160,48) node [anchor=north west][inner sep=0.75pt]    {$u_{2}$};
\draw (160,109) node [anchor=north west][inner sep=0.75pt]    {$v_{2}$};
\draw (210.1,115) node [anchor=north west][inner sep=0.75pt]    {$\omega _{2}$};
\draw (532.33,31) node [anchor=north west][inner sep=0.75pt]    {$\omega _{1}$};
\draw (197,17) node [anchor=north west][inner sep=0.75pt]    {$y_{2}$};
\draw (178,93) node [anchor=north west][inner sep=0.75pt]    {$x_{2}$};
\draw (420.1,93) node [anchor=north west][inner sep=0.75pt]    {$x$};
\draw (508.1,115) node [anchor=north west][inner sep=0.75pt]    {$\omega _{2}$};
\draw (273,115) node [anchor=north west][inner sep=0.75pt]    {$\omega _{3}$};
\draw (587,93) node [anchor=north west][inner sep=0.75pt]    {$x_{3}$};
\draw (380.33,45) node [anchor=north west][inner sep=0.75pt]    {$u'$};
\draw (376.77,102.73) node [anchor=north west][inner sep=0.75pt]    {$v'$};
\draw (496.1,17) node [anchor=north west][inner sep=0.75pt]    {$y_{2}$};
\draw (556.1,115) node [anchor=north west][inner sep=0.75pt]    {$\omega _{3}$};
\draw (443.43,17) node [anchor=north west][inner sep=0.75pt]    {$y$};
\draw (249.95,69) node [anchor=north west][inner sep=0.75pt]    {$\omega _{1}$};
\draw (300,63) node [anchor=north west][inner sep=0.75pt]    {$x_{3}$};
\draw (340,46.73) node [anchor=north west][inner sep=0.75pt]    {$u_{3}$};
\draw (340,109) node [anchor=north west][inner sep=0.75pt]    {$v_{3}$};
\draw (442.65,115) node [anchor=north west][inner sep=0.75pt]    {$z$};
\draw (478,93) node [anchor=north west][inner sep=0.75pt]    {$x_{2}$};
\draw (568,17) node [anchor=north west][inner sep=0.75pt]    {$y_{3}$};
\draw (170,137) node [anchor=north west][inner sep=0.75pt]   [align=left] {Illustration of \eqref{EQ:edge-T1-M-w1-cap-M=1}};
\draw (450,137) node [anchor=north west][inner sep=0.75pt]   [align=left] {Illustration of \eqref{EQ:edge-T1-w1-cap-M=1}};
\draw (630,46.73) node [anchor=north west][inner sep=0.75pt]    {$u$};
\draw (630,103) node [anchor=north west][inner sep=0.75pt]    {$v$};

\end{tikzpicture}

    \caption{the case $|\{\omega_1,\omega_2,\omega_3\}\cap V(\mathcal{M})|=1$.}
    \label{fig:cap=1}
\end{figure}

Using \eqref{EQ:edge-T1-M-w1-cap-M=1} and \eqref{EQ:edge-T1-w1-cap-M=1}, we update the terms involving  $e(\mathcal{T}_1)$ and $e(\mathcal{T}_1,\mathcal{M})$ in $h$, to obtain 
\begin{align*}
    |H |&\le f+im+m^2+(3+3m)t_4+(2+i)t_4+\binom{3t_4}{2}-2(m-2)-2(t_1-2)\\&\le e_1(n,t)+\frac{1}{2} ( n-t)-2m-2t_1+O(1)<N,
\end{align*}
a contradiction. The second inequality follows from Fact \ref{fact1}.

{\bf Case 3.} $|\{\omega_1,\omega_2,\omega_3\}\cap V(\mathcal{M})|=2$. 

Assume that $\{\omega_1,\omega_2,\omega_3\}\cap V(\mathcal{M})=\{\omega_1,\omega_2\}$ and $ xy\omega_3$ is an element in $\mathcal T_1$. Let  $\omega_1v_1,\omega_2v_2$ be two edges in $\mathcal M$. Frist, we bound $e(\mathcal M)$. It is straightforward to verify that among all matching edges in $\mathcal M$, only $\omega_1v_1$ and $\omega_2v_2$  see  $ xy\omega_3$, and $\omega_1\omega_2$ is the only edge connecting $\omega_1v_1$ and $\omega_2v_2$. Otherwise, if there are two edges between $\omega_1v_1$ and $\omega_2v_2$, they must be $\omega_1\omega_2$ and $v_1v_2$. Then $\mathcal{T}\setminus\{xy\omega_3\}\cup\{\omega_1\omega_2\omega_3, xv_1v_2\}$ (see Figure~\ref{FIG:cap=2}) forms a $(t+2)K_3$ in $H$. On the other hand, for $e_1, e_2\in \mathcal M$, it follows from that $H[\mathcal M]$ is triangle-free that $e(e_1,e_2)\le 2$. Thus, 
\begin{align}\label{edges-M-m2-1-w1-cap-M=2}
    e(\mathcal M)\le 2\binom{m}{2}+m-1=m^2-1. 
\end{align}

We then consider the number of edges between $\mathcal M$ and $xy\omega_3$. For other edges $uv\in\mathcal M\setminus\{\omega_1v_1,\omega_2v_2\}$,  since $uv$ does not see $xy\omega_3$,  we have $e(uv, xy\omega_3)\le 3$. We still have $e(\omega_iv_i, xy\omega_3)\le 3$, otherwise, $v_iy\in H$ and then $\mathcal T\setminus\{xy\omega_3\}\cup\{xyv_i,\omega_1\omega_2\omega_3\}$ (see Figure~\ref{FIG:cap=2}) is a copy of $(t + 2)K_3$ in $H$. Thus, $e(xy\omega_3, \mathcal M)\le 3t_1$, which together with \eqref{EQ:bound-edge-triangle-edge-T1-M} shows 
\begin{align}\label{edges-T1-M-w1-cap-M=2}
    e(\mathcal T_1,\mathcal M)=4mt_1-m.
\end{align}

At last, we bound $e(\mathcal T_1)$. For any $ x'y'z'\in \mathcal T_1\setminus\{ xy\omega_3\}$ that can be seen by an edge $uv\in\mathcal M\setminus\{\omega_1v_1,\omega_2v_2\}$, we have $e(x'y'z', xy\omega_3)\le 6$, otherwise, $xy',xz'\in E(H)$ and so 
\begin{align*}
    \left(\mathcal T\setminus\{ x'y'z', xy\omega_3\} \right) \cup \{ x'uv,\omega_1\omega_2\omega_3, xy'z' \}
\end{align*}
is a copy of $(t+2)K_3$ in $H$ (see Figure~\ref{FIG:cap=2}). By Claim \ref{CLM:at-most-two-triangle}, the number of choice $x'y'z'$ is at least $t_1-2$. This along with \eqref{EQ:edge-between-triangles-T1} gives 
\begin{align}\label{edges-T1-w1-cap-M=2}
e(\mathcal T_1)\le 7\binom{t_1}{2}+3t_1-(t_1-2)
\end{align}

\begin{figure}[H]
    \centering
    
\tikzset{every picture/.style={line width=0.75pt}} 

\begin{tikzpicture}[x=0.75pt,y=0.75pt,yscale=-0.98,xscale=0.98]

\draw [color={rgb, 255:red, 136; green, 4; blue, 22 }  ,draw opacity=1 ][line width=1.5]  [dash pattern={on 1.69pt off 2.76pt}]  (459.51,110.44) -- (501.73,75.58) ;
\draw [color={rgb, 255:red, 136; green, 4; blue, 22 }  ,draw opacity=1 ][line width=1.5]  [dash pattern={on 1.69pt off 2.76pt}]  (427.79,84.17) -- (398.84,29.14) ;
\draw [color={rgb, 255:red, 136; green, 4; blue, 22 }  ,draw opacity=1 ][line width=1.5]  [dash pattern={on 1.69pt off 2.76pt}]  (398.51,110.44) -- (427.79,84.17) ;
\draw [color={rgb, 255:red, 136; green, 4; blue, 22 }  ,draw opacity=1 ][line width=1.5]  [dash pattern={on 1.69pt off 2.76pt}]  (501.73,75.58) -- (490.73,38.58) ;
\draw [color={rgb, 255:red, 0; green, 0; blue, 0 }  ,draw opacity=1 ][line width=1.5]    (49.72,60.17) -- (49.72,110.17) ;
\draw [color={rgb, 255:red, 0; green, 0; blue, 0 }  ,draw opacity=1 ][line width=1.5]  [dash pattern={on 1.69pt off 2.76pt}]  (81.72,60.17) -- (111.46,84.84) ;
\draw [color={rgb, 255:red, 0; green, 0; blue, 0 }  ,draw opacity=1 ][line width=1.5]    (81.72,60.17) -- (81.72,110.17) ;
\draw [color={rgb, 255:red, 136; green, 4; blue, 22 }  ,draw opacity=1 ][line width=1.5]  [dash pattern={on 1.69pt off 2.76pt}]  (81.72,110.17) -- (111.46,84.84) ;
\draw [color={rgb, 255:red, 136; green, 4; blue, 22 }  ,draw opacity=1 ][line width=1.5]  [dash pattern={on 1.69pt off 2.76pt}]  (50.01,110.77) -- (81.72,110.17) ;
\draw [color={rgb, 255:red, 136; green, 4; blue, 22 }  ,draw opacity=1 ][line width=1.5]  [dash pattern={on 1.69pt off 2.76pt}]  (49.47,59.89) -- (81.72,60.17) ;
\draw [color={rgb, 255:red, 0; green, 0; blue, 0 }  ,draw opacity=1 ][line width=1.5]    (143.18,111.11) -- (111.46,85.84) ;
\draw [color={rgb, 255:red, 0; green, 0; blue, 0 }  ,draw opacity=1 ][line width=1.5]    (143.51,29.8) -- (111.46,85.84) ;
\draw [color={rgb, 255:red, 0; green, 0; blue, 0 }  ,draw opacity=1 ][line width=1.5]    (143.18,111.11) -- (143.51,29.8) ;
\draw  [fill={rgb, 255:red, 0; green, 0; blue, 0 }  ,fill opacity=1 ] (144.9,113.43) .. controls (143.62,114.39) and (141.8,114.12) .. (140.85,112.83) .. controls (139.89,111.55) and (140.16,109.73) .. (141.45,108.78) .. controls (142.74,107.82) and (144.55,108.09) .. (145.51,109.38) .. controls (146.46,110.66) and (146.19,112.48) .. (144.9,113.43) -- cycle ;
\draw [color={rgb, 255:red, 136; green, 4; blue, 22 }  ,draw opacity=1 ][line width=1.5]  [dash pattern={on 1.69pt off 2.76pt}]  (459.51,110.44) -- (490.73,38.58) ;
\draw [color={rgb, 255:red, 0; green, 0; blue, 0 }  ,draw opacity=1 ][line width=1.5]    (459.51,110.44) -- (427.79,85.17) ;
\draw [color={rgb, 255:red, 0; green, 0; blue, 0 }  ,draw opacity=1 ][line width=1.5]    (459.84,29.14) -- (427.79,85.17) ;
\draw [color={rgb, 255:red, 0; green, 0; blue, 0 }  ,draw opacity=1 ][line width=1.5]    (459.51,110.44) -- (459.84,29.14) ;
\draw  [fill={rgb, 255:red, 0; green, 0; blue, 0 }  ,fill opacity=1 ] (462.5,28.78) .. controls (462.71,30.37) and (461.6,31.83) .. (460.02,32.05) .. controls (458.43,32.26) and (456.97,31.15) .. (456.75,29.57) .. controls (456.54,27.98) and (457.65,26.52) .. (459.23,26.3) .. controls (460.82,26.08) and (462.28,27.19) .. (462.5,28.78) -- cycle ;
\draw  [fill={rgb, 255:red, 0; green, 0; blue, 0 }  ,fill opacity=1 ] (430.66,83.78) .. controls (430.88,85.37) and (429.77,86.83) .. (428.19,87.04) .. controls (426.6,87.26) and (425.14,86.15) .. (424.92,84.56) .. controls (424.7,82.98) and (425.81,81.52) .. (427.4,81.3) .. controls (428.99,81.08) and (430.45,82.19) .. (430.66,83.78) -- cycle ;
\draw  [fill={rgb, 255:red, 0; green, 0; blue, 0 }  ,fill opacity=1 ] (492.46,40.91) .. controls (491.17,41.86) and (489.36,41.6) .. (488.4,40.31) .. controls (487.45,39.03) and (487.72,37.21) .. (489.01,36.26) .. controls (490.29,35.3) and (492.11,35.57) .. (493.06,36.85) .. controls (494.02,38.14) and (493.75,39.96) .. (492.46,40.91) -- cycle ;
\draw  [fill={rgb, 255:red, 0; green, 0; blue, 0 }  ,fill opacity=1 ] (461.24,112.77) .. controls (459.95,113.72) and (458.14,113.45) .. (457.18,112.17) .. controls (456.23,110.88) and (456.5,109.07) .. (457.78,108.11) .. controls (459.07,107.16) and (460.88,107.43) .. (461.84,108.71) .. controls (462.79,110) and (462.52,111.81) .. (461.24,112.77) -- cycle ;
\draw [color={rgb, 255:red, 0; green, 0; blue, 0 }  ,draw opacity=1 ][line width=1.5]    (490.73,38.58) -- (530.77,38.6) ;
\draw  [fill={rgb, 255:red, 0; green, 0; blue, 0 }  ,fill opacity=1 ] (538.83,75.21) .. controls (538.61,76.79) and (539.72,78.26) .. (541.31,78.47) .. controls (542.89,78.69) and (544.35,77.58) .. (544.57,75.99) .. controls (544.79,74.41) and (543.68,72.94) .. (542.09,72.73) .. controls (540.51,72.51) and (539.05,73.62) .. (538.83,75.21) -- cycle ;
\draw  [fill={rgb, 255:red, 0; green, 0; blue, 0 }  ,fill opacity=1 ] (500.01,77.91) .. controls (501.29,78.87) and (503.11,78.6) .. (504.06,77.31) .. controls (505.01,76.03) and (504.74,74.21) .. (503.46,73.26) .. controls (502.17,72.3) and (500.36,72.57) .. (499.41,73.85) .. controls (498.45,75.14) and (498.72,76.96) .. (500.01,77.91) -- cycle ;
\draw [color={rgb, 255:red, 136; green, 4; blue, 22 }  ,draw opacity=1 ][line width=1.5]  [dash pattern={on 1.69pt off 2.76pt}]  (337.06,59.5) -- (366.79,84.17) ;
\draw [color={rgb, 255:red, 136; green, 4; blue, 22 }  ,draw opacity=1 ][line width=1.5]    (337.06,59.5) -- (337.06,109.5) ;
\draw [color={rgb, 255:red, 136; green, 4; blue, 22 }  ,draw opacity=1 ][line width=1.5]  [dash pattern={on 1.69pt off 2.76pt}]  (337.06,109.5) -- (366.79,84.17) ;
\draw  [fill={rgb, 255:red, 0; green, 0; blue, 0 }  ,fill opacity=1 ] (335.32,61.83) .. controls (336.61,62.78) and (338.43,62.51) .. (339.39,61.23) .. controls (340.34,59.94) and (340.07,58.13) .. (338.79,57.17) .. controls (337.5,56.22) and (335.68,56.49) .. (334.72,57.77) .. controls (333.77,59.06) and (334.04,60.87) .. (335.32,61.83) -- cycle ;
\draw  [fill={rgb, 255:red, 0; green, 0; blue, 0 }  ,fill opacity=1 ] (335.32,111.83) .. controls (336.61,112.78) and (338.43,112.51) .. (339.39,111.23) .. controls (340.34,109.94) and (340.07,108.13) .. (338.79,107.17) .. controls (337.5,106.22) and (335.68,106.49) .. (334.72,107.77) .. controls (333.77,109.06) and (334.04,110.87) .. (335.32,111.83) -- cycle ;
\draw [color={rgb, 255:red, 0; green, 0; blue, 0 }  ,draw opacity=1 ][line width=1.5]    (398.51,110.44) -- (366.79,85.17) ;
\draw [color={rgb, 255:red, 0; green, 0; blue, 0 }  ,draw opacity=1 ][line width=1.5]    (398.84,29.14) -- (366.79,85.17) ;
\draw [color={rgb, 255:red, 136; green, 4; blue, 22 }  ,draw opacity=1 ][line width=1.5]    (398.51,110.44) -- (398.84,29.14) ;
\draw  [fill={rgb, 255:red, 0; green, 0; blue, 0 }  ,fill opacity=1 ] (401.5,28.78) .. controls (401.71,30.37) and (400.6,31.83) .. (399.02,32.05) .. controls (397.43,32.26) and (395.97,31.15) .. (395.75,29.57) .. controls (395.54,27.98) and (396.65,26.52) .. (398.23,26.3) .. controls (399.82,26.08) and (401.28,27.19) .. (401.5,28.78) -- cycle ;
\draw  [fill={rgb, 255:red, 0; green, 0; blue, 0 }  ,fill opacity=1 ] (369.66,83.78) .. controls (369.88,85.37) and (368.77,86.83) .. (367.19,87.04) .. controls (365.6,87.26) and (364.14,86.15) .. (363.92,84.56) .. controls (363.7,82.98) and (364.81,81.52) .. (366.4,81.3) .. controls (367.99,81.08) and (369.45,82.19) .. (369.66,83.78) -- cycle ;
\draw  [fill={rgb, 255:red, 0; green, 0; blue, 0 }  ,fill opacity=1 ] (400.24,112.77) .. controls (398.95,113.72) and (397.14,113.45) .. (396.18,112.17) .. controls (395.23,110.88) and (395.5,109.07) .. (396.78,108.11) .. controls (398.07,107.16) and (399.88,107.43) .. (400.84,108.71) .. controls (401.79,110) and (401.52,111.81) .. (400.24,112.77) -- cycle ;
\draw [color={rgb, 255:red, 136; green, 4; blue, 22 }  ,draw opacity=1 ][line width=1.5]  [dash pattern={on 1.69pt off 2.76pt}]  (143.29,29.84) -- (49.47,59.89) ;
\draw [color={rgb, 255:red, 136; green, 4; blue, 22 }  ,draw opacity=1 ][line width=1.5]  [dash pattern={on 1.69pt off 2.76pt}]  (143.51,29.8) -- (81.72,60.17) ;
\draw [color={rgb, 255:red, 0; green, 0; blue, 0 }  ,draw opacity=1 ][line width=1.5]  [dash pattern={on 1.69pt off 2.76pt}]  (111.46,84.84) -- (49.47,59.89) ;
\draw [color={rgb, 255:red, 136; green, 4; blue, 22 }  ,draw opacity=1 ][line width=1.5]  [dash pattern={on 1.69pt off 2.76pt}]  (225.72,110.5) -- (287.18,111.44) ;
\draw [color={rgb, 255:red, 136; green, 4; blue, 22 }  ,draw opacity=1 ][line width=1.5]  [dash pattern={on 1.69pt off 2.76pt}]  (111.46,84.84) -- (50.01,110.77) ;
\draw  [fill={rgb, 255:red, 0; green, 0; blue, 0 }  ,fill opacity=1 ] (51.45,112.49) .. controls (50.16,113.45) and (48.35,113.18) .. (47.39,111.89) .. controls (46.44,110.61) and (46.71,108.79) .. (47.99,107.84) .. controls (49.28,106.88) and (51.1,107.15) .. (52.05,108.44) .. controls (53,109.72) and (52.74,111.54) .. (51.45,112.49) -- cycle ;
\draw  [fill={rgb, 255:red, 0; green, 0; blue, 0 }  ,fill opacity=1 ] (51.19,62.22) .. controls (49.91,63.18) and (48.09,62.91) .. (47.14,61.62) .. controls (46.18,60.34) and (46.45,58.52) .. (47.74,57.57) .. controls (49.03,56.61) and (50.84,56.88) .. (51.8,58.17) .. controls (52.75,59.45) and (52.48,61.27) .. (51.19,62.22) -- cycle ;
\draw  [fill={rgb, 255:red, 0; green, 0; blue, 0 }  ,fill opacity=1 ] (114.33,84.45) .. controls (114.55,86.03) and (113.44,87.49) .. (111.85,87.71) .. controls (110.27,87.93) and (108.81,86.82) .. (108.59,85.23) .. controls (108.37,83.64) and (109.48,82.18) .. (111.07,81.97) .. controls (112.65,81.75) and (114.11,82.86) .. (114.33,84.45) -- cycle ;
\draw  [fill={rgb, 255:red, 0; green, 0; blue, 0 }  ,fill opacity=1 ] (79.99,62.49) .. controls (81.28,63.45) and (83.1,63.18) .. (84.05,61.89) .. controls (85.01,60.61) and (84.74,58.79) .. (83.45,57.84) .. controls (82.16,56.88) and (80.35,57.15) .. (79.39,58.44) .. controls (78.43,59.72) and (78.7,61.54) .. (79.99,62.49) -- cycle ;
\draw  [fill={rgb, 255:red, 0; green, 0; blue, 0 }  ,fill opacity=1 ] (79.99,112.49) .. controls (81.28,113.45) and (83.1,113.18) .. (84.05,111.89) .. controls (85.01,110.61) and (84.74,108.79) .. (83.45,107.84) .. controls (82.16,106.88) and (80.35,107.15) .. (79.39,108.44) .. controls (78.43,109.72) and (78.7,111.54) .. (79.99,112.49) -- cycle ;
\draw  [fill={rgb, 255:red, 0; green, 0; blue, 0 }  ,fill opacity=1 ] (146.16,29.45) .. controls (146.38,31.03) and (145.27,32.5) .. (143.68,32.71) .. controls (142.1,32.93) and (140.64,31.82) .. (140.42,30.23) .. controls (140.2,28.65) and (141.31,27.18) .. (142.9,26.97) .. controls (144.48,26.75) and (145.95,27.86) .. (146.16,29.45) -- cycle ;
\draw [color={rgb, 255:red, 0; green, 0; blue, 0 }  ,draw opacity=1 ][line width=1.5]    (193.72,60.5) -- (193.72,110.5) ;
\draw [color={rgb, 255:red, 0; green, 0; blue, 0 }  ,draw opacity=1 ][line width=1.5]  [dash pattern={on 1.69pt off 2.76pt}]  (225.72,60.5) -- (255.46,85.17) ;
\draw [color={rgb, 255:red, 0; green, 0; blue, 0 }  ,draw opacity=1 ][line width=1.5]    (225.72,60.5) -- (225.72,110.5) ;
\draw [color={rgb, 255:red, 136; green, 4; blue, 22 }  ,draw opacity=1 ][line width=1.5]  [dash pattern={on 1.69pt off 2.76pt}]  (225.72,110.5) -- (255.46,85.17) ;
\draw [color={rgb, 255:red, 136; green, 4; blue, 22 }  ,draw opacity=1 ][line width=1.5]  [dash pattern={on 1.69pt off 2.76pt}]  (193.47,60.23) -- (225.72,60.5) ;
\draw [color={rgb, 255:red, 136; green, 4; blue, 22 }  ,draw opacity=1 ][line width=1.5]    (287.18,111.44) -- (255.46,86.17) ;
\draw [color={rgb, 255:red, 0; green, 0; blue, 0 }  ,draw opacity=1 ][line width=1.5]    (287.51,30.14) -- (255.46,86.17) ;
\draw [color={rgb, 255:red, 0; green, 0; blue, 0 }  ,draw opacity=1 ][line width=1.5]    (287.18,111.44) -- (287.51,30.14) ;
\draw  [fill={rgb, 255:red, 0; green, 0; blue, 0 }  ,fill opacity=1 ] (288.9,113.77) .. controls (287.62,114.72) and (285.8,114.45) .. (284.85,113.17) .. controls (283.89,111.88) and (284.16,110.07) .. (285.45,109.11) .. controls (286.74,108.16) and (288.55,108.43) .. (289.51,109.71) .. controls (290.46,111) and (290.19,112.81) .. (288.9,113.77) -- cycle ;
\draw [color={rgb, 255:red, 136; green, 4; blue, 22 }  ,draw opacity=1 ][line width=1.5]  [dash pattern={on 1.69pt off 2.76pt}]  (287.29,30.17) -- (193.47,60.23) ;
\draw [color={rgb, 255:red, 136; green, 4; blue, 22 }  ,draw opacity=1 ][line width=1.5]  [dash pattern={on 1.69pt off 2.76pt}]  (287.51,30.14) -- (225.72,60.5) ;
\draw [color={rgb, 255:red, 0; green, 0; blue, 0 }  ,draw opacity=1 ][line width=1.5]  [dash pattern={on 1.69pt off 2.76pt}]  (255.46,85.17) -- (193.47,60.23) ;
\draw [color={rgb, 255:red, 0; green, 0; blue, 0 }  ,draw opacity=1 ][line width=1.5]  [dash pattern={on 1.69pt off 2.76pt}]  (255.46,85.17) -- (194.01,111.1) ;
\draw  [fill={rgb, 255:red, 0; green, 0; blue, 0 }  ,fill opacity=1 ] (195.45,112.83) .. controls (194.16,113.78) and (192.35,113.51) .. (191.39,112.23) .. controls (190.44,110.94) and (190.71,109.13) .. (191.99,108.17) .. controls (193.28,107.22) and (195.1,107.49) .. (196.05,108.77) .. controls (197,110.06) and (196.74,111.87) .. (195.45,112.83) -- cycle ;
\draw  [fill={rgb, 255:red, 0; green, 0; blue, 0 }  ,fill opacity=1 ] (195.19,62.55) .. controls (193.91,63.51) and (192.09,63.24) .. (191.14,61.96) .. controls (190.18,60.67) and (190.45,58.86) .. (191.74,57.9) .. controls (193.03,56.95) and (194.84,57.21) .. (195.8,58.5) .. controls (196.75,59.78) and (196.48,61.6) .. (195.19,62.55) -- cycle ;
\draw  [fill={rgb, 255:red, 0; green, 0; blue, 0 }  ,fill opacity=1 ] (258.33,84.78) .. controls (258.55,86.37) and (257.44,87.83) .. (255.85,88.04) .. controls (254.27,88.26) and (252.81,87.15) .. (252.59,85.56) .. controls (252.37,83.98) and (253.48,82.52) .. (255.07,82.3) .. controls (256.65,82.08) and (258.11,83.19) .. (258.33,84.78) -- cycle ;
\draw  [fill={rgb, 255:red, 0; green, 0; blue, 0 }  ,fill opacity=1 ] (223.99,62.83) .. controls (225.28,63.78) and (227.1,63.51) .. (228.05,62.23) .. controls (229.01,60.94) and (228.74,59.13) .. (227.45,58.17) .. controls (226.16,57.22) and (224.35,57.49) .. (223.39,58.77) .. controls (222.43,60.06) and (222.7,61.87) .. (223.99,62.83) -- cycle ;
\draw  [fill={rgb, 255:red, 0; green, 0; blue, 0 }  ,fill opacity=1 ] (223.99,112.83) .. controls (225.28,113.78) and (227.1,113.51) .. (228.05,112.23) .. controls (229.01,110.94) and (228.74,109.13) .. (227.45,108.17) .. controls (226.16,107.22) and (224.35,107.49) .. (223.39,108.77) .. controls (222.43,110.06) and (222.7,111.87) .. (223.99,112.83) -- cycle ;
\draw  [fill={rgb, 255:red, 0; green, 0; blue, 0 }  ,fill opacity=1 ] (290.16,29.78) .. controls (290.38,31.37) and (289.27,32.83) .. (287.68,33.05) .. controls (286.1,33.26) and (284.64,32.15) .. (284.42,30.57) .. controls (284.2,28.98) and (285.31,27.52) .. (286.9,27.3) .. controls (288.48,27.08) and (289.95,28.19) .. (290.16,29.78) -- cycle ;
\draw [color={rgb, 255:red, 0; green, 0; blue, 0 }  ,draw opacity=1 ][line width=1.5]    (501.73,75.58) -- (541.7,75.6) ;
\draw  [fill={rgb, 255:red, 0; green, 0; blue, 0 }  ,fill opacity=1 ] (527.9,38.21) .. controls (527.68,39.79) and (528.79,41.26) .. (530.37,41.47) .. controls (531.96,41.69) and (533.42,40.58) .. (533.63,38.99) .. controls (533.85,37.41) and (532.74,35.94) .. (531.16,35.73) .. controls (529.58,35.51) and (528.12,36.62) .. (527.9,38.21) -- cycle ;

\draw (89.43,57) node [anchor=north west][inner sep=0.75pt]    {$\omega _{2}$};
\draw (40,45) node [anchor=north west][inner sep=0.75pt]    {$\omega _{1}$};
\draw (135.43,16) node [anchor=north west][inner sep=0.75pt]    {$\omega _{3}$};
\draw (478.33,22) node [anchor=north west][inner sep=0.75pt]    {$\omega _{1}$};
\draw (359,61) node [anchor=north west][inner sep=0.75pt]    {$x'$};
\draw (497.6,81) node [anchor=north west][inner sep=0.75pt]    {$\omega _{2}$};
\draw (321,49) node [anchor=north west][inner sep=0.75pt]    {$u$};
\draw (321,103.73) node [anchor=north west][inner sep=0.75pt]    {$v$};
\draw (445.1,15) node [anchor=north west][inner sep=0.75pt]    {$y$};
\draw (454.6,114) node [anchor=north west][inner sep=0.75pt]    {$\omega _{3}$};
\draw (384.43,15) node [anchor=north west][inner sep=0.75pt]    {$y'$};
\draw (400,105) node [anchor=north west][inner sep=0.75pt]    {$z'$};
\draw (433,77) node [anchor=north west][inner sep=0.75pt]    {$x$};
\draw (38,138) node [anchor=north west][inner sep=0.75pt]   [align=left] {Illustration of \eqref{edges-M-m2-1-w1-cap-M=2}};
\draw (104,92) node [anchor=north west][inner sep=0.75pt]    {$x$};
\draw (138.1,115) node [anchor=north west][inner sep=0.75pt]    {$y$};
\draw (77.71,115) node [anchor=north west][inner sep=0.75pt]    {$v_{2}$};
\draw (46.77,115) node [anchor=north west][inner sep=0.75pt]    {$v_{1}$};
\draw (233.43,58) node [anchor=north west][inner sep=0.75pt]    {$\omega _{2}$};
\draw (184,44) node [anchor=north west][inner sep=0.75pt]    {$\omega _{1}$};
\draw (279.43,16) node [anchor=north west][inner sep=0.75pt]    {$\omega _{3}$};
\draw (248,91) node [anchor=north west][inner sep=0.75pt]    {$x$};
\draw (282.1,116) node [anchor=north west][inner sep=0.75pt]    {$y$};
\draw (221.71,115) node [anchor=north west][inner sep=0.75pt]    {$v_{2}$};
\draw (190.77,115) node [anchor=north west][inner sep=0.75pt]    {$v_{1}$};
\draw (521.83,22) node [anchor=north west][inner sep=0.75pt]    {$v_{1}$};
\draw (540.83,81) node [anchor=north west][inner sep=0.75pt]    {$v_{2}$};
\draw (370,138) node [anchor=north west][inner sep=0.75pt]   [align=left] {Illustration of \eqref{edges-T1-w1-cap-M=2}};
\draw (182,138) node [anchor=north west][inner sep=0.75pt]   [align=left] {Illustration of \eqref{edges-T1-M-w1-cap-M=2}};

\end{tikzpicture}

    \caption{$|\{\omega_1,\omega_2,\omega_3\}\cap V(\mathcal{M})|=2$.}
    \label{FIG:cap=2}
\end{figure}
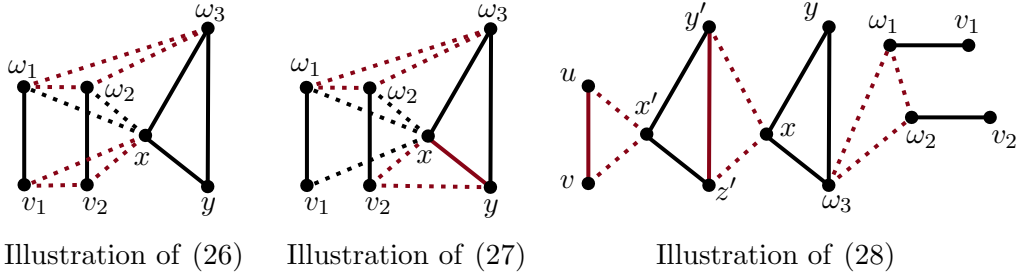

Using \eqref{edges-M-m2-1-w1-cap-M=2}, \eqref{edges-T1-M-w1-cap-M=2} and \eqref{edges-T1-w1-cap-M=2}, we update the terms involving  $e(\mathcal{T}_1)$,  $e(\mathcal{T}_1,\mathcal{M})$ and $e(\mathcal T_1)$ in $h$ to obtain 
\begin{align*}
    |H |\le q(t_1,t_2,t_3,t_4,m,i)+1\le  e_1(n,t)+\frac{5}{4}<N, 
\end{align*}
where the second inequality follows from Fact \ref{FACT:function-q}, completing the proof of Claim \ref{CLAIM:K3-free-V'-M}. 
\end{proof}

By Claim \ref{CLM:two-edge-contian-C4}, we choose $u_1v_1,u_2v_2\in\mathcal M$ such that $H[\{u_1,v_1,u_2,v_2\}]$ contains a cycle of length 4. Without loss of generality, we may assume that $u_iv_j\in E(H)$ for $i\neq j\in[2]$.
In the following, we consider the color on $v_1v_2$. Observe that at least one of the triangles $v_1v_2u_1$ or $v_1v_2u_2$ must be rainbow. Without loss of generality, we assume that $v_1v_2u_1$ is rainbow. Note that $\chi(v_1v_2)\in \chi(\mathcal T)$; otherwise,  $v_1v_2u_1$ together with triangles in $\mathcal{T}$ would form a rainbow copy of $(t+2)K_3$, a contradiction. Suppose $\chi(v_1v_2)\in \chi( xyz)$ for some $ xyz\in\mathcal{T}$.

We will consider the following two cases: $xyz\in\mathcal{T}_1$ and $xyz\in \mathcal{T} \setminus \mathcal{T}_1$. In both cases, we can find a copy of $K_{t_1-10, m + t_1-10, m + t_1-10}$ in $ H\left[V(\mathcal{T}_1) \cup V(\mathcal{M})\right]$, completing the proof of Theorem \ref{THM:tripartite-graph}. 

\medskip 

{\bf{Case 1}}: $xyz\in\mathcal{T}_1$

 Note that $u_1v_1,u_2v_2$ are the only edges in $\mathcal{M}$ that see $xyz$; otherwise, for each $uv\in \mathcal{M}\setminus\{u_1v_1,u_2v_2\}$ that sees $ xyz$, $\mathcal{T}\setminus\{ xyz\}\cup\{ v_1v_2u_1, xuv\}$ is a rainbow-$(t + 2)K_3$ in $K_n$. This means that $e(x, \mathcal{M})\le m+2$. Recall that $V''$ be the set of critical vertex in $\mathcal{T}_1$. Then 
\begin{align}\label{EQ:upperbound-criticalvtx-M}
    e(V'',V(\mathcal M))\le2mt_1-(m-2).
\end{align}



Recall that $H'=H[V'\cup V(\mathcal{M})]$ is triangle-free by Claim \ref{CLAIM:K3-free-V'-M}. We have the following:

\begin{claim}\label{CLM:M-bipartite-case-xyz-T1}
    The graph $H'$ is a bipartite graph.
\end{claim}

\begin{proof}
By Theorem \ref{THM:AES-Thm}, it suffices to establish $\delta(H')>4(m+t_1)/5$. Suppose to the contrary there exits  $w\in V'\cup V(\mathcal M)$ such that $d_{H'}(w)\le4(m+t_1)/5$.   We will update $e(\mathcal T_1)+e(\mathcal T_1,\mathcal M)+e(\mathcal M)$ in $h$ to get a contradiction. Note that
    \begin{align*}
        e(\mathcal T_1)+e(\mathcal T_1,\mathcal M)+e(\mathcal M)=\left|H[V'']\right|+|H'|+e(V'',V')+e(V'', V(\mathcal M)).
    \end{align*}
Naturally, we have
\begin{align}\label{EQ:trivial-bound-V1-V2-and-in-V2}
       \left|H[V'']\right|\le\binom{t_1}{2}.
\end{align}
As in the rightmost figure of Figure~\ref{FIG:cap=2}, if $ x'y'z'\in \mathcal T_1\setminus\{ xyz\}$ is seen by an edge $uv\in\mathcal M\setminus\{u_1v_1,u_2v_2\}$, then at most one of $\{xy',xz'\}$ lies in $E(H)$.
By Claim~\ref{CLM:at-most-two-triangle}, the number of choice $x'y'z'$ is at least $t_1-2$, thus
\begin{align}\label{EQ:bound-V1-V2}
       e(V'',V')\le 2t_1^2-(t_1-2).
\end{align}
Now we bound $|H'|$. By Claim \ref{CLAIM:K3-free-V'-M}, we know that $H'-w$ is triangle-free. Using the classical Mentel's Theorem, 
    \begin{align*}
       |H'\setminus\{w\}|\le \left\lfloor\frac{(2t_1+2m-1)^2}{4}\right\rfloor. 
    \end{align*}
Thus, 
    \begin{align}\label{EQ:upperbound-V'&M}
        |H'|&=|H'\setminus\{w\}|+d_{H'}(w)
        \le \left\lfloor\frac{(2t_1+2m-1)^2}{4}\right\rfloor+\frac{4(m+t_1)}{5}= (t_1+m)^2-\frac15(t_1+m). 
    \end{align} 
We refine the upper bound of $e(\mathcal T_1) + e(\mathcal T_1,\mathcal M) + e(\mathcal M)$ in $h$ by replacing the earlier estimate from Lemma~\ref{LEM:upper-bound-each-part-1}~\ref{itm:m}~\ref{itm:mt1}~\ref{itm:t1}, i.e., $e(\mathcal T_1) + e(\mathcal T_1,\mathcal M) + e(\mathcal M) \le 7\binom{t_1}{2} + 3t_1 + 4mt_1 + m^2$, with:
    \begin{align*}
        e(\mathcal T_1,\mathcal M) + &e(\mathcal M) + e(\mathcal T_1)
        \\\le& 2mt_1 + \binom{t_1}{2} + 2t_1^2 + (t_1 + m)^2 - (m - 2)-(t_1-2) - \frac{1}{5}(t_1 + m)\\
        \le& 7\binom{t_1}{2} + 3t_1 + m^2 + 4mt_1 - (m - 2)-(t_1-2) - \frac{1}{5}(t_1 + m), 
    \end{align*}
given by \eqref{EQ:upperbound-criticalvtx-M}, \eqref{EQ:trivial-bound-V1-V2-and-in-V2}, \eqref{EQ:bound-V1-V2} and \eqref{EQ:upperbound-V'&M}. 
Thus, 
    \begin{align*}
        |H |
        &\le f + im + m^2 + (3 + 3m)t_4 + (2 + i)t_4 + \binom{3t_4}{2} - (m - 2) - (t_1 - 2) - \frac{t_1 + m}{5}\\
        &= q_1(t_1,t_2,t_3,t_4,m,i) - \frac{t_1 + m}{5} + 4 < N,
    \end{align*}
where the last inequality follows from Fact \ref{FACT:function-q} and $t_1,m=\Theta(n)$. 
\end{proof}

\begin{claim}\label{CLAIM:number-T-case-T1}
    We have 
    \begin{multline*}
        e\left(\mathcal{T}_1\setminus\{ xyz\} \right) + e\left(\mathcal T_1\setminus\{ xyz\},\, \mathcal M \right) + e\left(\mathcal M \right) 
        \geq 7\binom{t_1-1}{2}+3(t_1-1) +m^2+4m(t_1-1)-9. 
    \end{multline*}
\end{claim} 
\begin{proof}
By contradiction, we will update the bound on $e(\mathcal T_1)+e(\mathcal M)+e(\mathcal T_1,\mathcal M)$ in $h$, replacing the earlier estimate given by Lemma~\ref{LEM:upper-bound-each-part-1}.

For any $ x'y'z'\in \mathcal{T}_1'$ with the critical vertex $x'$, $y'z'$ cannot see any vertex of $ xyz$; otherwise, for an $\mathcal M$-edge $uv\in \mathcal{M}\setminus\{u_1v_1,u_2v_2\}$ and $y'z'$ sees $x_0\in\{x,y,z\}$ of $xyz$, $\mathcal{T}\setminus\{ xyz, x'y'z'\}\cup\{ v_1v_2u_1, x'uv, y'z'x_0\}$ is a rainbow-$(t + 2)K_3$ in $K_n$. This means $e(xyz,  x'y'z')\le6$ and then 
    \begin{align}\label{EQ:upperbound-T1(xyz)-xyz}
        e(xyz, \mathcal T_1\setminus\{xyz\})\le7(t_1-1)-(t_1-2).
    \end{align}

Recall that $u_1v_1,u_2v_2$ are the only edges in $\mathcal{M}$ that see $xyz$. We have 
    \begin{align}\label{EQ:upperbound-xyz-M-2}
        e(xyz,V(\mathcal M))\le4m-(m-2).
    \end{align}

Combining \eqref{EQ:upperbound-T1(xyz)-xyz} and \eqref{EQ:upperbound-xyz-M-2}, we have 
\begin{align*}
    e(\mathcal T_1)+e(\mathcal M)+e(\mathcal T_1,\mathcal M)
    &=e(\mathcal T_1\setminus\{ xyz\})+e(\mathcal T_1\setminus\{ xyz\},\mathcal M)+e(\mathcal M)+e( xyz)\\
    &\quad+e( xyz,\mathcal M)+e(\mathcal T_1\setminus\{ xyz\}, xyz)\\
    &\le 7\binom{t_1-1}{2} +3(t_1-1)+m^2+4m(t_1-1)- 10+3\\
    &\quad+4m-(m-2)+7(t_1-1)-(t_1-2)\\
    &=7\binom{t_1}{2} +3t_1+m^2+4mt_1- 10-(m-2)-(t_1-2),
\end{align*}
which together with Fact \ref{FACT:function-q} gives
    \begin{align*}
        |H |&\le f+im+m^2+(3+3m)t_4+(2+i)t_4+\binom{3t_4}{2}-(m-2)-(t_1-2)-10\\
        &\le q_1(t_1,t_2,t_3,t_4,m,i)-5<N,
    \end{align*}
a contradiction.
\end{proof}

Let $U'=V'-\{y,z\}$, $U''=V''-\{x\}$, and let $T$ be the subgraph of $H$ induced by $U'\cup U''\cup V(\mathcal{M})$, with all edges between vertices in $U''$ removed. By Claim~\ref{CLM:M-bipartite-case-xyz-T1}, $T$ is a $(t_1-1)\times (t_1+m-1)\times (t_1+m-1)$ tripartite graph. On one hand, 
    \begin{align*}
        e(T)\le (t_1-1)(2(t_1-1)+2m)+(t_1-1+m)^2,
    \end{align*}
with equality if $T$ is complete. 
On the other hand, it follows from Claim \ref{CLAIM:number-T-case-T1} that
\begin{align*}
    e(T)
    &= e(\mathcal T_1\setminus\{ xyz\})+e(\mathcal T_1\setminus\{ xyz\},\mathcal M)+e(\mathcal M)-e(H[U''])\\
    &\ge 7\binom{t_1-1}{2}+3(t_1-1) +m^2+4m(t_1-1)- 9-\binom{t_1-1}{2}\\
    &=(t_1-1)(2(t_1-1)+2m)+(t_1-1+m)^2-9.
\end{align*}
Clearly, $T$ is obtained from a  $(t_1-1)\times (t_1+m-1)\times (t_1+m-1)$ complete tripartite graph by deleting at most 9 edges and thus it contains a copy of $K_{t_1-10, m + t_1-10, m + t_1-10}$.

\begin{figure}
    \centering

\tikzset{every picture/.style={line width=0.75pt}} 

\begin{tikzpicture}[x=0.75pt,y=0.75pt,yscale=-0.92,xscale=0.92]

\tikzset{every picture/.style={line width=0.75pt}} 

\draw [color={rgb, 255:red, 57; green, 94; blue, 178 }  ,draw opacity=1 ][line width=1.5]    (412.26,97.48) -- (380.2,124.82) ;
\draw [color={rgb, 255:red, 136; green, 4; blue, 22 }  ,draw opacity=1 ][line width=1.5]  [dash pattern={on 1.69pt off 2.76pt}]  (247.51,47.14) -- (284.13,116.45) ;
\draw [color={rgb, 255:red, 57; green, 94; blue, 178 }  ,draw opacity=1 ][line width=1.5]    (247.29,47.17) -- (215.46,102.17) ;
\draw [color={rgb, 255:red, 136; green, 4; blue, 22 }  ,draw opacity=1 ][line width=1.5]  [dash pattern={on 1.69pt off 2.76pt}]  (470.18,124.82) -- (379.64,125.68) ;
\draw [color={rgb, 255:red, 136; green, 4; blue, 22 }  ,draw opacity=1 ][line width=1.5]  [dash pattern={on 1.69pt off 2.76pt}]  (507.13,61.83) -- (470.51,43.52) ;
\draw [color={rgb, 255:red, 136; green, 4; blue, 22 }  ,draw opacity=1 ][line width=1.5]  [dash pattern={on 1.69pt off 2.76pt}]  (469.2,43.52) -- (507.13,112.83) ;
\draw [color={rgb, 255:red, 136; green, 4; blue, 22 }  ,draw opacity=1 ][line width=1.5]  [dash pattern={on 1.69pt off 2.76pt}]  (438.46,98.55) -- (380.2,124.82) ;
\draw [color={rgb, 255:red, 0; green, 0; blue, 0 }  ,draw opacity=1 ][line width=1.5]    (470.29,43.56) -- (438.46,98.55) ;
\draw [color={rgb, 255:red, 57; green, 94; blue, 178 }  ,draw opacity=1 ][line width=1.5]  [dash pattern={on 1.69pt off 2.76pt}]  (111.17,52.8) -- (63.67,52.77) ;
\draw [color={rgb, 255:red, 0; green, 0; blue, 0 }  ,draw opacity=1 ][line width=1.5]    (380.08,43.56) -- (411.84,98.55) ;
\draw [color={rgb, 255:red, 0; green, 0; blue, 0 }  ,draw opacity=1 ][line width=1.5]    (380.2,124.82) -- (379.87,43.52) ;
\draw  [fill={rgb, 255:red, 0; green, 0; blue, 0 }  ,fill opacity=1 ] (378.47,127.15) .. controls (379.76,128.1) and (381.57,127.84) .. (382.52,126.55) .. controls (383.47,125.27) and (383.2,123.45) .. (381.92,122.5) .. controls (380.64,121.54) and (378.83,121.81) .. (377.87,123.09) .. controls (376.92,124.38) and (377.19,126.19) .. (378.47,127.15) -- cycle ;
\draw  [fill={rgb, 255:red, 0; green, 0; blue, 0 }  ,fill opacity=1 ] (408.98,98.16) .. controls (408.76,99.75) and (409.87,101.21) .. (411.45,101.43) .. controls (413.03,101.64) and (414.49,100.54) .. (414.71,98.95) .. controls (414.92,97.36) and (413.82,95.9) .. (412.23,95.68) .. controls (410.65,95.47) and (409.19,96.57) .. (408.98,98.16) -- cycle ;
\draw  [fill={rgb, 255:red, 0; green, 0; blue, 0 }  ,fill opacity=1 ] (377.22,43.16) .. controls (377,44.75) and (378.11,46.21) .. (379.69,46.43) .. controls (381.27,46.65) and (382.73,45.54) .. (382.95,43.95) .. controls (383.16,42.36) and (382.06,40.9) .. (380.47,40.69) .. controls (378.89,40.47) and (377.43,41.58) .. (377.22,43.16) -- cycle ;
\draw [color={rgb, 255:red, 136; green, 4; blue, 22 }  ,draw opacity=1 ][line width=1.5]    (470.18,124.82) -- (438.46,99.55) ;
\draw [color={rgb, 255:red, 0; green, 0; blue, 0 }  ,draw opacity=1 ][line width=1.5]    (470.51,43.52) -- (470.18,124.82) ;
\draw  [fill={rgb, 255:red, 0; green, 0; blue, 0 }  ,fill opacity=1 ] (471.9,127.15) .. controls (470.62,128.1) and (468.8,127.84) .. (467.85,126.55) .. controls (466.89,125.27) and (467.16,123.45) .. (468.45,122.5) .. controls (469.74,121.54) and (471.55,121.81) .. (472.51,123.09) .. controls (473.46,124.38) and (473.19,126.19) .. (471.9,127.15) -- cycle ;
\draw  [fill={rgb, 255:red, 0; green, 0; blue, 0 }  ,fill opacity=1 ] (441.33,98.16) .. controls (441.55,99.75) and (440.44,101.21) .. (438.85,101.43) .. controls (437.27,101.65) and (435.81,100.53) .. (435.59,98.95) .. controls (435.37,97.36) and (436.48,95.9) .. (438.07,95.68) .. controls (439.65,95.46) and (441.11,96.58) .. (441.33,98.16) -- cycle ;
\draw  [fill={rgb, 255:red, 0; green, 0; blue, 0 }  ,fill opacity=1 ] (473.16,43.16) .. controls (473.38,44.75) and (472.27,46.21) .. (470.68,46.43) .. controls (469.1,46.65) and (467.64,45.54) .. (467.42,43.95) .. controls (467.2,42.36) and (468.31,40.9) .. (469.9,40.68) .. controls (471.48,40.47) and (472.95,41.58) .. (473.16,43.16) -- cycle ;
\draw [color={rgb, 255:red, 136; green, 4; blue, 22 }  ,draw opacity=1 ][line width=1.5]    (507.13,61.83) -- (507.13,112.83) ;
\draw  [fill={rgb, 255:red, 0; green, 0; blue, 0 }  ,fill opacity=1 ] (505.4,64.16) .. controls (506.69,65.11) and (508.51,64.84) .. (509.47,63.56) .. controls (510.42,62.27) and (510.15,60.46) .. (508.86,59.5) .. controls (507.58,58.55) and (505.76,58.82) .. (504.8,60.1) .. controls (503.84,61.39) and (504.11,63.2) .. (505.4,64.16) -- cycle ;
\draw  [fill={rgb, 255:red, 0; green, 0; blue, 0 }  ,fill opacity=1 ] (505.4,115.16) .. controls (506.69,116.11) and (508.51,115.85) .. (509.47,114.56) .. controls (510.42,113.28) and (510.15,111.46) .. (508.86,110.51) .. controls (507.58,109.55) and (505.76,109.82) .. (504.8,111.1) .. controls (503.84,112.39) and (504.11,114.21) .. (505.4,115.16) -- cycle ;
\draw [color={rgb, 255:red, 136; green, 4; blue, 22 }  ,draw opacity=1 ][line width=1.5]  [dash pattern={on 1.69pt off 2.76pt}]  (284.13,65.45) -- (247.51,47.14) ;
\draw [color={rgb, 255:red, 0; green, 0; blue, 0 }  ,draw opacity=1 ][line width=1.5]    (247.18,128.44) -- (215.46,103.17) ;
\draw [color={rgb, 255:red, 0; green, 0; blue, 0 }  ,draw opacity=1 ][line width=1.5]    (247.51,47.14) -- (247.18,128.44) ;
\draw  [fill={rgb, 255:red, 0; green, 0; blue, 0 }  ,fill opacity=1 ] (248.9,130.77) .. controls (247.62,131.72) and (245.8,131.45) .. (244.85,130.17) .. controls (243.89,128.88) and (244.16,127.07) .. (245.45,126.11) .. controls (246.74,125.16) and (248.55,125.43) .. (249.51,126.71) .. controls (250.46,128) and (250.19,129.81) .. (248.9,130.77) -- cycle ;
\draw  [fill={rgb, 255:red, 0; green, 0; blue, 0 }  ,fill opacity=1 ] (218.33,101.78) .. controls (218.55,103.37) and (217.44,104.83) .. (215.85,105.04) .. controls (214.27,105.26) and (212.81,104.15) .. (212.59,102.56) .. controls (212.37,100.98) and (213.48,99.52) .. (215.07,99.3) .. controls (216.65,99.08) and (218.11,100.19) .. (218.33,101.78) -- cycle ;
\draw  [fill={rgb, 255:red, 0; green, 0; blue, 0 }  ,fill opacity=1 ] (250.16,46.78) .. controls (250.38,48.37) and (249.27,49.83) .. (247.68,50.05) .. controls (246.1,50.26) and (244.64,49.15) .. (244.42,47.57) .. controls (244.2,45.98) and (245.31,44.52) .. (246.9,44.3) .. controls (248.48,44.08) and (249.95,45.19) .. (250.16,46.78) -- cycle ;
\draw [color={rgb, 255:red, 136; green, 4; blue, 22 }  ,draw opacity=1 ][line width=1.5]    (284.13,65.45) -- (284.13,116.45) ;
\draw  [fill={rgb, 255:red, 0; green, 0; blue, 0 }  ,fill opacity=1 ] (282.4,67.77) .. controls (283.69,68.73) and (285.51,68.46) .. (286.47,67.17) .. controls (287.42,65.89) and (287.15,64.07) .. (285.86,63.12) .. controls (284.58,62.16) and (282.76,62.43) .. (281.8,63.72) .. controls (280.84,65) and (281.11,66.82) .. (282.4,67.77) -- cycle ;
\draw  [fill={rgb, 255:red, 0; green, 0; blue, 0 }  ,fill opacity=1 ] (282.4,118.78) .. controls (283.69,119.73) and (285.51,119.46) .. (286.47,118.18) .. controls (287.42,116.89) and (287.15,115.08) .. (285.86,114.12) .. controls (284.58,113.17) and (282.76,113.44) .. (281.8,114.72) .. controls (280.84,116.01) and (281.11,117.82) .. (282.4,118.78) -- cycle ;
\draw [color={rgb, 255:red, 0; green, 0; blue, 0 }  ,draw opacity=1 ][line width=1.5]    (111.17,52.8) -- (111.17,116.11) ;
\draw  [fill={rgb, 255:red, 0; green, 0; blue, 0 }  ,fill opacity=1 ] (112.89,118.43) .. controls (111.61,119.39) and (109.79,119.12) .. (108.84,117.83) .. controls (107.88,116.55) and (108.15,114.73) .. (109.44,113.78) .. controls (110.73,112.82) and (112.54,113.09) .. (113.5,114.38) .. controls (114.45,115.66) and (114.18,117.48) .. (112.89,118.43) -- cycle ;
\draw  [fill={rgb, 255:red, 0; green, 0; blue, 0 }  ,fill opacity=1 ] (114.05,52.41) .. controls (114.26,54) and (113.15,55.46) .. (111.57,55.68) .. controls (109.98,55.89) and (108.52,54.78) .. (108.3,53.2) .. controls (108.09,51.61) and (109.2,50.15) .. (110.78,49.93) .. controls (112.37,49.71) and (113.83,50.82) .. (114.05,52.41) -- cycle ;
\draw [color={rgb, 255:red, 0; green, 0; blue, 0 }  ,draw opacity=1 ][line width=1.5]    (63.67,52.77) -- (63.67,116.07) ;
\draw  [fill={rgb, 255:red, 0; green, 0; blue, 0 }  ,fill opacity=1 ] (65.39,118.39) .. controls (64.11,119.35) and (62.29,119.08) .. (61.34,117.8) .. controls (60.38,116.51) and (60.65,114.69) .. (61.94,113.74) .. controls (63.23,112.79) and (65.04,113.05) .. (66,114.34) .. controls (66.95,115.62) and (66.68,117.44) .. (65.39,118.39) -- cycle ;
\draw  [fill={rgb, 255:red, 0; green, 0; blue, 0 }  ,fill opacity=1 ] (66.55,52.37) .. controls (66.76,53.96) and (65.65,55.42) .. (64.07,55.64) .. controls (62.48,55.86) and (61.02,54.75) .. (60.8,53.16) .. controls (60.59,51.57) and (61.7,50.11) .. (63.28,49.89) .. controls (64.87,49.68) and (66.33,50.79) .. (66.55,52.37) -- cycle ;
\draw [color={rgb, 255:red, 0; green, 0; blue, 0 }  ,draw opacity=1 ][line width=1.5]    (63.67,52.77) -- (111.17,116.11) ;
\draw [color={rgb, 255:red, 0; green, 0; blue, 0 }  ,draw opacity=1 ][line width=1.5]    (63.67,116.07) -- (111.17,52.8) ;

\draw (450.77,32.62) node [anchor=north west][inner sep=0.75pt]    {$x'$};
\draw (510.47,50.32) node [anchor=north west][inner sep=0.75pt]    {$u$};
\draw (478.77,119.95) node [anchor=north west][inner sep=0.75pt]    {$y'$};
\draw (444.17,80) node [anchor=north west][inner sep=0.75pt]    {$z'$};
\draw (407.93,80) node [anchor=north west][inner sep=0.75pt]    {$z$};
\draw (365.43,29.45) node [anchor=north west][inner sep=0.75pt]    {$y$};
\draw (510.86,113.91) node [anchor=north west][inner sep=0.75pt]    {$v$};
\draw (364.77,122.78) node [anchor=north west][inner sep=0.75pt]    {$x$};
\draw (240,32.23) node [anchor=north west][inner sep=0.75pt]    {$x$};
\draw (287.47,53.93) node [anchor=north west][inner sep=0.75pt]    {$u$};
\draw (252.77,122.57) node [anchor=north west][inner sep=0.75pt]    {$y$};
\draw (200,100) node [anchor=north west][inner sep=0.75pt]    {$z$};
\draw (287.86,117.52) node [anchor=north west][inner sep=0.75pt]    {$v$};
\draw (40.77,46.57) node [anchor=north west][inner sep=0.75pt]    {$u_{1}$};
\draw (44.77,111.57) node [anchor=north west][inner sep=0.75pt]    {$v_{1}$};
\draw (114.77,46.57) node [anchor=north west][inner sep=0.75pt]    {$u_{2}$};
\draw (115.77,111.57) node [anchor=north west][inner sep=0.75pt]    {$v_{2}$};
\draw (180,147.28) node [anchor=north west][inner sep=0.75pt]   [align=left] {Illustration of  \eqref{EQ:upperbound-criticalvtx-M}};
\draw (348,149.28) node [anchor=north west][inner sep=0.75pt]   [align=left] {Illustration of  \eqref{EQ:upperbound-T1(xyz)-xyz} and \eqref{update-edge-T1-Ti}};

\end{tikzpicture}

    \caption{$\chi(u_1u_2)=\chi(xz)$.}
    \label{FIG:chi(u_1u_2)=chi(xz)}
\end{figure}
\medskip

{\bf{Case 2}}: $xyz \in \mathcal{T} \setminus \mathcal{T}_1$.

We proceed using the same strategy as in Case 1. Let $\mathcal T_1'$ denote the triangles in $\mathcal{T}_1$ seen by at least $\varepsilon n$ $\mathcal{M}$-edges, and $\mathcal T_1''$ denote the remaining triangles in $\mathcal{T}_1$. By Claim \ref{CLM:at-most-two-triangle}, $|\mathcal T_1''|\le 2$. Now we update  $e(\mathcal{T}_1, \mathcal{T}_2 \cup \mathcal{T}_3 \cup \mathcal{T}_4).$ For $x'y'z' \in \mathcal{T}'_1$, then $y'z'$ cannot see any vertex in $xyz$; otherwise, assume $y'z'$ sees $x_0 \in \{x, y, z\}$, and an $\mathcal{M}$-edge $uv \in \mathcal{M} \setminus \{u_1v_1, u_2v_2\}$ sees $x'y'z' \in \mathcal{T}_1$ with the critical vertex $x'$. Then 
\begin{align}\label{update-edge-T1-Ti}
    \mathcal{T} \setminus \{xyz, x'y'z'\} \cup \{uvx', x_0y'z', u_1v_1v_2\}
\end{align}
is a rainbow-$(t + 2)K_3$ in $K_n$ (see Figure~\ref{FIG:chi(u_1u_2)=chi(xz)}).
This together with Lemma~\ref{LEM:upper-bound-each-part-2}~\ref{itm:t1>1-t1tj} gives, 
\begin{align*}
    e(\mathcal T_1',xyz)&\le6|\mathcal{T}_1'|,\\
    e(\mathcal T_1'', xyz)&\le7|\mathcal{T}_1''|,\\ e(\mathcal{T}_1, \mathcal{T}_2 \cup \mathcal{T}_3 \cup \mathcal{T}_4\setminus\{xyz\})&\le7t_1(t_2+t_3+t_4-1),
\end{align*}
and thus
\begin{align}\label{EQ:T1-T4-edge-case-notin-T1}
    e(\mathcal{T}_1, \mathcal{T}_2 \cup \mathcal{T}_3 \cup \mathcal{T}_4)&= e(\mathcal T_1, xyz)  + e(\mathcal{T}_1, \mathcal{T}_2 \cup \mathcal{T}_3 \cup \mathcal{T}_4 \setminus \{xyz\})  \notag\\
    &= e(\mathcal T_1', xyz) + e(\mathcal T_1'', xyz) + e(\mathcal{T}_1, \mathcal{T}_2 \cup \mathcal{T}_3 \cup \mathcal{T}_4 \setminus \{xyz\}) \notag\\
    &\leq 6|\mathcal T_1'| + 7|\mathcal T_1''| + 7t_1(t_2 + t_3 + t_4 - 1) \notag\\
    &\le 6(t_1-2)+7\times2+7t_1t_2+7t_1t_3+7t_1t_4-7t_1   \notag\\
    &= 7t_1t_2 + 7t_1 t_3 +  7t_1t_4 - (t_1 - 2).
\end{align}

\begin{claim}\label{CLM:M-bipartite-case-xyz-notin-T1}
    The graph $H'$ is a bipartite graph. 
\end{claim}

\begin{proof}
Recall that $H'=H[V'\cup V(\mathcal{M})]$ is triangle-free by Claim \ref{CLAIM:K3-free-V'-M}. Suppose to the contrary that $H'$ is not bipartite, then $\delta(H')\le \frac{4(t_1 + m)}{5}$ by 
Theorem \ref{THM:AES-Thm}. By arguments analogous to the proof of Claim \ref{CLAIM:K3-free-V'-M}, we have 
\begin{align*}
    e(\mathcal T_1)+e(\mathcal T_1,\mathcal M)+e(\mathcal M) \le7\binom{t_1}{2}+3t_1+4mt_1+m^2-\frac15(t_1+m).
\end{align*}

We update the bounds on $e(\mathcal T_1,\mathcal T_2\cup \mathcal T_3\cup\mathcal T_4)$ and $e(\mathcal T_1)+e(\mathcal T_1,\mathcal M)+e(\mathcal M)$ in $h$, replacing the earlier estimate given by~Lemma~\ref{LEM:upper-bound-each-part-2}~\ref{itm:t1>1-t1tj} and~Lemma~\ref{LEM:upper-bound-each-part-2}~\ref{itm:t1}~\ref{itm:mt1}~\ref{itm:m}, yields that
\begin{align*}
    |H | &\leq f + im + m^2 + (3 + 3m)t_4 + (2 + i)t_4 + \binom{3t_4}{2} - (t_1 - 2) - \frac{t_1 + m}{5} \\
    &\leq e_1(n,t) + t + 6 - (t_1 - 2) - \frac{1}{5}(t_1 + m) < N,
\end{align*}
where the second inequality follows from the moreover part of Fact \ref{fact1}. 
\end{proof}
We claim that  
\begin{align}\label{EQ:bound-T1-M-case-notin-T1}
    e(\mathcal T_1)+e(\mathcal T_1,\mathcal M)+e(\mathcal M)
    \geq 7\binom{t_1}{2}+3t_1 +m^2+4mt_1- 9;
\end{align}
otherwise, by Fact \ref{fact1} and \eqref{EQ:T1-T4-edge-case-notin-T1}, we have
\begin{align*}
    |H |
    &\le f+im+m^2+(3+3m)t_4+(2+i)t_4+\binom{3t_4}{2}-(t_1-2)-9 \\
    &\le e_1(n,t)+t+6-(t_1-2)-9
    < N,
\end{align*}
a contradiction. The last inequality follows from $t_1\ge t-1$.

Let $T'$ be the subgraph of $H$ induced by $V'\cup V''\cup V(\mathcal{M})$, with all edges between vertices in $V''$ removed. Then by Claim \ref{CLM:M-bipartite-case-xyz-notin-T1}, $T'$ is a $t_1\times(t_1+m)\times (t_1+m)$ tripartite graph. Clearly, $|T_2|\le t_1(2t_1+2m)+(t_1+m)^2$. On the other hand, by \eqref{EQ:bound-T1-M-case-notin-T1}, 
\begin{align*}
   |T'|
    & = e(\mathcal T_1)+e(\mathcal T_1,\mathcal M)+e(\mathcal M)-e(H[V''])\\
    & \ge 7\binom{t_1}{2}+3t_1 +m^2+4mt_1- 9-\binom{t_1}{2}\\
    & = t_1(2t_1+2m)+(t_1+m)^2-9.
\end{align*}
Thus, $T$ is obtained from a  $t_1\times(t_1+m)\times (t_1+m)$ complete tripartite graph by deleting at most 9 edges and thus it contains a copy of $K_{t_1-10, m + t_1-10, m + t_1-10}$. 
\end{proof}

\section{Proof of Theorem~\ref{THM:main-first-interval}}\label{SEC:PF-Main-result}
In this section, we prove Theorem~\ref{THM:main-first-interval}. Let $\delta>0$ be a constant given by Theorem~\ref{THM:Anti-ramsey-(t+2)K3-t-small}, $n\ge \min\{N_{\ref{THM:Anti-ramsey-(t+2)K3-t-small}}, N_{\ref{THM:STABILITY}}, N_{\ref{THM:tripartite-graph}}\}$ be a sufficiently large integer and $t\le \frac{1}{9}(2n-6)-2.$ By Theorem~\ref{THM:Anti-ramsey-(t+2)K3-t-small}, we may assume that $t\ge \delta n$.  Suppose to the contrary, there exists an edge-coloring surjective
$\chi\colon K_n\to[N]$ without any rainbow copy of $(t + 2)K_3$. Let $H$ be a representative graph under $\chi$. Then $|H |=N$. Combining Theorems~\ref{THM:STABILITY} and \ref{THM:tripartite-graph}, $H$ contains a copy of $K_{t_1-10, m+t_1-10, m+t_1-10}$. Recall that $t_1\ge t-1$ and $m\ge n/6-o(n)$. There exist  $V_1,V_2,V_3\subset V(H)$ with $|V_1| = t-11$ and $|V_2| = |V_3| = t + 28$, such that $H[V_1, V_2, V_3]$ is a copy of $K_{t-11, t + 28, t + 28}$ in $H$.
    
Let $H'\coloneqq H-V_1$. Then $v(H')=n-(t-11)$ and 
    \begin{align*}
        |H'|&\geq |H |-\left(\binom{|V_1|}{2}+|V_1|(n-|V_1|)\right)\\
        &=\left(\binom{t}{2}+t(n-t)+\left\lfloor\frac{(n-t)^2}{4}\right\rfloor+2\right)-\left(\binom{t-11}{2}+(t-11)(n-(t-11))\right)\\
        &=\binom{11}{2}+11(n-(t-11)-11)+\left\lfloor\frac{(n-t)^2}{4}\right\rfloor+2\\
        &= \mathrm{ex}(n-(t-11),12K_3)+2.
    \end{align*}
By Theorem \ref{THM:Anti-ramsey-(t+2)K3-t-small}, $H'$ contains a rainbow copy of $13K_3$, and we  denote by $C$ be the set of colors used on it. Then $|C|=39$. For $i=2,3$, if a vertex in $V_i$ is incident with an edge colored by any color in $C$, then we delete it from $V_i$, obtaining a subset  $V_i'\subseteq V_i$. It follows from $|V_i| = t + 28$ and $|C|=39$ that $H[V_1, V_2', V_3']$ contains a copy of $K_{t-11,t-11,t-11}$. This means that $H[V_1, V_2', V_3']$ contains a  rainbow copy of $(t-11)K_3$, which together with the rainbow copy of $13K_3$ in $H'$ gives a rainbow copy of $(t + 2)K_3$ in $K_n$. It completes the proof of Theorem~\ref{THM:main-first-interval}.

\section{Concluding remarks}\label{SEC:remark}
In the present work, we established Conjecture~\ref{Conj:Our-conjecture} for the first interval. In a subsequent work, we will extend this result to the second and third intervals. In particular, our results will show that in the first and second intervals, 
\begin{align*}
    \mathrm{ar}(n, (t+2)K_3) = \mathrm{ex}(n, (t+1)K_3) + 2.
\end{align*} 
while in the third interval, the difference between $\mathrm{ar}(n, (t+2)K_3)$ and $\mathrm{ex}(n, (t+1)K_3)$ is no longer a constant, namely, 
\begin{align*}
    \mathrm{ar}(n, (t+2)K_3)=\mathrm{ex}(n, (t+1)K_3) + n-2t.
\end{align*}

Let us briefly mention the difference between Conjecture~\ref{Conj:Our-conjecture} and the conjecture of Liu--Ning--Tian~{\cite[Conjecture~3]{LNT25}}.
The first, second, and fourth constructions in their conjecture are identical to the first, second, and fifth constructions in our conjecture.
Their third construction differs from our third construction in that the size of their set $X$ is smaller by one. 
Another difference is that their conjecture does not include the construction $E_{4}(n,t)$.

\section*{Acknowledgments}
J.H. was supported by the National Key R\&D Program of China (No.~2023YFA1010202) 
and by the Central Guidance on Local Science and Technology Development Fund of Fujian Province (No.~2023L3003). 
X.L. was supported by the Excellent Young Talents Program (Overseas) of the National Natural Science Foundation of China. 

\bibliographystyle{alpha}
\bibliography{antiRamsey}

\begin{thebibliography}{WZLX23}

\bibitem[ABHP15]{ABHP15}
Peter Allen, Julia B\"{o}ttcher, Jan Hladk\'{y}, and Diana Piguet.
\newblock A density {C}orr\'{a}di-{H}ajnal theorem.
\newblock {\em Canad. J. Math.}, 67(4):721--758, 2015.

\bibitem[AES74]{AES74}
B.~Andr\'{a}sfai, P.~Erd\H{o}s, and V.~T. S\'{o}s.
\newblock On the connection between chromatic number, maximal clique and
  minimal degree of a graph.
\newblock {\em Discrete Math.}, 8:205--218, 1974.

\bibitem[Alo83]{A83}
Noga Alon.
\newblock On a conjecture of {E}rd{\H o}s-{S}imonovits-{S}\'os concerning
  anti-{R}amsey theorems.
\newblock {\em J. Graph Theory}, 7(1):91--94, 1983.

\bibitem[CLSW25]{CLSW25}
Nannan Chen, Xizhi Liu, Lin Sun, and Guanghui Wang.
\newblock Tiling $h$ in dense graphs.
\newblock {\em arXiv preprint arXiv:2501.11450}, 2025.

\bibitem[CLT09]{CLT09}
He~Chen, Xueliang Li, and Jianhua Tu.
\newblock Complete solution for the rainbow numbers of matchings.
\newblock {\em Discrete Math.}, 309(10):3370--3380, 2009.

\bibitem[DHLY25]{DHLY25}
Jinghua Deng, Jianfeng Hou, Xizhi Liu, and Caihong Yang.
\newblock Tight bounds for rainbow partial {$F$}-tiling in edge-colored
  complete hypergraphs.
\newblock {\em J. Graph Theory}, 2025.

\bibitem[ESS75]{ESS75}
P.~Erd\H{o}s, M.~Simonovits, and V.~T. S\'{o}s.
\newblock Anti-{R}amsey theorems.
\newblock In {\em Infinite and finite sets ({C}olloq., {K}eszthely, 1973;
  dedicated to {P}. {E}rd\H{o}s on his 60th birthday), {V}ols. {I}, {II},
  {III}}, volume Vol. 10 of {\em Colloq. Math. Soc. J\'{a}nos Bolyai}, pages
  633--643. North-Holland, Amsterdam-London, 1975.

\bibitem[FKSS09]{FKSS09}
Shinya Fujita, Atsushi Kaneko, Ingo Schiermeyer, and Kazuhiro Suzuki.
\newblock A rainbow {$k$}-matching in the complete graph with {$r$} colors.
\newblock {\em Electron. J. Combin.}, 16(1):Research Paper 51, 13, 2009.

\bibitem[FMO10]{FMO10}
Shinya Fujita, Colton Magnant, and Kenta Ozeki.
\newblock Rainbow generalizations of {R}amsey theory: a survey.
\newblock {\em Graphs Combin.}, 26(1):1--30, 2010.

\bibitem[GH12]{GH12}
Codru{\c t} Grosu and Jan Hladk\'y.
\newblock The extremal function for partial bipartite tilings.
\newblock {\em European J. Combin.}, 33(5):807--815, 2012.

\bibitem[HHLZ25]{HHLZ25}
Jianfeng Hou, Caiyun Hu, Xizhi Liu, and Yixiao Zhang.
\newblock Density {H}ajnal--{S}zemer{\' e}di theorem for cliques of size four.
\newblock {\em arXiv preprint arXiv:2501.00801}, 2025.

\bibitem[JP09]{JP09}
T.~Jiang and O.~Pikhurko.
\newblock Anti-{R}amsey numbers of doubly edge-critical graphs.
\newblock {\em J. Graph Theory}, 61(3):210--218, 2009.

\bibitem[JW04]{JW04}
Tao Jiang and Douglas~B. West.
\newblock Edge-colorings of complete graphs that avoid polychromatic trees.
\newblock {\em Discrete Mathematics}, 274(1):137--145, 2004.

\bibitem[LLM25]{LLM25}
Hongliang Lu, Xinyue Luo, and Xinxin Ma.
\newblock New bounds on the anti-ramsey number of independent triangles.
\newblock {\em arXiv preprint arXiv:2506.07115}, 2025.

\bibitem[LNT25]{LNT25}
Xu~Liu, Bo~Ning, and Yuting Tian.
\newblock Two conjectures on vertex-disjoint rainbow triangles.
\newblock {\em arXiv preprint arXiv:2510.01880}, 2025.

\bibitem[MBNL02]{MN02}
J.~J. Montellano-Ballesteros and V.~Neumann-Lara.
\newblock An anti-{R}amsey theorem.
\newblock {\em Combinatorica}, 22(3):445--449, 2002.

\bibitem[MBNL05]{MN05}
J.~J. Montellano-Ballesteros and V.~Neumann-Lara.
\newblock An anti-{R}amsey theorem on cycles.
\newblock {\em Graphs Combin.}, 21(3):343--354, 2005.

\bibitem[OY13]{OY13}
Lale \"Ozkahya and Michael Young.
\newblock Anti-{R}amsey number of matchings in hypergraphs.
\newblock {\em Discrete Math.}, 313(20):2359--2364, 2013.

\bibitem[Sch04]{Sch04}
I.~Schiermeyer.
\newblock Rainbow numbers for matchings and complete graphs.
\newblock {\em Discrete Math.}, 286(1-2):157--162, 2004.

\bibitem[SS84]{SS84}
Mikl\'os Simonovits and Vera~T. S\'os.
\newblock On restricted colourings of {$K_n$}.
\newblock {\em Combinatorica}, 4(1):101--110, 1984.

\bibitem[Tur41]{TU41}
P.~Tur{\'a}n.
\newblock On an extremal problem in graph theory.
\newblock {\em Mat. Fiz. Lapok}, 48:436--452, 1941.

\bibitem[WZLX23]{WZLX23}
Fangfang Wu, Shenggui Zhang, Binlong Li, and Jimeng Xiao.
\newblock Anti-{R}amsey numbers for vertex-disjoint triangles.
\newblock {\em Discrete Math.}, 346(1):Paper No. 113123, 13, 2023.

\bibitem[YZ19]{YZ19}
Long-Tu Yuan and Xiao-Dong Zhang.
\newblock {A}nti-{R}amsey numbers of graphs with some decomposition family
  sequences.
\newblock {\em arXiv preprint arXiv:1903.10319}, 2019.

\end{thebibliography}
\begin{appendix}
\section{Results concerning $tK_{3}$-free graphs}
%
Given an $n$-vertex $(t+1)K_3$-free graph $H$, let $(\mathcal{T},\mathcal{M},\mathcal{I})$ be the maximal tiling triple of $H$ with   an ideal partition $(\mathcal{T}_1,\mathcal{T}_2,\mathcal{T}_3,\mathcal{T}_4)$ of $\mathcal{T}$. The following give upper bounds on the number of edges in  each part and between them.
\begin{lemma}[{\cite[Lemma~4.2]{ABHP15}}]\label{LEM:upper-bound-each-part-1}
The following bounds hold.
\begin{enumerate}[label=\alph*)]
    \item $e(\mathcal I)=0.$\label{itm:i}
    \item $e(\mathcal I,\mathcal M)\le im.$\label{itm:im}
    \item $e(\mathcal M)\le m^2.$\label{itm:m}
    \item $e(\mathcal M,\mathcal T_1)\le 4mt_1$.\label{itm:mt1}
    \item $e(\mathcal I,\mathcal T_1)\le 2it_1$.\label{itm:it1}
    \item $e(\mathcal I,\mathcal T_1)\le 7\binom{t_1}{2}$.\label{itm:t1}
    \item $e(\mathcal I,\mathcal T_2)\le2it_2$.\label{itm:it2}
    \item $e(\mathcal T_2)\le8\binom{t_2}{2}+3t_2$.\label{itm:t2}
    \item $e\left(\mathcal{T}_{3}\right)+e\left(\mathcal{T}_{3}, \mathcal{T}_{4}\right) \leq 8\binom{t_3}{2}+8 t_{3} t_{4}+3 t_{3}.$\label{itm:t3+t3t4}
\end{enumerate}
\end{lemma}

\medskip

\begin{lemma}[{\cite[Lemma~4.3]{ABHP15}}]\label{LEM:upper-bound-each-part-2}
The following bounds hold.
\begin{enumerate}[label=\alph*)]
     \item If $t_{1} \neq1$ and $j \ge2$, then $e(\mathcal T_{1}, \mathcal T_{j}) \le7 t_{1} t_{j}.$\label{itm:t1>1-t1tj}
     \item If $t_{2} \neq1$ and $j\ge3$, then $e(\mathcal T_{2}, \mathcal T_{j}) \le8 t_{2} t_{j}.$\label{itm:t2>1-t2tj}
\end{enumerate}
\end{lemma}

\medskip

\begin{lemma}[{[Lemma~4.4]\cite{ABHP15}}]\label{LEM:upper-bound-each-part-3}
    The following bounds hold.
    \begin{enumerate}[label=\alph*)]
        \item \label{itm:t1t2}
     \begin{align*}
    e\left(\mathcal{T}_{1}, \mathcal{T}_{2}\right) + e\left(\mathcal{M}, \mathcal{T}_{2}\right) \leq 
    \begin{cases} 
    7t_{1}t_{2} + (2 + 3m)t_{2} & \text{if } m \geq 1, \\ 
    0 & \text{if } m = 0.
    \end{cases} 
    \end{align*}
    \item\label{itm:t1tj} If $j=3,4$, then 
     \begin{align*}
    e\left(\mathcal{T}_{1}, \mathcal{T}_{j}\right)+e\left(\mathcal{M}, \mathcal{T}_{j}\right) \leq 
    \begin{cases} 
    7 t_{1} t_{j}+(3+3 m) t_{j} & \text{if } m \geq 1, \\ 
    0 & \text{if } m=0.
    \end{cases} 
    \end{align*}
    \item\label{itm:t2tj} If $j=3,4$, then 
     \begin{align*}
    e\left(\mathcal{T}_{2}, \mathcal{T}_{j}\right)+e\left(\mathcal{I}, \mathcal{T}_{j}\right) \leq
    \begin{cases} 
    8 t_{2} t_{j}+(2+i) t_{j} & \text{if } i \geq 1, \\ 
    0 & \text{if } i=0.
    \end{cases} 
    \end{align*} 
    For $ t_2 = 1 $ and any $ T\in \mathcal{T}_3\cup\mathcal T_4 $, we have $ e(\mathcal{I}, T) \leq i $ if $ e(\mathcal{T}_2, T) = 9 $, and $ e(\mathcal{I}, T) \leq i + 2 $ if $ e(\mathcal{T}_1, T) \leq 8 $.  
    \end{enumerate}
\end{lemma}
 
\end{appendix}

\newpage

\section*{\normalfont Authors}

\begin{multicols}{2}
\begin{flushleft}

\vbox{%
Jinghua Deng \\
{\small Center for Discrete Mathematics} \\
{\small Fuzhou University} \\
{\small Fujian, 350003, China} \\
\texttt{Jinghua\_deng@163.com}
}

\vspace{0.7cm}

\vbox{%
Caiyun Hu \\
{\small Center for Discrete Mathematics} \\
{\small Fuzhou University} \\
{\small Fujian, 350003, China} \\
\texttt{hucaiyun.fzu@gmail.com}
}

\vspace{0.7cm}

\vbox{%
Jianfeng Hou \\
{\small Center for Discrete Mathematics} \\
{\small Fuzhou University} \\
{\small Fujian, 350003, China} \\
\texttt{jfhou@fzu.edu.cn}
}

\vspace{0.7cm}

\vbox{%
Xizhi Liu \\
{\small School of Mathematical Sciences} \\
{\small University of Science and Technology of China} \\
{\small Hefei, 230026, China} \\
\texttt{liuxizhi@ustc.edu.cn}
}

\end{flushleft}
\end{multicols}
\end{document}